\documentclass[a4paper,11pt]{article}

\usepackage[T1]{fontenc}
\usepackage{lmodern}
\usepackage[french,english]{babel}
\usepackage{csquotes}

\usepackage{amsmath}
\usepackage{amssymb}
\usepackage{amsthm}
\usepackage{tikz}
\usepackage{tikz-cd} 
\usetikzlibrary{positioning,intersections} 
\usetikzlibrary{decorations.pathmorphing} 
\usetikzlibrary{arrows}

\usepackage[hidelinks]{hyperref} 
\usepackage[many]{tcolorbox} 
\usepackage{stmaryrd} 
\usepackage{indentfirst}
\usepackage{soul}
\usepackage{chngpage}
\usepackage{todonotes}
\setuptodonotes{color=blue!10!white}
\usepackage{enumerate}		

\DeclareFontFamily{U}{min}{}
\DeclareFontShape{U}{min}{m}{n}{<-> dmjhira}{}

\usepackage{xspace} 

\renewcommand{\epsilon}{\varepsilon}
\renewcommand{\phi}{\varphi}

\newcommand{\Hom}{\operatorname{Hom}}

\renewcommand{\lim}{\operatorname{lim}}
\newcommand{\colim}{\operatorname{colim}}
\newcommand{\dom}{{\operatorname{dom}}}

\newcommand{\comma}[2]			
{\mbox{$(#1\!\downarrow\!#2)$}}

\newcommand{\Sh}{\mathbf{Sh}}

\newcommand{\imp}{\!\Rightarrow\!}

\makeatletter
\newbox\xrat@below
\newbox\xrat@above
\newcommand{\xrightarrowtail}[2][]{%
	\setbox\xrat@below=\hbox{\ensuremath{\scriptstyle #1}}%
	\setbox\xrat@above=\hbox{\ensuremath{\scriptstyle #2}}%
	\pgfmathsetlengthmacro{\xrat@len}{max(\wd\xrat@below,\wd\xrat@above)+.6em}%
	\mathrel{\tikz [>->,baseline=-.75ex]
		\draw (0,0) -- node[below=-2pt] {\box\xrat@below}
		node[above=-2pt] {\box\xrat@above}
		(\xrat@len,0) ;}}
\makeatother

\tikzset{Rightarrow/.style={double equal sign distance,>={Implies},->},
	triple/.style={-,preaction={draw,Rightarrow}}}

\theoremstyle{plain}
\newtheorem{thm}{Theorem}
\numberwithin{thm}{subsection}
\newtheorem{prop}[thm]{Proposition}

\newtheorem{cor}[thm]{Corollary}

\theoremstyle{definition}
\newtheorem{defn}[thm]{Definition}

\newtheorem{remark}[thm]{Remark}
\newtheorem{remarks}[thm]{Remarks}

\renewcommand{\cal}[1]{\mathcal{#1}}

\begin{document}

\title{Local fibrations and morphisms of relative toposes}
\author{Léo Bartoli and Olivia Caramello}

\maketitle

\begin{abstract}
We introduce the notion of local fibration, a generalization of the notion of fibration which takes into account the presence of Grothendieck topologies on the two categories, and show that the classical results about fibrations lift to this more general setting. As an application of this notion, we obtain a characterization of the functors between relative sites that induce a morphism between the corresponding relative toposes: these are exactly the morphisms of sites which are morphisms of local fibrations. Also, we prove a weak version of Diaconescu's theorem, providing an equivalence between the continuous morphisms of local fibrations towards the canonical stack of a relative topos and the weak morphisms between the associated indexed toposes. The paper also contains a number of results of independent interest on morphisms of toposes and their associated stacks, including a fibrational characterization of locally connected (resp. totally connected) morphisms.    
\end{abstract}

\tableofcontents

\section{Introduction}

The theory of pseudofunctors provides a formalism to do category theory \textquoteleft
over a base category\textquoteright. In this setting, Grothendieck introduced the notion of a \emph{fibration}, which provides a more geometric and flexible way for approaching indexed categories. Fibrations have proven to be a powerful tool across various areas of mathematics, notably in geometry (through the theory of stacks) or in logic (via hyperdoctrines).

In particular, Giraud emphasized the usefulness of fibrations as a geometric framework for presenting relative toposes \cite{giraud.classifying}. His approach involves a topology on the base, a cartesian stack, and the so-called \emph{Giraud topology} on the fibration. This topology is the minimal one making the fibration into a \emph{comorphism of sites}, that is, the minimal topology ensuring that the fibration induces a relative topos; it serves as a relative analogue of the trivial topology. In this well-structured but restricted setting—which can be understood as the relative analogue of cartesian categories endowed with trivial topologies—Giraud was able to formulate a relative analogue of Diaconescu's theorem.

While this framework provides a foundational and conceptual setting, it can be extended. It has since been shown (see \cite{CaramelloZanfa}) that fibrations equipped with topologies containing this Giraud topology provide the right notion of \emph{relative site}: every relative topos admits a presentation by such a structured fibration over a site.

In our previous work \cite{bartolicaramello}, we generalized Giraud’s relative Diaconescu theorem to this more flexible framework of general fibrations endowed with topologies containing Giraud's one. These \emph{relative sites} allow for more intricate topological data than the trivial relative topology considered by Giraud.

However, it is worth noting that the definition of a fibration itself does not depend on any topology. This highlights a certain incompleteness: while fibrations can be equipped with topologies, their definition ignores them entirely. This observation motivates the need for a more refined notion—one that incorporates topological information more intrinsically. 

This paper introduces the notions of \emph{locally cartesian arrows} and \emph{local fibrations}, which refine the classical notions of cartesian arrows and fibrations by addressing the lack of topological data taken into account in their definitions. These structures sit naturally between fibrations—whose definition ignores topology—and general comorphisms of sites. They provide a more flexible yet structured framework for presenting relative toposes.

Within this framework, we establish a general criterion for when a morphism of sites between local fibrations induces a morphism of relative toposes: namely, this occurs if and only if the morphism of sites is a morphism of local fibrations. This result extends earlier theorems restricted to ordinary fibrations, and yields a relative version of Diaconescu’s theorem in a broader setting.

We then turn to the case where the comorphism of sites is also a continuous functor, which admits an extension at the topos level. We investigate the interaction between the extension being a fibration and the original comorphism being a local fibration. This leads to a new characterization of locally connected geometric morphisms.

Finally, we study weak morphisms of relative toposes, distinguishing between two different notions: weak relative geometric morphisms and weak indexed geometric morphisms. We prove a weak, indexed version of Diaconescu’s theorem: continuous morphisms of local fibrations correspond precisely to weak indexed geometric morphisms.

In more detail, here is the content of each section:

The preliminary section \ref{section:notprel} recalls the setting of relative topos theory via stacks, or more generally, via fibrations. In particular, we recall the definition of \emph{relative sites}: these consist of fibrations over a site, equipped with the minimal Giraud topology with respect to the topology on the base category. We then recall two central constructions: the \emph{canonical relative site}, which generalizes the absolute canonical site to the relative context, and the canonical functor $\eta$, which plays the role of the Yoneda embedding followed by sheafification in the absolute case. These notions lead naturally to the concept of \emph{$\eta$-extension}, an analogue of extension along the canonical functor in the relative setting. We recall the key characterization: a morphism of sites over a common base site induces a morphism between the associated relative toposes if and only if its $\eta$-extension is a morphism of fibrations. Finally, we briefly discuss another perspective on the canonical relative site, namely as a topos in its own right; from this viewpoint, the $\eta$-extension corresponds to the inverse image of a geometric morphism.

Section \ref{section3} is devoted to the introduction of the general notions of \emph{locally cartesian arrows} and \emph{local fibrations}. An arbitrary comorphism of sites is not necessarily a fibration, but it can always be compared to one—namely, the canonical stack via its canonical functor $\eta$. This naturally leads to defining an arrow as \emph{locally cartesian} when its image under $\eta$ in the canonical stack is cartesian. Similarly, whereas a fibration is characterized by a lifting property up to isomorphism, we define a local fibration as a comorphism satisfying a lifting property up to covering. We also show that local fibrations admit a factorization structure inherited from the vertical–horizontal factorization of classical fibrations, and we define the natural notion of morphisms of local fibrations.

Section \ref{section4} shows that this broader framework of local fibrations provides the natural setting for a conceptual proof of the fact that a morphism of sites between relative sites which also is a morphism of fibrations induces a morphism of relative toposes. Indeed, from the theory of $\eta$-extensions, we know that a morphism of sites over a base induces a geometric morphism between the associated relative toposes if and only if its $\eta$-extension is a morphism of fibrations. However, a conceptual proof that the fibrational nature of a morphism of sites is transmitted to its $\eta$-extension was missing. To address this, we show that cartesian arrows in the canonical stack can always be localized in two successive ways (via pullbacks and then via colimit presentations), and that this localization process reduces the claim to the preservation of locally cartesian arrows by the original morphism of sites. Notably, in this framework, the necessity of the condition of being a morphism of local fibrations becomes immediate from the very definition of locally cartesian arrows. We then deduce from this characterization a relative Diaconescu theorem for local fibrations. We also examine how these results specialize to ordinary fibrations: any functor between two fibrations comes equipped with canonical natural transformations that express how far it is to be a morphism of fibrations. We show that such a functor is a morphism of local fibrations if and only if these transformations become isomorphisms at the level of the associated toposes—meaning that any shortcoming in the preservation of cartesian arrows becomes invisible at the topos level. Furthermore, we explain how the relative Diaconescu theorem for local fibrations immediately specializes to the one previously established in the case of fibrations in \cite{bartolicaramello}.

Section \ref{section5} investigates the situation where the comorphism presenting a relative topos is also continuous. Indeed, a continuous comorphism of sites induces an essential geometric morphism, whose essential image can be seen as the extension of the base functor at the topos level. It is then natural to compare the condition under which this extension is a fibration with the condition under which the original continuous comorphism is a local fibration. Since such an essential image is a left adjoint, we study the general case of a left adjoint, and provide a reformulation of the notions of cartesian arrow and fibration using the unit of the adjunction. This reformulation leads to a conceptual interpretation: the essential image is a fibration if and only if the essential geometric morphism is locally connected. Observing that the canonical stack of a relative topos defines a totally connected morphism, and in view of our new characterization of locally connected morphisms, we present it as the free totally connected relative topos associated with a given relative topos. Finally, we provide a full characterization of continuous comorphisms inducing fibrations at the topos level, in terms of a cofinality condition.

In the final section \ref{section6}, we turn to the study of weak morphisms of relative toposes. We begin by distinguishing two notions of weak geometric morphism relevant to the relative setting. While relative geometric morphisms are equivalent to indexed geometric morphisms (i.e., indexed adjunctions between the canonical stacks of which inverse images preserve finite limits fiberwise), this correspondence breaks down in the weak setting. We first discuss the notion of weak relative geometric morphism, which admits a characterization in terms of morphisms of opfibrations. We then study weak indexed geometric morphisms, which constitute the appropriate generalization of the weak geometric morphisms we usually consider in the absolute case. In this context, we prove an indexed weak Diaconescu theorem, yielding an equivalence between continuous morphisms of local fibrations and weak indexed geometric morphisms. This theorem specializes to the case of fibrations, and allows us to conclude that the Giraud topos of a fibration is canonically equivalent to the Giraud topos of its stackification.

\section{Preliminaries}\label{section:notprel}
This work is set in the context of relative topos theory, approached via stacks or, more generally, via fibrations, according to \cite{CaramelloZanfa}. In this section, we recall the essential notions of relative topos theory on which we will rely: the notion of fibration and comorphism of sites leading to the concept of relative sites, the construction of the canonical relative site of a relative topos and the associated canonical functor, the interpretation of the canonical relative site as a topos, and finally the notion of $\eta$-extension, which provides a sound understanding of morphisms of sites in the relative setting. We also fix the notations that will be used throughout the article.

\subsection{Fibrations and relative sites}
A first notation that will appear repeatedly concerns the topos of sheaves on a site. To emphasize the idea of the topos of sheaves on \((\mathcal{C}, J)\) as a kind of cocompletion of the category \(\mathcal{C}\) subject to relations imposed by \(J\), we adopt the light and suggestive notation \(\widehat{\mathcal{C}}_J\) proposed by Laurent Lafforgue.

The functor denoted by \( l_J : \mathcal{C} \to \widehat{\mathcal{C}}_J \) is the \emph{canonical functor} related to the site $(\cal C,J)$, defined as the composite of the Yoneda embedding \( y_{\cal C} : \mathcal{C} \to \widehat{\mathcal{C}} \) with the sheafification functor \( a_J : \widehat{\mathcal{C}} \to \widehat{\mathcal{C}}_J \); for the sake of lighter notation, we will occasionally omit the index \( J \).

Also, we will often denote a natural transformation $y_{\cal C}(c) \to P$ by $ev_x$, where $x$ is the element of $P(c)$ to which the identity is sent.

There will be many different adjoint functors throughout the text, and we will use the following notation: for a pair of adjoint functors $F \vdash G$ and an arrow $u : G(d) \to c$, we will denote its transpose by $u^t : d \to F(c)$. Conversely, for an arrow $v : d \to F(c)$, we will also write $v^t : G(d) \to c$ for its transpose.

In this paper, we will be interested in \emph{relative toposes}: 

\begin{defn}
A relative topos over some base topos $\cal E$ is a geometric morphism $f: \cal F \to \cal E$. A \emph{relative geometric morphism} between two relative toposes $g : [f] \to [f']$ is a geometric morphism $g : \cal F \to \cal F'$ such that the following triangle of geometric morphisms commutes (up to iso):

\[\begin{tikzcd}
	{\cal F} && {\cal F'} \\
	& {\cal E}
	\arrow["g", from=1-1, to=1-3]
	\arrow["f"', from=1-1, to=2-2]
	\arrow["{f'}", from=1-3, to=2-2]
\end{tikzcd}\]
\end{defn}

The natural way to induce such relative toposes is by the use of \emph{fibrations}. We recall that a (Street) fibration is a particular kind of functor $p : {\cal D} \to {\cal C}$ satisfying a certain lifting property:

\begin{defn}
A fibration is a functor $p : {\cal D} \to {\cal C}$ such that for each arrow $f : c \to p(d)$ in $\cal C$ there exists a \emph{cartesian} arrow $\widehat{f} : d' \to d$ in $\cal D$ with $\sigma : p(d') \simeq c$ such that $p(\widehat{f}) \simeq f\sigma$.
\end{defn} 

With the definition of a cartesian arrow being:

\begin{defn}
For a functor $p : {\cal D} \to {\cal C}$, a cartesian arrow is an arrow $f : d' \to d$ in $\cal D$ such that: for each arrow $h : p(d'') \to p(d')$ in $\cal C$ such that $p(f)\circ h = p(g)$ for $g : d'' \to d$, there exists a unique arrow $h' : d'' \to d'$ such that $h=p(h')$ and $fh'=g$.
\end{defn}

As is well known (see \cite{CaramelloZanfa}, Corollary 2.2.6), fibrations over some base category $\cal C$ are equivalent to $\cal C$-indexed categories. Accordingly, we shall often present a fibration $p : \cal D \to \cal C$ by means of such an indexed category $p : \mathcal{G}(\mathbb D) \to \cal C$ in order to have a more manageable description of it.

This definition of the notion of fibration naturally comes with a notion of morphisms:

\begin{defn}
Let  $p : \mathcal{G}(\mathbb D) \to \cal C$ and  $p' : \mathcal{G}(\mathbb D') \to \cal C$ be two fibrations over $\cal C$. A morphism of fibrations between them is a functor $A : \mathcal{G}(\mathbb D) \to  \mathcal{G}(\mathbb D')$ such that it sends cartesian arrows to cartesian arrows, and it commutes with the projections: $p'A \simeq p$.
\end{defn}

We will see that fibrations are related to the notion of \emph{comorphism of sites}, as defined:

\begin{defn}[Subsection 3.3 \cite{denseness}]
A functor $F : (\cal C',J') \to (\cal C,J)$ between two sites is called a comorphism of sites when: for every $J$-covering sieve $S$ on an object of the form $F(c')$, there exists a covering $S'$ for $J'$ on $c'$ such that $F(S') \subset S$.
\end{defn}

Those functors are aimed to induce geometric morphisms covariantly:

\begin{prop}[Subsection 3.3 \cite{denseness}]
Let  $F : (\cal C',J') \to (\cal C,J)$ be a comorphism of sites. It induces a geometric morphism denoted as $C_F : \widehat{\cal C'}_{J'} \to \widehat{\cal C}_J$ having for inverse image $C_F^* := a_{J'}(-\circ F^{op}) : \widehat{\cal C}_J \to \widehat{\cal C'}_{J'}$.
\end{prop}

This notion of comorphism of sites will allow us to see fibrations as a very natural way to induce relative toposes: as we present any absolute topos (toposes over $\mathbf{Set}$) as coming from a category endowed with a topology (containing the trivial one, which is a minimal topology), we will see in \ref{canonicalresitecanonicalfunct} that we can present every relative topos by the means of a fibration endowed with some topology (containing a minimal one). This topology is called the \emph{Giraud topology}, defined in the following definition-proposition, and constitutes the relative analogue of the trivial topology in the absolute case in the sense that it is the \emph{minimal} topology making the projection into a comoprhism of sites: 

\begin{defn}\label{comorphinduitgeom}[Theorem 3.13 \cite{denseness}]
For a fibration $p : \mathcal{G}(\mathbb D) \to \cal C$ over a base site $(\cal C,J)$, the \emph{Giraud topology} (denoted as $J_{\mathbb D}$) on $\mathcal{G}(\mathbb D)$ for the topology $J$ is defined by declaring as covering sieves those containing a family of cartesian arrows $((f_i,1): (\mathbb D(f_i)(x),c_i) \to (x,c))_i$ whose projections $(f_i)_i$ form a $J$-covering family. 
\end{defn}

The Giraud topology $J_{\mathbb D}$ on $\mathcal{G}(\mathbb D)$ is is the smallest topology making $p$ a comorphism of sites towards $(\cal C, J)$.
 
Recall the following definition:

\begin{defn}[Definition 2.1 \cite{bartolicaramello}]
A \emph{relative site} is a fibration $p : (\mathcal{G}(\mathbb D),K) \to (\cal C,J)$ over a site such that $K$ contains the associated Giraud topology.
\end{defn}

The name of \emph{relative site} is justified by the fact that it induces a \emph{relative topos}, just as (absolute) sites induce (absolute) toposes. Indeed, since the projection functor $p$ of a relative site $p : (\mathcal{G}(\mathbb D),K) \to (\cal C,J)$ is a comorphism of sites, it induces a relative topos $C_p : \widehat{\cal G(\mathbb D)}_K \to \widehat{\cal C}_J$.

\subsection{Canonical relative site, canonical functor}\label{canonicalresitecanonicalfunct}
We have just seen that every relative site induces a relative topos; conversely, every relative topos arises from a relative site. This fact is provided by the so-called \emph{canonical relative site} of a relative topos, which we now describe in this subsection.

This canonical relative site being presented as a comma category, we introduce here our notation: for two functors $F: \cal D \to \cal C$ and $G : \cal D' \to \cal C$, we denote as $(F/G)$ the comma category having for objects the triplets $(d,d',u:F(d)\to G(d'))$. Also, if one of the two functors, for example $F$, is the identity, we use the more concise notation $(\cal D/G)$.

As already presented Subsection 8.2.2 \cite{CaramelloZanfa}, we summarize some important properties of the canonical relative site of a relative topos in the following definition/proposition:

\begin{defn}
Let $f : \cal F \to \cal E$ be a relative topos.

\begin{enumerate}
    \item We define the topology $J_f$ on the comma category $(\cal F/f^*)$ by saying that a sieve $S$ is covering for $J_f$ if and only if its projection $\pi_{\cal F}(S)$ in $(\cal F,J_{\cal F}^{can})$ is covering; that is, if and only if it is covering on the first component.
    \item The projection $\pi_f:(\cal F/f^* ) \to \cal E$ is a fibration, and even a stack for the canonical topology on $\cal E$.
    \item Its associated indexed category will be denoted $S_f : \cal E^{op} \to \mathbf{CAT}$ and acts by sending an object $E$ of $\cal E$ to the slice topos $(\cal F/f^*(E))$, and an arrow $u : E \to E'$ to the pullback functor $S_f(u) : (\cal F/f^*(E')) \to (\cal F/f^*(E))$.
    \item When the geometric morphism considered is the identity of $\cal E$, we will simply denote this fibration as $S_{\cal E}$.
    \item  Together with the topology $J_f$, this construction provides a relative site $\pi_f:((\cal F/f^* ),J_f) \to (\cal E,J_{\cal E}^{can})$.
\end{enumerate} 
\end{defn}

This canonical relative sites allows us to recover our initial relative topos:

\begin{prop}[Theorem 8.2.5. \cite{CaramelloZanfa}]
Let $f : \cal F \to \cal E$ be a relative topos. The functor $\pi_{\cal F} : ((\cal F/f^* ),J_f) \to  (\cal F,J_{\cal F}^{can})$ is a dense morphism and a dense comorphism of sites, and we have the following equivalence of relative toposes over $\cal E$:

\[\begin{tikzcd}
	{\cal F} && {\widehat{\cal F/f^*}_{J_f}} \\
	& {\cal E}
	\arrow["{\Sh(\pi_{\cal F})}", shift left=4, from=1-1, to=1-3]
	\arrow["\simeq"{description}, shift left=2, draw=none, from=1-1, to=1-3]
	\arrow["f"', from=1-1, to=2-2]
	\arrow["{C_{\pi_{\cal F}}}", shift left, from=1-3, to=1-1]
	\arrow["{C_{\pi_f}}", from=1-3, to=2-2]
\end{tikzcd}\]
\end{prop}

Note that, when the base topos $\cal E$ is presented as $\widehat{\cal C}_J$, all the previous definitions and properties can be stated not only over $\cal E$, but also over $(\cal C,J)$, see \cite{bartolicaramello} subsections 2.2 and 2.3. But it can be more convenient to work over $\cal E$ and not an arbitrary base site, since this base category has finite limits; in this case, the canonical stack is a \emph{cartesian} fibration:

\begin{defn}
Let $\cal C$ be a cartesian category. We call $p : \cal G(\mathbb{D}) \to \cal C$ a cartesian fibration if $\cal G(\mathbb D)$ has finite limits and $p$ preserves them. Alternatively, from an indexed point of view, this is equivalent to $\mathbb D(c)$ having finite limits for every object $c$ of $\cal C$, and the functors $\mathbb D(f)$ preserving them for every $f : c' \to c$ in $\cal C$.
\end{defn}

The following proposition recalls some technical results about fibrations, and notably about the canonical stack, that will be used later on:

\begin{prop}\label{enumloccart}
Let $f : \cal F \to E$ be a relative topos.

\begin{enumerate}[(i)]
    \item A cartesian arrows in $({\cal F}/f^*)$ between $(F,E,\alpha: F \to f^*E)$ and  $(F',E',\alpha': F' \to f^*E')$ is a pair of arrows $(g:F \to F', g':E\to E')$ such that the commutative square given by $\alpha' \circ g = f^*(g') \circ \alpha$ is a pullback one:

\[\begin{tikzcd}
	F & {F'} \\
	{f^*E} & {f^*E'}
	\arrow["g", from=1-1, to=1-2]
	\arrow["\alpha"', from=1-1, to=2-1]
	\arrow["\lrcorner"{anchor=center, pos=0.125}, draw=none, from=1-1, to=2-2]
	\arrow["{\alpha'}", from=1-2, to=2-2]
	\arrow["{f^*(g')}"', from=2-1, to=2-2]
\end{tikzcd}\]

    \item In a fibration over a cartesian base (so, in particular in $({\cal F}/f^*)$), the pullback of a cartesian arrow along any arrow exists and is still cartesian.

    \item In any fibration, if we have two composable arrows $f$ and $g$, then $f$ and $f \circ g$ being cartesian implies that $g$ also is, and $f$ and $g$ being cartesian implies that $f \circ g$ also is.
\end{enumerate}   
\end{prop}

\begin{proof}
The point (i) is just a reminder from Part 2.2 of \cite{bartolicaramello}

For point (ii), let $(f,1)$ be a cartesian arrow and $(g,v)$ any arrow in a fibration $p: \mathcal{G}(\mathbb D) \to {\cal C}$ where $\cal C$ is a cartesian category. If we denote by $g'$ and $f'$ the pullbacks in the basis of $g$ along $f$ and of $f$ along $g$, then one can easily check that the following square in $\mathcal{G}(\mathbb D)$ is a pullback:

\[\begin{tikzcd}
	{(\mathbb D(f')(x'),\overline{c})} & {(\mathbb D(f)(x),c')} \\
	{(x',c'')} & {(x,c)}
	\arrow["{(g',\mathbb D(f')(v))}", shift left, from=1-1, to=1-2]
	\arrow["{(f',1)}"', from=1-1, to=2-1]
	\arrow["\lrcorner"{anchor=center, pos=0.125}, draw=none, from=1-1, to=2-2]
	\arrow["{(f,1)}", from=1-2, to=2-2]
	\arrow["{(g,v)}"', from=2-1, to=2-2]
\end{tikzcd}\]

The last point follows directly from the indexed formulation: if we have two arrows $(f,u)$ and $(g,v)$ in a fibration $\mathcal{G}(\mathbb{D})$, the equality $(f,u) \circ (g,v) = (f \circ g, \mathbb{D}(g)(u) \circ v)$ holds. Recall also that an arrow is cartesian if and only if its vertical part is an isomorphism, so that $\mathbb{D}(g)(u) \circ v$ being an isomorphism implies that $v$ is one as well; and, if $u$ and $v$ are isomorphisms, then $\mathbb{D}(g)(u) \circ v$ is also an isomorphism.
\end{proof}

Now, the following is the definition of the relative analogue to the canonical functor denoted as $l$ in the absolute setting: it is a dense morphism and comorphism of sites having for domain the site and for codomain the canonical relative site. It exists in a full generality: for any comorphism of sites. A proof can be found in Proposition 2.5 \cite{fibered}.

\begin{prop}
Let $p : (\cal D,K) \to (\cal C,J)$ be a comorphism of sites. There is a dense comorphism and dense morphism of sites denoted as $\eta_{\cal D} : (\cal D,K) \to ((\widehat{\cal D}_K/C_p^*),J_{C_p})$. This functor is defined (with the help of the Yoneda lemma) by sending an object $d$ of $\cal D$ on $(l_K(d),l_J(p(d)),a_K(ev_{1_{p(d)}}) : l_K(d) \to C_p^*l_J(p(d)))$, and an arrow $f : d' \to d$ on $(l_K(f),C_p^*l_J(p(f)))$.
\end{prop}

As in the absolute setting, where the index $J$ of the functor $l_J$ is often omitted, we will likewise write $\eta_{\cal D}$ simply as $\eta$ for the sake of simplicity.

When $p : (\cal G(\mathbb D),K) \to (\cal C,J)$ is a relative site, we sometimes denote $\eta_p$ as $\eta_{\mathbb D}$; in this case, it has the additional property that it preserves the fibration structure:

\begin{prop}\label{etamorphfib}[Proposition 3.14. \cite{bartolicaramello}]
Let $p : (\cal G(\mathbb D),K) \to (\cal C,J)$ be a relative site. The functor $\eta_{\mathbb D} : \cal G(\mathbb D) \to (\widehat{\cal G(\mathbb D)}_K/C_p^*)$ is a cover-preserving morphism of fibrations.
\end{prop}

\subsection{The notion of $\eta$-extension}\label{etaextsubsection}
Now that we have introduced the canonical relative site and $\eta$, the relative analogue of the canonical functor $l$, we are ready to present the notion of $\eta$-extension. 

In the absolute case, we study morphisms of sites by extending them along the $l$ functors. Indeed, a functor $A : ({\cal D},K) \to (\cal D',K')$ is a morphism of sites if and only if it admits an extension along the $l$ functors preserving colimits and finite limits, as dashed in the following diagram:

\[\begin{tikzcd}
	{\widehat{\cal D}_K} & {\widehat{\cal D'}_{K'}} \\
	{(\cal D,K)} & {(\cal D',K')}
	\arrow[dashed, from=1-1, to=1-2]
	\arrow["{l_K}", from=2-1, to=1-1]
	\arrow["\simeq"{description}, draw=none, from=2-1, to=1-2]
	\arrow["A"', from=2-1, to=2-2]
	\arrow["{l_{K'}}"', from=2-2, to=1-2]
\end{tikzcd}\]

A relative analogue of this notion of extension is already defined and systematically studied in Section 3 of \cite{bartolicaramello}. We recall its definition in the following proposition, along with its main properties (Proposition 3.4 of \cite{bartolicaramello}), which show that it provides the correct relative version of the extension along the functor $l$. Here, we work with a natural isomorphism $\phi$ as this is what will be relevant in our setting, although the definition given in \cite{bartolicaramello} applies to any natural transformation.

\begin{prop}\label{etaextensionproperties}[Section 3 \cite{bartolicaramello}]
Let $p : (\cal D,K) \to (\cal C,J)$ and $p : (\cal D',K') \to (\cal C,J)$ be two comorphisms of sites, together with a continuous functor $A : (\cal D,K) \to (\cal D',K')$ such that we have a natural isomorphism $\phi : p'A \simeq p $. 

\begin{enumerate}
    \item The natural isomorphism $\phi$ induces a mate $\overline{\phi} : lan_{A^{op}}(-\circ p^{op}) \Rightarrow (-\circ p'^{op})$ which in turns gives, by sheafification, a natural transformation $\widetilde{\phi} : \Sh(A)^*C_p^* \Rightarrow C_{p'}^*$.
    \item This allows to define the $\eta$-extension of $A$ as the functor $\widetilde{A}^* : (\widehat{\cal D}_K/C_p^*) \to (\widehat{\cal D'}_{K'}/C_{p'}^*)$ sending a triplet $(F,E, \alpha : F \to C_p^*E)$ to the triplet \\ $(\Sh(A)^*F,E,\widetilde{\phi}_E\Sh(A)^*(\alpha) : \Sh(A)^*F \to C_{p'}^*E)$ and an arrow $(g,h)$ to the arrow $(\Sh(A)^*(g),h)$.
    \item This extension is a continuous functor $\widetilde{A}^* : ((\widehat{\cal D}_K/C_p^*),J_{C_p}) \to ((\widehat{\cal D'}_{K'}/C_{p'}^*),J_{C_{p'}})$. Moreover, if $A$ is a morphism of sites, $\widetilde{A}^*$ also is.
    \item The following is a commutative square of colimit-preserving functors:

        \[\begin{tikzcd}
	{\widehat{(\widehat{\cal D}_K/C_p^*)}_{J_{C_p}}} & {\widehat{(\widehat{\cal D'}_{K'}/C_{p'}^*)}_{J_{C_{p'}}}} \\
	{\widehat{\cal D}_K} & {\widehat{\cal D'}_{K'}}
	\arrow["{\widetilde{A}^*}", from=1-1, to=1-2]
	\arrow["{\Sh(\pi_{\widehat{\cal D}_K})^*}"', shift right=2, draw=none, from=1-1, to=2-1]
	\arrow["\simeq"{description}, from=1-1, to=2-1]
	\arrow["{\Sh(\pi_{\widehat{\cal D'}_{K'}})^*}", shift left=3, draw=none, from=1-2, to=2-2]
	\arrow["\simeq"{description}, shift right=3, from=1-2, to=2-2]
	\arrow["{\Sh(A)^*}"', from=2-1, to=2-2]
    \end{tikzcd}\]
    giving us that $\widetilde{A}^*$ and $A$ induces, at the topos-level, the same inverse image.
    \item The following is a commutative square of functors, giving us that $\widetilde{A}^*$ is an extension of $A$ along the $\eta$ functors:

    \[\begin{tikzcd}
	{\widehat{(\widehat{\cal D}_K/C_p^*)}_{J_{C_p}}} & {\widehat{(\widehat{\cal D'}_{K'}/C_{p'}^*)}_{J_{C_{p'}}}} \\
	{(\cal D,K)} & {(\cal D',K')}
	\arrow["{\widetilde{A}^*}", from=1-1, to=1-2]
	\arrow["{\eta_{\cal D}}", from=2-1, to=1-1]
	\arrow["\simeq"{description}, draw=none, from=2-1, to=1-2]
	\arrow["A"', from=2-1, to=2-2]
	\arrow["{\eta_{\cal D'}}"', from=2-2, to=1-2]
    \end{tikzcd}\] 
    \item When $A$ is a morphism of sites, the following points are equivalent:
    \begin{itemize}
        \item We have that $\widetilde{\phi} : C_p\Sh(A) \simeq C_{p'}$, that is, $(A,\phi)$ induces a morphism of relative toposes.
        \item We have that $\widetilde{A}^*$ is a morphism of fibrations.
        \item We have that $\widetilde{A}^*$ preserves finite limits not only globally, but also fiberwise.
    \end{itemize}
\end{enumerate}
\end{prop}

\subsection{The canonical relative site as a topos}

The most important object of study in the context of relative toposes is the canonical stack of a relative topos. Up to now, we have considered it only as a site presenting a relative topos; however, the canonical stack has a dual nature: it is both a fibration over the base topos and a topos in its own right. Indeed, it is a special case of \emph{Artin gluing} along the inverse image functor of the relative topos. Since we present our relative toposes via comorphisms, we will show that, when a relative topos is presented via a comorphism of sites, this presentation naturally induces a site presentation for the canonical stack viewed as a topos. Moreover, this site presentation will provide an interpretation of the $\eta$-extension as an inverse image functor induced by a morphism of sites between these presentation sites of the canonical relative stacks as toposes.

For a relative topos $C_p : \widehat{\cal D}_K \to \widehat{\cal C}_J$ presented by some comorphism of sites $p$, the two following propositions will guide us in order to obtain a small set of generators for the comma category $(\widehat{\cal D}_K/C_p^*)$ coming from this comorphism of presentation:

\begin{prop}[Theorem 3.20. \cite{denseness}]\label{propdualcomma}
Let $p:({\cal D}, K)\to ({\cal C}, J)$ be a comorphism of sites. We have a functor
	\[
	\xi_{p}: (p/{\cal C}) \to (\widehat{\cal D}_K/C_{p}^{\ast})
	\]
	sending an object $(d, c, u:p(d)\to c)$ of the category $(p/{\cal C})$ to the object $(l_{K}d, l_Jc, l_Kd\to C_{p}^{\ast}(l_Jc))\cong a_{K}((- \circ p)(y_{\cal D}d))=a_{K}(\Hom_{\cal C}(p(-), c))$ which is the sheafification of the evaluation of the identity of the source representable onto the arrow $u$, well-defined by the Yoneda lemma. This functor is a dense bimorphism (comorphism and morphism) of sites
	$((p/{\cal C}), \overline{K}) \to ((\widehat{\cal D}_K/C_{p}^{\ast}), J_{C_{p}})$ where $\overline{K}$ has for coverings the ones having arrows of the first component being a cover in $\cal D$, and $J_{C_{p}}$ is the canonical relative topology of $C_p : \widehat{\cal D}_K \to \widehat{\cal C}_J$.
\end{prop}

We want to show that $(\widehat{\cal D}_K/C_{p}^{\ast})$ is a topos; to do so we will refine the previous density result to the (absolute) canonical topology on it. A first step is to identify the epimorphic families of this comma category:

\begin{prop}\label{epimorphismsincanonicalstack}
For a relative topos $f : \cal F \to \cal E$, the epimorphic families of $(\cal F/f^*)$ are the families $((g_i,h_i) : (F_i,E_i,\alpha_i:F_i \to f^*(E_i)) \to (F,E,\alpha:F \to f^*(E)))_i$ such that their projections on the two components are epimorphic families in $\cal F$ and $\cal E$, respectively.
\end{prop}

\begin{proof}
It is immediate to see that it is a sufficient condition, and it also is necessary since epimorphisms are colimits, which are computed componentwise in this comma category of colimit-preserving functors.
\end{proof}

To obtain a separating set for $(\widehat{\cal D}_K/C_{p}^{\ast})$, we will need to proceed with a minor augmentation of the comorphism of sites $p : (\cal D, K) \to (\cal C, J)$, consisting in the addition of an initial object to the sites, as in (a) of Remarks 5.37. in \cite{denseness}:

\begin{prop}\label{ajoutinitial}
Let $p : (\cal D, K) \to (\cal C, J)$ be a comorphism of sites. We have a Morita-equivalent comorphism of sites $p_0 : (\cal D_0, K_0) \to (\cal C_0, J_0)$: the categories $\cal C_0$ and $\cal D_0$ are just the categories $\cal C$ and $\cal D$ where we freely added an initial object (no arrow has this new initial object for codomain); the topologies $J_0$ and $K_0$ are the same as $J$ and $K$ plus the empty sieves covering the new initial objects; and $p_0$ acts as $p$ on the objects coming from $\cal D$ and sends the initial object of $\cal D_0$ to the one of $\cal C_0$. The Morita-equivalence is given by the commutation of the following diagram:

\[\begin{tikzcd}
	{\widehat{\cal D}_K} & {\widehat{\cal D_0}_{K_0}} \\
	{\widehat{\cal C}_J} & {\widehat{\cal C_0}_{J_0}}
	\arrow["{C_p}"', from=1-1, to=2-1]
	\arrow["{C_{j_0}}"', from=1-2, to=1-1]
	\arrow["\simeq"{description}, shift left=2, draw=none, from=1-2, to=1-1]
	\arrow["{C_{p_0}}", from=1-2, to=2-2]
	\arrow["{C_{i_0}}", from=2-2, to=2-1]
	\arrow["\simeq"{description}, shift right=2, draw=none, from=2-2, to=2-1]
\end{tikzcd}\]

\noindent where $i_0 : (\cal C,J) \to (\cal C_0,J_0)$, the obvious inclusion, is a dense bimorphism of sites, and same for $j_0$ with $\cal D$.
\end{prop}
\qed

Since these new sites and this new comorphism are Morita-equivalent to the previous ones, the codomain of the functor $\xi_{p_0} : (p_0/\cal C_0) \to (\widehat{\cal D}_K/C_{p}^{\ast})$ is indeed the canonical stack we want to study. This allows to explicit a separating set:

\begin{prop}\label{xilocal}
Let $p:({\cal D}, K)\to ({\cal C}, J)$ be a comorphism of sites. 

\begin{enumerate}
    \item The objects of the form $\xi_{p_0}(d,c,u)$ form a separating family in $(\widehat{\cal D}_K/C_{p}^{\ast})$, hence it is small-generated.
    \item The comma category $(\widehat{\cal D}_K/C_{p}^{\ast})$ is a topos.
\end{enumerate}
\end{prop}

\begin{proof}
First, it is immediate that $(\widehat{\cal D}_K/C_{p}^{\ast})$ satisfies all of Giraud's axioms a priori except for the existence of a small generating set: indeed, since arbitrary colimits and finite limits exist in $\widehat{\cal D}_K$ and $\widehat{\cal C}_J$, and since the identity on $\widehat{\cal D}_K$ and $C_p^*$ preserve them, the comma category $(\widehat{\cal D}_K /C_p^*)$ is cocomplete and finitely complete (and these finite limits and arbitrary colimits are preserved by the projections on $\widehat{\cal D}_K$ and $\widehat{\cal C}_J$). From this, it is also easy to check that, since this holds component-wise, coproducts are disjoint (monomorphisms being characterized by finite limits, they are also computed component-wise) and distribute over pullbacks in $(\widehat{\cal D}_K /C_p^*)$. Again, since equivalence relations and coequalizers are computed component-wise, and since they exist and are effective in $\widehat{\cal D}_K$ and $\widehat{\cal C}_J$, they also exist and are effective in $(\widehat{\cal D}_K /C_p^*)$. Thus, it remains only to show that $(\widehat{\cal D}_K /C_p^*)$ has a small generating set in order to conclude, by Giraud's axioms, that it is a topos. 

To this end, we begin by recalling \ref{epimorphismsincanonicalstack}: a family of arrows in $(\widehat{\cal D}_K/C_{p}^{\ast})$ is jointly epimorphic if and only if the two families of its components are jointly epimorphic in respectively $\widehat{\cal D}_K$ and $\widehat{\cal C}_J$.

Let $(F,E,\alpha : F \to C_p^*E)$ be an object of the comma. We can cover $E$ with objects coming from $\cal C_0$: $(u_i : l_J(c_i) \to E)_i$ and pullback the $C_p^*(u_i)$ along $\alpha$ to obtain moreover a covering of $E$:
\[\begin{tikzcd}
	{F_i} & {C_p^*l_J(c_i)} \\
	F & {C_p^*E}
	\arrow["{\alpha_i}", from=1-1, to=1-2]
	\arrow["{v_i}"', from=1-1, to=2-1]
	\arrow["\lrcorner"{anchor=center, pos=0.125}, draw=none, from=1-1, to=2-2]
	\arrow["{C_p^*(u_i)}", from=1-2, to=2-2]
	\arrow["\alpha"', from=2-1, to=2-2]
\end{tikzcd}\]

Then, we can cover the $F_i$ by some objects coming from $\cal D_0$: $(v_{ij} : l_K(d_{ij}) \to F_i)_{ij}$:

\[\begin{tikzcd}
	{l_K(d_i)} & {C_p^*l_J(c_i)} \\
	{F_i} & {C_p^*l_J(c_i)} \\
	F & {C_p^*E}
	\arrow["{\alpha_iv_{ij}}", from=1-1, to=1-2]
	\arrow["{v_{ij}}"', from=1-1, to=2-1]
	\arrow[equals, from=1-2, to=2-2]
	\arrow["{\alpha_i}", from=2-1, to=2-2]
	\arrow["{v_i}"', from=2-1, to=3-1]
	\arrow["\lrcorner"{anchor=center, pos=0.125}, draw=none, from=2-1, to=3-2]
	\arrow["{C_p^*(u_i)}", from=2-2, to=3-2]
	\arrow["\alpha"', from=3-1, to=3-2]
\end{tikzcd}\]

Since both components $v_i v_{ij}$ and $u_i$ have been chosen to form epimorphic families, by the stability of epimorphic families under multi-composition, we are reduced to consider an object of the form $(l_K(d),l_J(c),\alpha : l_K(d) \to C_p^*l_J(c))$.

So let $(l_K(d),l_J(c),\alpha : l_K(d) \to C_p^*l_J(c))$ be an object of $(\widehat{\cal D}_K/C_{p}^{\ast})$ of this form.

Recall that, by the co-Yoneda lemma, $C_p^*l_J(c) \simeq \colim_{(p / c)} l_K\pi_p^c$ where $\pi_p^c : (p / c) \to \cal D$ is the projection functor. We have the diagrams:

\[\begin{tikzcd}
	{F_{\overline{d},u}} & {l_K(\pi_p^c(\overline{d},u))} \\
	{l_K(d)} & {C_p^*l_J(c)}
	\arrow[from=1-1, to=1-2]
	\arrow[from=1-1, to=2-1]
	\arrow["\lrcorner"{anchor=center, pos=0.125}, draw=none, from=1-1, to=2-2]
	\arrow["{a_k(ev_u)}", from=1-2, to=2-2]
	\arrow["\alpha"', from=2-1, to=2-2]
\end{tikzcd}\]

\noindent where $ev_u$ is given by the Yoneda lemma, evaluating the identity at $u$. Because of the expression of $C_p^*l_J(c)$ as a colimit, we have that it is in particular covered by the legs of this colimit cocone. Hence, their pullbacks $F_{\overline{d},u}$ are covering $l_K(d)$. Now we can cover the $F_{\overline{d},u}$ by some $l_K(d_i)$ such that the composition of the projections of this pullback with this covering come from $\cal D$:

\[\begin{tikzcd}
	{l_K(d_i)} \\
	& {F_{\overline{d},u}} & {l_K(\pi_p^c(\overline{d},u))} \\
	& {l_K(d)} & {C_p^*l_J(c)}
	\arrow[from=1-1, to=2-2]
	\arrow["{l_K(f_i)}", from=1-1, to=2-3]
	\arrow["{l_K(g_i)}"', from=1-1, to=3-2]
	\arrow[from=2-2, to=2-3]
	\arrow[from=2-2, to=3-2]
	\arrow["\lrcorner"{anchor=center, pos=0.125}, draw=none, from=2-2, to=3-3]
	\arrow["{a_k(ev_u)}", from=2-3, to=3-3]
	\arrow["\alpha"', from=3-2, to=3-3]
\end{tikzcd}\]

Finally, we see that $a_K(ev_u)l_K(f_i)$ is just $a_K(ev_{u\circ p(f_i)})$: the covering $(g_i)_i$ and the arrows $p(f_i)u$ give us the wished covering, as pictured:

\[\begin{tikzcd}
	{l_K(d_i)} & {C_p^*l_J(c)} \\
	{l_K(d)} & {C_p^*l_J(c)}
	\arrow["{\xi_p(d_i,c,u\circ p(f_i))}", shift left=2, from=1-1, to=1-2]
	\arrow["{l_K(g_i)}"', from=1-1, to=2-1]
	\arrow[equals, from=1-2, to=2-2]
	\arrow["\alpha"', from=2-1, to=2-2]
\end{tikzcd}\]

Hence, every object $(F,E,\alpha)$ can be covered by objects coming from $(p_0 / {\cal C}_0)$ for the canonical topology: finally, $(\widehat{\cal D}_K /C_p^*)$ is a topos.
\end{proof}

\begin{remark}
We must be careful about a particular case in this proof, which motivates the introduction of the slightly augmented categories $\cal D_0$ and $\cal C_0$. Indeed, it is important to observe that, in a general topos, the initial object is obtained as a colimit of generators by taking the empty diagram.

In the proof, we covered our object $E$ with generators as $(u_i : l_J(c_i) \to E)_i$, and then we formed the pullbacks:

\[\begin{tikzcd}
	{F_i} & {C_p^*l_J(c_i)} \\
	F & {C_p^*E}
	\arrow["{\alpha_i}", from=1-1, to=1-2]
	\arrow["{v_i}"', from=1-1, to=2-1]
	\arrow["\lrcorner"{anchor=center, pos=0.125}, draw=none, from=1-1, to=2-2]
	\arrow["{C_p^*(u_i)}", from=1-2, to=2-2]
	\arrow["\alpha"', from=2-1, to=2-2]
\end{tikzcd}\]

We stated that we can further cover the $F_i$ with objects coming from the site $(\cal D_0,K_0)$. It may happen, however, that one of these pullbacks is the initial object:

\[\begin{tikzcd}
	0 & E \\
	{C_p^*l_J(c_i)} & {C_p^*F}
	\arrow[from=1-1, to=1-2]
	\arrow[from=1-1, to=2-1]
	\arrow["\lrcorner"{anchor=center, pos=0.125}, draw=none, from=1-1, to=2-2]
	\arrow["\alpha", from=1-2, to=2-2]
	\arrow["{C_p^*(u_i)}"', from=2-1, to=2-2]
\end{tikzcd}\]

There is an obstruction to covering an object of the form $(0, l_J(c), \mathbf{!}_{C_p^*l_J(c)})$ in $(\widehat{\cal D}_K/C_p^*)$ with objects simply coming from $(p/\cal C)$: one might be tempted to consider the empty diagram, but it would not cover the object $l_J(c)$ (a covering in $(\widehat{\cal D}_K/C_p^*)$ must be epimorphic on \emph{both} components), and the initial object $0$ does not in general come from the site, making impossible to have any arrow $l_K(d) \to 0$ to cover it on the first component. This is the reason we need to make the initial object existing already at the site-level: for any object of the form $(0, l_J(c), \mathbf{!}_{C_p^*l_J(c)})$, it provides the following (trivial, since it is just an isomorphism) covering with an object from $(p_0/\cal C_0)$:

\[\begin{tikzcd}
	{l_{K_0}(0_{\cal D_0})} & 0 \\
	{C_p^*l_J(c)} & {C_p^*l_J(c)}
	\arrow["\simeq", from=1-1, to=1-2]
	\arrow["{\xi_{p_0}(0_{\cal D},c,\mathbf{!}_c)}"', from=1-1, to=2-1]
	\arrow["{\mathbf{!}_{C_p^*l_j(c)}}", from=1-2, to=2-2]
	\arrow[equals, from=2-1, to=2-2]
\end{tikzcd}\]
\end{remark}

The objects coming from $(p_0/\cal C_0)$ forming a generating set for the topos $(\widehat{\cal D}_K/C_p^*)$, this category is a suitable candidate for a site representation of this topos. In the following proposition, we establish the existence of an appropriate topology:

\begin{prop}
Let $p : ({\cal D},K) \to ({\cal C},J)$ be a comorphism of sites. We have a topology $(K_0 / J_0)$ on $(p_0 / \cal C_0)$ for which the covering sieves are those $S$ such that their projections on the two categories $\cal D$ and $\cal C$ are coverings for $K_0$ and $J_0$.
\end{prop}

\begin{proof}
The maximality and transitivity axioms are immediately verified in virtue of $K$ and $J$ satisfying them. 

For the pullback stability: let $(v_i,w_i)_i$ be a covering on an object $(d,c,u)$ of $(p_0 / \cal C_0)$, and $(f,g) : (d',c',u') \to (d,c,u)$ an arrow along which we want to pull it back. The situation is depicted in the following diagram:

\[\begin{tikzcd}
	& {p_0(d_i)} & {c_i} \\
	& {p_0(d)} & c \\
	{p_0(d')} & {c'}
	\arrow["{u_i}", from=1-2, to=1-3]
	\arrow["{p_0(v_i)}"', from=1-2, to=2-2]
	\arrow["{w_i}", from=1-3, to=2-3]
	\arrow["u", from=2-2, to=2-3]
	\arrow["{p_0(f)}", from=3-1, to=2-2]
	\arrow["{u'}"', from=3-1, to=3-2]
	\arrow["g"', from=3-2, to=2-3]
\end{tikzcd}\]

Since the $w_i$ are covering for $J_0$, we can pull them back along $g$ in order to obtain a $J_0$-covering $(w_{ij})_{ij}$ of $c'$, as depicted:

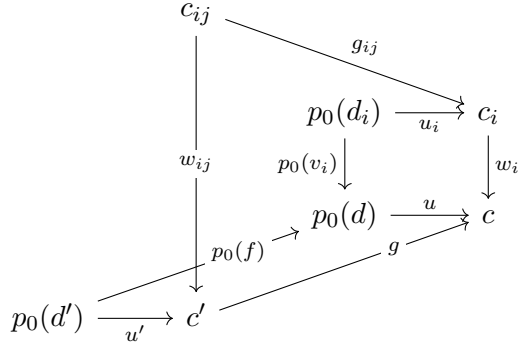
\[\begin{tikzcd}
	& {c_{ij}} \\
	&& {p_0(d_i)} & {c_i} \\
	&& {p_0(d)} & c \\
	{p_0(d')} & {c'}
	\arrow["{g_{ij}}", from=1-2, to=2-4]
	\arrow["{w_{ij}}"{description}, from=1-2, to=4-2]
	\arrow["{u_i}"', from=2-3, to=2-4]
	\arrow["{p_0(v_i)}"', from=2-3, to=3-3]
	\arrow["{w_i}", from=2-4, to=3-4]
	\arrow["u", from=3-3, to=3-4]
	\arrow["{p_0(f)}"{description, pos=0.7}, from=4-1, to=3-3]
	\arrow["{u'}"', from=4-1, to=4-2]
	\arrow["g"{description, pos=0.7}, from=4-2, to=3-4]
\end{tikzcd}\]

Here, we can already notice that the projection onto $\cal C_0$ of the pullback of the covering $(v_i,w_i)_i$ along $(f,g)$ is covering for $J_0$; indeed, all the objects $(0_{\cal D_0},c_{ij},\mathbf{!}_{c_{ij}})$ belong in it, giving us that every arrow $w_{ij}$ is in its projection onto $\cal C_0$. This is depicted in the following diagram:

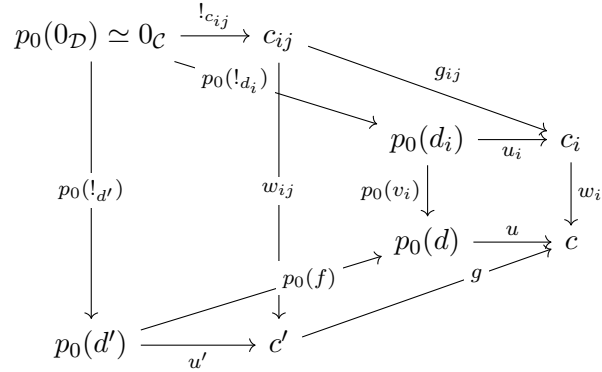
\[\begin{tikzcd}
	{p_0(0_{\cal D})\simeq0_{\cal C}} & {c_{ij}} \\
	&& {p_0(d_i)} & {c_i} \\
	&& {p_0(d)} & c \\
	{p_0(d')} & {c'}
	\arrow["{\mathbf{!}_{c_{ij}}}", from=1-1, to=1-2]
	\arrow["{p_0(\mathbf{!}_{d_i})}"{description, pos=0.3}, from=1-1, to=2-3]
	\arrow["{p_0(\mathbf{!}_{d'})}"{description}, from=1-1, to=4-1]
	\arrow["{g_{ij}}", from=1-2, to=2-4]
	\arrow["{w_{ij}}"{description}, from=1-2, to=4-2]
	\arrow["{u_i}"', from=2-3, to=2-4]
	\arrow["{p_0(v_i)}"', from=2-3, to=3-3]
	\arrow["{w_i}", from=2-4, to=3-4]
	\arrow["u", from=3-3, to=3-4]
	\arrow["{p_0(f)}"{description, pos=0.7}, from=4-1, to=3-3]
	\arrow["{u'}"', from=4-1, to=4-2]
	\arrow["g"{description, pos=0.7}, from=4-2, to=3-4]
\end{tikzcd}\]

There remains to see that the pullback of the covering $(v_i,w_i)_i$ along $(f,g)$ also has for projection onto $\cal D_0$ a covering for $K_0$. We can pull back the covering $(w_{ij})_{ij}$ previously obtained along $u' : p_0(d') \to c'$ in order to obtain a covering of $p_0(d')$; but $p_0$ is a comorphism of sites, so we can lift it into a covering $(v_k)_k$ for $K_0$ in $\cal D_0$:

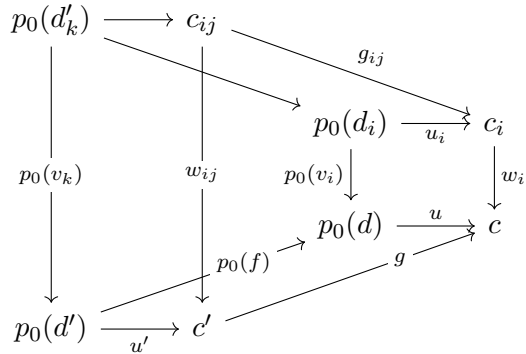
\[\begin{tikzcd}
	{p_0(d'_k)} & {c_{ij}} \\
	&& {p_0(d_i)} & {c_i} \\
	&& {p_0(d)} & c \\
	{p_0(d')} & {c'}
	\arrow[from=1-1, to=1-2]
	\arrow[from=1-1, to=2-3]
	\arrow["{p_0(v_k)}"{description}, from=1-1, to=4-1]
	\arrow["{g_{ij}}", from=1-2, to=2-4]
	\arrow["{w_{ij}}"{description}, from=1-2, to=4-2]
	\arrow["{u_i}"', from=2-3, to=2-4]
	\arrow["{p_0(v_i)}"', from=2-3, to=3-3]
	\arrow["{w_i}", from=2-4, to=3-4]
	\arrow["u", from=3-3, to=3-4]
	\arrow["{p_0(f)}"{description, pos=0.7}, from=4-1, to=3-3]
	\arrow["{u'}"', from=4-1, to=4-2]
	\arrow["g"{description, pos=0.7}, from=4-2, to=3-4]
\end{tikzcd}\]

This ensures that the projection onto $\cal D_0$ of the pullback of the $(v_i,w_i)$ contains a covering for $K_0$.

Hence, the pullback stability is satisfied for $(K_0/J_0)$, so that it is indeed a topology.
\end{proof}

With this definition, we can express a site presentation for the canonical relative site as a topos:

\begin{prop}
Let $p : ({\cal D},K) \to ({\cal C},J)$ be a comorphism of sites. We have $(\widehat{\cal D}_K /C_p^*) \simeq \widehat{(p_0 / {\cal C}_0)}_{(K_0 / J_0)}$, induced by the dense bimorphism (morphism and comorphism) of sites $ \xi_{p_0}: ((p_0/{\cal C}_0),(K_0/J_0)) \to ((\widehat{\cal D}_K/C_{p}^{\ast}),J_{(\widehat{\cal D}_K/C_{p}^{\ast})}^{can})$.
\end{prop}

\begin{proof}
By \ref{xilocal} we know that every object of the target topos is covered by objects coming from $\xi_{p_0}$, so that it is $J_{(\widehat{\cal D}_K/C_{p}^{\ast})}^{can}$-dense. It is immediate that $\xi_{p_0}$ reflects coverings, since they are, for both topologies, those covering on the two components. Using the same kind of techniques of localization as in the proof of \ref{xilocal}, it is straightforward to verify that $\xi_{p_0}$ also is $(K_0 / J_0)$-full, $(K_0 / J_0)$-faithful, and cover-lifting; hence it is a dense bimorphism of sites, inducing the wished equivalence.
\end{proof}

The following proposition is an explicit description of this equivalence when the topologies are trivial, yielding a clear correspondence between sheaves on $(p_0/\cal C_0)$ and the topos $(\widehat{\cal D}/(-\circ p))$. The proof is lengthy but straightforward.

\begin{prop}
Let $p : \cal D \to \cal C$ be a functor. We denote the trivial topologies on $\cal C$ and $\cal D$ by $M^{\cal C}$ and $M^{\cal D}$. The topologies $M^{\cal C}_0$ and $M^{\cal D}_0$ on $\cal C_0$ and $\cal D_0$ are just the trivial topologies where we added the empty sieves as coverings for the initial objects. The equivalence $\widehat{(p_0/\cal C_0)}_{(M^{\cal D}_0/M^{\cal C}_0)} \simeq (\widehat{\cal D}/(-\circ p))$ is given by:

\begin{enumerate}
    \item For the first direction: to an object $(F,E, \alpha : F \to E\circ p)$ of the comma $(\widehat{\cal D} /(-\circ p))$, we associate the sheaf $\overline{\alpha}$ on $((p_0/ {\cal C}_0),(M^{\cal D}_0/M^{\cal C}_0))$ sending an object $(d,c,u)$ to the following pullback:

    \[\begin{tikzcd}
	{\overline{\alpha}[u]} & {F(d)} \\
	{E(c)} & {Ep(d)}
	\arrow[from=1-1, to=1-2]
	\arrow[from=1-1, to=2-1]
	\arrow["\lrcorner"{anchor=center, pos=0.125}, draw=none, from=1-1, to=2-2]
	\arrow["{\alpha_d}", from=1-2, to=2-2]
	\arrow["{E(u)}"', from=2-1, to=2-2]
    \end{tikzcd}\]
    \item For the second direction: to a sheaf $Q$ on $((p_0/ {\cal C}_0),(M^{\cal D}_0/M^{\cal C}_0))$, we associate the triplet $(Q[-,p(-),1_{p(-)}],Q[0_{\cal D_0},-,\mathbf{!}_{-}],\alpha_Q)$ where the component at an object $d$ of $\cal D$ of the natural transformation $\alpha_Q^d$ is defined with the help of the action of $Q$ by: $Q(\mathbf{!}_d,1_{p(d)}) : Q[d,p(d),1_{p(d)}] \to Q[0_{\cal D_0},p(d),\mathbf{!}_{p(d)}]$.
\end{enumerate}
\end{prop}
\qed

When we endow the categories with non trivial topologies, it is straightforward to verify that:

\begin{prop}
Let $p : ({\cal D},K) \to ({\cal C},J)$ be a comorphism of sites. Under the equivalence of the previous proposition, we have that the sheafification functor is given by: $a_{(K_0 / J_0)}: (\widehat{\cal D} / (- \circ p)) \to (\widehat{\cal D}_K /C_p^*)$ sends an object $(P,Q, \alpha : P \to Qp)$ to the object given by the arrow $a_K(\alpha) : a_K(P) \to C_p^*a_J(Q)$. Its right adjoint is given by the pullback:

\[\begin{tikzcd}
	{F'} & i_K(F) \\
	{E \circ p} & {i_Ka_K(E\circ p)}
	\arrow["{\eta_E}"', from=2-1, to=2-2]
	\arrow["i_K(\alpha)", from=1-2, to=2-2]
	\arrow["{i_{(K_0 / J_0)}(\alpha)}"', from=1-1, to=2-1]
	\arrow[from=1-1, to=1-2]
	\arrow["\lrcorner"{anchor=center, pos=0.125}, draw=none, from=1-1, to=2-2]
\end{tikzcd}\]

along the unit $\eta_F$ of the sheafification for $K$.
\end{prop}
\qed

Now that we have established that the comma category $(\widehat{\cal D}_K/C_p^*)$ is a topos, we can describe certain geometric morphisms relating it to the toposes $\widehat{\cal C}_J$ and $\widehat{\cal D}_K$ from which it is built. In particular, since it arises as a specific case of \emph{Artin gluing}, in which the toposes $\widehat{\cal C}_J$ and $\widehat{\cal D}_K$ are glued along $C_p^*$, we have natural embeddings of both into this glued topos. But our case is even more specific: in order to construct an Artin gluing, in general, we just need a finite limits preserving functor; here, we not only have that $C_p^*$ preserves finite limits, but it also preserves arbitrary colimits. This fact makes $\widehat{\cal C}_J$ a \emph{dense} open subtopos of $(\widehat{\cal D}_K/C_p^*)$, and gives us that the embedding of toposes it induces comes with a retraction. The following proposition describes those geometric morphisms relating $\widehat{\cal C}_J$ to $(\widehat{\cal D}_K/C_p^*)$, as already expressed in \cite{nlab:artin_gluing}:

\begin{prop}\label{rhodef}
Let $p : ({\cal D},K) \to ({\cal C},J)$ be a comorphism of sites.

\begin{enumerate}[(i)]
    \item We have the geometric morphism $i_{C_p} : \widehat{\cal C}_J \hookrightarrow  (\widehat{\cal D}_K /C_p^*) $ having for inverse image the functor $\pi_{\widehat{\cal C}_J} : (\widehat{\cal D}_K /C_p^*) \to \widehat{\cal C}_J $ and for direct image the functor $\tau_{\widehat{\cal C}_J} :  \widehat{\cal C}_J \to (\widehat{\cal D}_K /C_p^*) $ sending an object $E$ to the object $(C_p^*E,E,1_{C_p^*E})$. 
    \item This geometric morphism is an étale open embedding, it can be identified with the geometric morphism: $ (\widehat{\cal D}_K /C_p^*)/(\mathbf{0}_{\widehat{\cal D}_K},\mathbf{1}_{\widehat{\cal C}_J},\mathbf{!}_{\mathbf{0}}) \to (\widehat{\cal D}_K /C_p^*)$, where $ !_{\mathbf{0}} : \mathbf{0}_{\widehat{\cal D}_K} \to C_p^*(\mathbf{1}_{\widehat{\cal C}_J})$ is the unique possible morphism from the initial to the terminal objects of $\widehat{\cal D}_K$.
    \item In particular, this étale geometric morphism comes with an essential image: $\mathcal{O}_{\cal C} : \widehat{\cal C}_J \to (\widehat{\cal D}_K /C_p^*)$ sending an object $E$ to $(\mathbf{0}_{\widehat{\cal D}_K},E,\mathbf{!}_{C_p^*E})$.
    \item The functor $\tau_{\widehat{\cal C}_J}$ also has a right adjoint $(\rho_{C_p})_* : (\widehat{\cal D}_K /C_p^*) \to  \widehat{\cal C}_J$ which sends a triplet $(E,F,\alpha)$ to the pullback:

    \[\begin{tikzcd}
	{(\rho_{C_p})_*(E,F,\alpha)} & E \\
	{{C_p}_*F} & {{C_p}_*C_p^*E}
	\arrow["{\mu^p_E}", from=1-2, to=2-2]
	\arrow["{{C_p}_*(\alpha)}"', from=2-1, to=2-2]
	\arrow[from=1-1, to=2-1]
	\arrow[from=1-1, to=1-2]
	\arrow["\lrcorner"{anchor=center, pos=0.125}, draw=none, from=1-1, to=2-2]
    \end{tikzcd}\]
    \noindent where $\mu^p_E$ is the unit of $({C_p}_* \vdash C_p^*)$. 
    \item This functor together with the further left adjoint $\pi_{\widehat{\cal C}_J}$ of $\tau_{\widehat{\cal C}_J}$ give an essential and connected geometric morphism: $\rho_{C_p} : (\widehat{\cal D}_K /C_p^*) \to  \widehat{\cal C}_J$
    \item These two geometric morphisms constitute an adjoint retraction: we have $ \rho_{C_p} \circ i_{C_p} \simeq 1_{\widehat{\cal C}_J}$ with $i_{C_p}\vdash \rho_{C_p}$.
    \item In this light, the embedding $i_{C_p}$ is not only a geometric morphism, but also a relative geometric $i_{C_p} : [1_{\widehat{\cal C}_J}] \to [\rho_{C_p}]$ over $\widehat{\cal C}_J$.
\[\begin{tikzcd}
	{\widehat{\cal C}_J} && {(\widehat{\cal D}_K/C_p^*)} \\
	& {\widehat{\cal C}_J}
	\arrow["{i_{C_p}}", hook, from=1-1, to=1-3]
	\arrow["{1_{\widehat{\cal C}_J}}"', equals, from=1-1, to=2-2]
	\arrow["{\rho_{C_p}}", two heads, from=1-3, to=2-2]
\end{tikzcd}\]
    Since $[1_{\widehat{\cal C}_J}]$ is the terminal relative topos, this inclusion is a $\widehat{\cal C}_J$-point of the relative topos $[\rho_{C_p}]$.
\end{enumerate}
\end{prop}
\qed

\begin{remark}
This geometric morphism $\rho_{C_p} : (\widehat{\cal D}_K /C_p^*) \to  \widehat{\cal C}_J$ allows to see this canonical relative site not only as a topos, or a relative site, but also as a relative topos over $\widehat{\cal C}_J$.
\end{remark}

Again, as it can be seen for example in \cite{nlab:artin_gluing}, we have the following proposition expressing how $\widehat{\cal D}_K$ is related to $(\widehat{\cal D}_K/C_p^*)$:

\begin{prop}
Let $p : ({\cal D},K) \to ({\cal C},J)$ be a comorphism of sites.

\begin{enumerate}[(i)]
    \item The functor $\pi_{\cal D} : (\widehat{\cal D}_K /C_p^*) \to \widehat{\cal D}_K $ is the inverse image part of a geometric embedding $j_{C_p}$ having for direct image the functor $\mathbf{t}_{\widehat{\cal D}_K} : \widehat{\cal D}_K \to (\widehat{\cal D}_K /C_p^*)$ sending an object $F$ of $\widehat{\cal D}_K$ to $(F,\mathbf{1}_{\widehat{\cal C}_J},\mathbf{!}_{F})$ where $\mathbf{!}_{F} : F \to C_p^*(\mathbf{1}_{\widehat{\cal C}_J}) \simeq \mathbf{1}_{\widehat{\cal D}_K}$ is the unique possible arrow into the terminal object of $\widehat{\cal D}_K$.
    \item This embedding is a \emph{closed} subtopos of  $(\widehat{\cal D}_K /C_p^*)$.
    \item It is moreover a relative embedding, that is, the following triangle of geometric morphisms commutes:

    \[\begin{tikzcd}
    	{\widehat{\cal D}_K} && {(\widehat{\cal D}_K/C_p^*)} \\
    	& {\widehat{\cal C}_J}
    	\arrow["{j_{C_p}}", hook, from=1-1, to=1-3]
    	\arrow["{C_p}"', from=1-1, to=2-2]
    	\arrow["{\rho_{C_p}}", two heads, from=1-3, to=2-2]
    \end{tikzcd}\]
    \item If moreover the geometric morphism $C_p$ is essential, we have another functor $\eta_{C_p}:\widehat{\cal D}_K \to (\widehat{\cal D}_K /C_p^*) $ sending $F$ to the unit $(F,{C_p}_!(F),F \to C_p^*{C_p}_!(F))$ of $C_p^* \vdash {C_p}_!$ at $F$, which is an essential image for the geometric embedding $j_{C_p}$.
\end{enumerate}
\end{prop}
\qed

\begin{remark}
These two inclusions give us a lax triangle of toposes:
\[\begin{tikzcd}
	{\widehat{\cal D}_K} && {\widehat{\cal C}_J} \\
	& {(\widehat{\cal D}_K/C_p^*)}
	\arrow["{C_p}", from=1-1, to=1-3]
	\arrow[""{name=0, anchor=center, inner sep=0}, "{j_{C_p}}"', hook, from=1-1, to=2-2]
	\arrow["{i_{C_p}}", hook', from=1-3, to=2-2]
	\arrow["{\alpha_{C_p}}"{description}, shorten <=10pt, shorten >=16pt, Rightarrow, from=0, to=1-3]
\end{tikzcd}\]

Indeed, an object $(F,E,\alpha : F \to C_p^*E)$ of $(\widehat{\cal D}_K/C_p^*)$ is sent to $F$ by $j_{C_p}^*$ and to $C_p^*E$ by $C_p^*i_{C_p}^*$, which are two objects of $\widehat{\cal D}_K$ naturally related by $\alpha$.
\end{remark}

Our study of relative toposes relies on their presentation via fibrations, or more generally, via comorphisms of sites $p : (\cal D,K) \to (\cal C,J)$. It is therefore natural to seek descriptions of all the previously discussed geometric morphisms arising from a relative topos, in terms of a representation of that relative topos through a comorphism. The following proposition gives us such presentations for the previously discussed geometric morphisms relating $\widehat{\cal C}_{J}$ and $(\widehat{\cal D}_K/C_p^*)$:

\begin{prop}
Let $p : ({\cal D},K) \to ({\cal C},J)$ be a comorphism of sites. 
\begin{enumerate}
    \item The functor $\pi_{\cal C_0} : ((p_0 / {\cal C}_0),(K_0 / J_0)) \to ({\cal C_0},J_0)$ is a morphism of sites, and $i_{C_p}$ and $ \Sh(\pi_{\cal C})$ are the same geometric morphisms, in the sense that the following square commutes:
    \[\begin{tikzcd}
	{\widehat{\cal C_0}_{J_0}} & {\widehat{\cal C}_{J}} \\
	{\widehat{(p_0/\cal C_0)}_{(K_0/J_0)}} & {(\widehat{\cal D}_K/C_p^*)}
	\arrow["{\Sh(i_0)}", from=1-1, to=1-2]
	\arrow["\simeq"{description}, shift right=2, draw=none, from=1-1, to=1-2]
	\arrow["{\Sh(\pi_{\cal C_0})}"', hook', from=1-1, to=2-1]
	\arrow["{i_{C_p}}", hook, from=1-2, to=2-2]
	\arrow["{\Sh(\xi_{p_0})}"', shift right, from=2-1, to=2-2]
	\arrow["\simeq"{description}, shift left=2, draw=none, from=2-1, to=2-2]
    \end{tikzcd}\]
    
    \item The functor $\pi_{\cal C_0} : ((p_0 / {\cal C}_0),(K_0 / J_0)) \to ({\cal C_0},J_0)$ is a comorphism of sites, and $\rho_{C_p}$ and $C_{\pi_{\cal C_0}}$ are the same geometric morphisms, in the sense that the following square commutes:

\[\begin{tikzcd}
	{\widehat{\cal C_0}_{J_0}} & {\widehat{\cal C}_{J}} \\
	{\widehat{(p_0/\cal C_0)}_{(K_0/J_0)}} & {(\widehat{\cal D}_K/C_p^*)}
	\arrow["{C_{i_0}}"', from=1-2, to=1-1]
	\arrow["\simeq"{description}, shift left=2, draw=none, from=1-2, to=1-1]
	\arrow["{C_{\pi_{\cal C_0}}}", two heads, from=2-1, to=1-1]
	\arrow["{\rho_{C_p}}"', two heads, from=2-2, to=1-2]
	\arrow["{C_{\xi_{p_0}}}", shift left, from=2-2, to=2-1]
	\arrow["\simeq"{description}, shift right=2, draw=none, from=2-2, to=2-1]
\end{tikzcd}\]
\end{enumerate}
\end{prop}

\begin{proof}
The fact that $\pi_{\cal C_0}$ preserves coverings is immediate in virtue of the definition of $(K_0/J_0)$. The conditions for it to be a morphism of sites can be easily checked in the light of $p_0$ preserving the initial object; for example, the first condition: if we have an object $c$ of $\cal C_0$ then $1_c : c \to \pi_{\cal C_0}(0_{\cal D_0},c,\mathbf{!}_c)$ is a generalized element as needed; the other conditions go the same way.

To deduce that $\pi_{\cal C_0}$ is a comorphism of sites., let $(d,c,u)$ be an object of $(p_0/\cal C_0)$ and $(w_i : c_i \to c)_i$ be a $J_0$-covering of $c$. We can pull it back along $u$, and then lift it into $\cal D_0$, since $p_0$ is a comorphisme. It is pictured in the following diagram:

\[\begin{tikzcd}
	{p(d_k)} \\
	& {c_{ij}} & {c_i} \\
	& {p(d)} & c
	\arrow["{a_k}", from=1-1, to=2-2]
	\arrow["{p(v_k)}"', from=1-1, to=3-2]
	\arrow[from=2-2, to=2-3]
	\arrow["{w_{ij}}", from=2-2, to=3-2]
	\arrow["{w_i}", from=2-3, to=3-3]
	\arrow["u"', from=3-2, to=3-3]
\end{tikzcd}\]

We can add the arrows of the form $(\mathbf{!}_d,w_i)$:

\[\begin{tikzcd}
	{p(0_{\cal D_0})\simeq 0_{\cal C_0}} & {c_i} \\
	{p(d)} & c
	\arrow["{\mathbf{!}_{c_i}}", from=1-1, to=1-2]
	\arrow["{p(\mathbf{!}_d)}"', from=1-1, to=2-1]
	\arrow["{w_i}", from=1-2, to=2-2]
	\arrow["u"', from=2-1, to=2-2]
\end{tikzcd}\]

\noindent to the arrows $(v_k,w_i)_k$ already obtained, so that the projection of all those arrows onto $\cal C_0$ gives us the covering $(w_i)_i$ which is obviously included in itself, and all those arrows constitute a covering for $(K_0/J_0)$: their projection on $\cal C_0$ is the $J_0$-covering $(w_i)_i$, and their projection on $\cal D_0$ contains the $K_0$-covering $(v_k)_k$. So the arrows $(v_k,w_i)_k$ together with the arrows $(\mathbf{!}_d,w_i)_i$ form a covering $R$ for $(K_0/J_0)$ with $\pi_{\cal C_0}(R) \subset (w_i)_i$: that is, $\pi_{\cal C_0}$ is a comorphism of sites. 

Now, if we want to check the commutation of the first square of geometric morphisms, we can reduce to verify it for the inverse images evaluated on the generators of $(\widehat{\cal D}_K/C_p^*)$. A triplet $(l_K(d),l_J(c),u : l_K(d) \to C_p^*l_J(c))$ is sent by $i_{C_p}^* \simeq \pi_{\widehat{\cal C}_J}$ to the object $l_J(c)$, which in turn is sent to the object $l_{J_0}(i_0(c))$ (where $i_0 : \cal C \to \cal C_0$ is the obvious inclusion). On the other hand, the object $(l_K(d),l_J(c),u : l_K(d) \to C_p^*l_J(c)$ is sent to the generator $l_{(K_0/J_0)}(j_0(d),i_0(c),u:p_0(d)\to c)$ which is, in turn, sent to its projection $l_{J_0}(i_0(c))$, so that the square commutes.

The second square also commutes, since the inverse images of the geometric morphisms it involves are the direct images of the previous square which is commutative.
\end{proof}

We also have the same kind of site-presentation for the inclusion of the other subtopos:

\begin{prop}
Let $p : ({\cal D},K) \to ({\cal C},J)$ be a comorphism of sites. The functor $\pi_{\cal D_0} : ((p_0 / {\cal C_0}),(K_0 / J_0)) \to ({\cal D_0},K_0)$ is a morphism of sites, having for left adjoint the comorphism of sites $\tau_{\cal D}: ({\cal D}_0,K_0) \to ((p_0 / {\cal C}_0),(K_0 / J_0))$ (sending $d$ to $(d,p_0(d),1_{p_0(d)})$). We have that $j_{C_p}$ and $ \Sh(\pi_{\cal D_0}) \simeq C_{\tau_{\cal D_0}}$ are equivalent geometric morphisms, in the sense that the following square commutes:

\[\begin{tikzcd}
	{\widehat{\cal D_0}_{K_0}} & {\widehat{\cal D}_{K}} \\
	{\widehat{(p_0/\cal C_0)}_{(K_0/J_0)}} & {(\widehat{\cal D}_K/C_p^*)}
	\arrow["{\Sh(j_0)}", from=1-1, to=1-2]
	\arrow["\simeq"{description}, shift right=2, draw=none, from=1-1, to=1-2]
	\arrow["{\Sh(\pi_{\cal D_0})}"', hook', from=1-1, to=2-1]
	\arrow["{j_{C_p}}", hook, from=1-2, to=2-2]
	\arrow["{\Sh(\xi_{p_0})}"', shift right, from=2-1, to=2-2]
	\arrow["\simeq"{description}, shift left=2, draw=none, from=2-1, to=2-2]
\end{tikzcd}\]

Moreover, if $p$ is continuous, the functor $\eta_{C_p}:\widehat{\cal D}_K \to (\widehat{\cal D}_K /C_p^*)$ is induced by $\tau_{\cal D_0}$ as a continuous comorphism of sites.
\end{prop}

\begin{proof}
Since $\pi_{\cal D_0}$ preserves the coverings and has a left adjoint, it is a morphism of sites and its left adjoint $\tau_{\cal D_0}$ is a comorphism inducing the same geometric morphism. The proof that the square commutes goes the same way as in the previous proposition. The proof that $p$ being continuous implies that $\tau_{\cal D_0}$ is continuous is lengthy but straightforward.
\end{proof}

Now, in this setting of the canonical stack seen as a topos, we can express the $\eta$-extension of a morphism of sites as the inverse image of a geometric morphism between the two canonical stacks:

\begin{prop}
Let $p : ({\cal D},K) \to ({\cal C},J)$ and $p' : ({\cal D'},K') \to ({\cal C},J)$ be comorphisms of sites, and $A: ({\cal D},K) \to ({\cal D'},K')$ a morphism of sites together with $\phi : p'A \Rightarrow p$ a natural transformation. The $\eta$-extension $\widetilde{A}^* : (\widehat{\cal D}_K /C_p^*) \to (\widehat{\cal D'}_{K'} /C_{p'}^*)$ of $A$ is the inverse image of a geometric morphism $\widetilde{A} : (\widehat{\cal D'}_{K'} /C_{p'}^*) \to (\widehat{\cal D}_K /C_p^*)$, having for direct image $\widetilde{A}_* : (\widehat{\cal D'}_{K'} / C_{p'}^*) \to (\widehat{\cal D}_K / C_{p}^*)$ sending a triplet $(F',E, \alpha' : F' \to C_{p'}^*E)$ to the triplet given by the pullback:
\[\begin{tikzcd}
	F & {\Sh(A)_*(F')} \\
	{C_p^*(E)} & {\Sh(A)_*C_{p'}^*(E)}
	\arrow[from=1-1, to=1-2]
	\arrow["{\widetilde{A}_*[\alpha]}"', from=1-1, to=2-1]
	\arrow["\lrcorner"{anchor=center, pos=0.125}, draw=none, from=1-1, to=2-2]
	\arrow["{\Sh(A)_*(\alpha')}", from=1-2, to=2-2]
	\arrow["{a_K(- \circ \phi)}"', from=2-1, to=2-2]
\end{tikzcd}\]
\end{prop}

\begin{proof}
Straightforward in the light of the universal property of the pullback and the adjunction $\Sh(A)_* \vdash \Sh(A)^*$.
\end{proof}

In the case of the base category $\cal C$ being cartesian, this geometric morphism admits a description at the site-level:

\begin{prop}
Let $p : ({\cal D},K) \to ({\cal C},J)$ and $p' : ({\cal D'},K') \to ({\cal C},J)$ be comorphisms of sites with $\cal C$ a cartesian category, and $A: ({\cal D},K) \to ({\cal D'},K')$ a morphism of sites together with $\phi : p'A \Rightarrow p$ a natural transformation. 

The functor $(A_0/{\cal C_0}) : ((p_0/\cal C_0),(K_0/J_0)) \to ((p'_0/\cal C_0),(K'_0/J_0))$ sending an object $(d,c,u:p(d)\to c)$ to $(A_0(d),c,u\phi^0_d:p'(A_0(d))\to c)$ (where $A_0 : \cal D_0 \to \cal D'_0$ is the extension of $A$ preserving the initial objects) is a morphism of sites. Moreover, $\Sh(A/\cal C)$ and $ \widetilde{A} : (\widehat{\cal D'}_{K'} /C_{p'}^*) \to (\widehat{\cal D}_K /C_p^*)$ are equivalent geometric morphisms, in the sense that the following square commutes:

\[\begin{tikzcd}
	{(\widehat{\cal D}_K/C_p^*)} & {(\widehat{\cal D'}_{K'}/C_p'^*)} \\
	{\widehat{(p_0/\cal C_0)}_{(K_0/J_0)}} & {\widehat{(p_0'/\cal C_0)}_{(K_0'/J_0)}}
	\arrow["{\Sh(\xi_{p_0})}"', from=1-1, to=2-1]
	\arrow["\simeq"{description}, shift left=3, draw=none, from=1-1, to=2-1]
	\arrow["{\widetilde{A}}"', from=1-2, to=1-1]
	\arrow["{\Sh(\xi_{p'_0})}", from=1-2, to=2-2]
	\arrow["\simeq"{description}, shift right=3, draw=none, from=1-2, to=2-2]
	\arrow["{\Sh(A_0/\cal C_0)}", shift left, from=2-2, to=2-1]
\end{tikzcd}\]
\end{prop}

\begin{proof}
Since $A$ is a morphism of sites it is cover-preserving, so that $(A_0/\cal C_0)$ also preserves covering. The three conditions for it to be a morphism of sites are proved the same way, so we just give the proof for the first condition. Let $(d',c,u':p_0(d') \to c)$ be an object of $(p'_0/\cal C_0)$. Since $A$ is a morphism of sites (it is easy to verify that it implies $A_0$ morphism of sites), there exist a $K'_0$-covering $(v_i' : d'_i \to d')_i$ together with generalized elements $(e_i' : d'_i \to A(d_i))_i$. With the help of the terminal object of the cartesian category $\cal C_0$ (it is easy to verify that $\cal C$ cartesian implies $\cal C_0$ cartesian), we can derive from it generalized elements for $(A_0/\cal C_0)$, as the $(e_i',\mathbf{!}_{c}) : (d'_i,c,u'p_0'(v_i')) \to (A_0/\cal C_0)(d_i,\mathbf{1}_{\cal C_0},\mathbf{!}_{p_0(d_i)})$ where $\mathbf{1}_{\cal C_0}$ is the terminal object. In picture, this is represented as:
\[\begin{tikzcd}
	& {p_0'(d')} & c \\
	& {p_0'(d'_i)} & c \\
	{p'_0(A(d_i))} & {p_0(d_i)} & {\mathbf{1}_{\cal C_0}}
	\arrow["{u'}", from=1-2, to=1-3]
	\arrow["{p_0'(v_i')}", from=2-2, to=1-2]
	\arrow["{u'p_0'(v_i')}"', from=2-2, to=2-3]
	\arrow["{p'_0(e_i')}"', from=2-2, to=3-1]
	\arrow[equals, from=2-3, to=1-3]
	\arrow["{\mathbf{!}_{c}}", from=2-3, to=3-3]
	\arrow["{\phi^0_{d_i}}"', from=3-1, to=3-2]
	\arrow["{\mathbf{!}_{p_0(d_i)}}"', from=3-2, to=3-3]
\end{tikzcd}\]

The commutation of the square of geometric morphisms can easily be checked on the generator for the inverse images.
\end{proof}

Finally, we notice that $A$ induces a relative geometric morphism if and only if it also is the case for $\widetilde{A}^*$, relating these two geometric morphisms:

\begin{prop}\label{characcanonicalstack}
Let $p : ({\cal D},K) \to ({\cal C},J)$ and $p' : ({\cal D'},K') \to ({\cal C},J)$ be comorphisms of sites, and $A: ({\cal D},K) \to ({\cal D'},K')$ a morphism of sites together with $\phi : p'A \Rightarrow p$ a natural transformation. It is a morphism of sites over $({\cal C},J)$ (that is, it induces a relative geometric morphism) if and only if $\widetilde{A}$ is a relative geometric morphism, that is, the following diagram commutes:

\[\begin{tikzcd}
	{(\widehat{\cal D'}_{K'} / C_{p'}^*)} && {(\widehat{\cal D}_{K} /C_{p}^*)} \\
	& {\widehat{\cal C}_J}
	\arrow["{\rho_{C_{p'}}}"', from=1-1, to=2-2]
	\arrow["{\rho_{C_{p}}}", from=1-3, to=2-2]
	\arrow["{\widetilde{A}}", from=1-1, to=1-3]
\end{tikzcd}\]

Moreover, the following square of geometric morphisms commutes, that is, $\Sh(A)$ is the restriction of $\widetilde{A}$:
\[\begin{tikzcd}
	{(\widehat{\cal D}_K/C_p^*)} & {(\widehat{\cal D}_{K'}/C_{p'}^*)} \\
	{\widehat{\cal D}_K} & {\widehat{\cal D'}_{K'}}
	\arrow["{\widetilde{A}}"', from=1-2, to=1-1]
	\arrow["{j_{C_p}}", hook', from=2-1, to=1-1]
	\arrow["{j_{C_{p'}}}"', hook', from=2-2, to=1-2]
	\arrow["{\Sh(A)}", from=2-2, to=2-1]
\end{tikzcd}\]
\end{prop}

\begin{proof}
Recall that the inverse images of $\rho_{C_p}$ and $\rho_{C_{p'}}$ are the functors sending an object of the base topos into the terminal of the respective fibers over this object. But \ref{etaextensionproperties} says exactly that $A$ induces a relative geometric morphism if and only if $\widetilde{A}^*$ preserves finite limits fiberwise; and this is the case if and only if it preserves the terminal objects fiberwise, since it already preserves finite limits globally.

The commutation of the square is immediate, since the inverse images of $j_{C_p}$ and $j_{C_{p'}}$ are the projections, and $\widetilde{A}^*$ acts just as $\Sh(A)$ on the first component.
\end{proof}

\section{From classical to local fibrations}\label{section3}

As recalled in \ref{section:notprel}, fibrations between categories equipped with compatible topologies naturally act as comorphisms of sites, and thus provide a suitable framework for presenting relative toposes. However, while fibrations can be equipped with various topologies, the definition of a fibration itself does not involve any topological data: the lifting conditions only refer to isomorphisms, they make no reference to the topologies.

Building on the fact that fibrations, viewed as comorphisms, provide a natural way to present relative toposes, we are led to consider arbitrary comorphisms \( p : (\mathcal{D}, K) \to (\mathcal{C}, J) \) as a more general means of inducing relative toposes. Between the structured case of fibrations equipped with topologies (relative sites) and the more general case of comorphisms, it is natural to consider a “localized” notion of fibration that takes into account the topologies of the sites involved. We will see that while the classical notion of fibration is well-suited to the case of toposes equipped with their canonical topologies, it is not the case in general: the presence of topologies allows for the definition of a more refined and general notion of fibration.

\subsection{Locally cartesian arrows}

While an arbitrary comorphism of sites is not, in general, a fibration, it always admits a canonical functor relating it to one: the canonical stack of the relative topos it induces. This observation allows us to define a notion of locally cartesian arrow:

\begin{defn}
Let $p : ({\cal D},K) \to ({\cal C}, J)$ be a comorphism of sites. We define an arrow $f$ in $\cal D$ to be \emph{locally cartesian} (with respect to $K$) if $\eta_{\cal D}(f)$ is a cartesian arrow of its canonical relative site.
\end{defn}

\begin{remark}
An immediate computation as in \ref{etamorphfib} shows that usual cartesian arrows are locally cartesian ones, for any topology.
\end{remark}

Although this definition is abstract, it can be reformulated more concretely: an arrow is locally cartesian if, locally, it satisfies the usual lifting properties of cartesian arrows. This is the following proposition:

\begin{prop}
An arrow $f:d' \to d$ in $({\cal D},K)$ is locally cartesian if it satisfies the two conditions:

\begin{enumerate}[(i)]
    \item for any arrow $h:p(d'')\to p(d')$ in $\cal C$ such that $p(f)\circ h=p(g)$ for some arrow $g: d'' \to d$ in $\cal D$, there exist a $K$-covering family $(v_i:d''_i \to d'')_i$ and for each $i$ an arrow $h_{i}: d''_i \to d'$ such that $h\circ p(v_i)=p(h_i)$ and $f\circ h_i = g\circ v_i$. 
    \item for any two arrows $h,h' : d'' \to d'$ such that $f\circ h = f \circ h'$ and $p(h)=p(h')$, we have a $K$-covering $(v_i : d''_i \to d'')_i$ such that $h\circ v_i = h'\circ v_i$.
\end{enumerate}
\end{prop}

\begin{proof}
Cartesian arrows of the canonical relative site are those represented by a pullback square in $\widehat{\cal D}_K$ (Proposition \ref{enumloccart}(i)), so we can look at the presheaf-level when an arrow $f$ in $\cal D$ is sent by $\eta$ to an arrow giving a cartesian square. This is summed up in the following diagram:

\[\begin{tikzcd}
	{y_{\cal D}(d)} \\
	& \bullet & {{\cal C}(p-,pd)} \\
	& {y_{\cal D}(d')} & {{\cal C}(p-,pd')}
	\arrow[dashed, from=1-1, to=2-2]
	\arrow[from=2-2, to=2-3]
	\arrow["{p-}"', from=3-2, to=3-3]
	\arrow[from=2-2, to=3-2]
	\arrow["{p(f)\circ-}", from=2-3, to=3-3]
	\arrow["{f \circ-}"', from=1-1, to=3-2]
	\arrow["{p-}", from=1-1, to=2-3]
	\arrow["\lrcorner"{anchor=center, pos=0.125}, draw=none, from=2-2, to=3-3]
\end{tikzcd}\]

We want to characterize when the canonical arrow from $y_{\cal D}(d)$ to the pullback is sheafified into an isomorphism. The pullback being computed pointwise, we have that at some $\overline{d}$ it is given by $\{(g:\overline{d} \to d', h:p(\overline{d}) \to p(d)) \mid p(f) \circ h = p(g) \}$. We also have that the canonical arrow from $y_{\cal D}(d)$ to the pullback is given, at $\overline{d}$, by: $g \mapsto (f \circ g,p(g))$. Hence, the first point of the proposition is the local surjectivity condition, and the second one is the local injectivity condition.
\end{proof}

We see that this definition mainly depends on the topology $K$; we only need $J$ to be such that $p$ is a comorphism.

We may be interested in seeing which arrows are the locally cartesian ones in a fibration. The following proposition characterize them:

\begin{prop}\label{kcartesianforfib}
Let $p : (\mathcal{G}(\mathbb D),K) \to ({\cal C},J)$ be a relative site. An arrow $(f,u): (x',c') \to (x,c)$ in $\mathcal{G}(\mathbb D)$ is locally cartesian if and only if $l_K(1,u)$ is an iso, that is, the two following explicit conditions are satisfied:

\begin{enumerate}[(i)]
    \item The arrow $(1,u)$ is $K$-covering.
    \item For every arrow $g: c'' \to c'$ in $\cal C$, and vertical arrows $v,v' :   x'' \to \mathbb D (g)(x')$ such that $\mathbb D(g) (u)v = \mathbb D(g) (u)v'$, there exists a covering $(g_i,u_i)_i$ fpr $K$ of $(x'',c'')$ such that $\mathbb D(g_i)(v)u_i = \mathbb D(g_i)(v')u_i$.
\end{enumerate}
\end{prop}

\begin{proof}
Recall that $\eta_{\mathbb D}$ is a morphism of fibrations (Proposition 3.14 \cite{bartolicaramello}), so it preserves cartesian and vertical part of an arrow: we have $\eta_{\mathbb D}(f,u) =  \eta_{\mathbb D}(f,1) \circ \eta_{\mathbb D}(1,u)$, with $\eta_{\mathbb D}(f,1)$ a cartesian arrow. Then, we have that  $\eta_{\mathbb D}(f,u)$ is cartesian (that is, $(f,u)$ is locally cartesian) if and only if $\eta_{\mathbb D}(1,u)$ is. But, $\eta_{\mathbb D}(1,u)$ is a vertical arrow, hence it is a cartesian arrow if and only if it is an isomorphism, that is $l_K(1,u)$ is an isomorphism. 

The two conditions enumerated are just the locally epi and locally mono characterizations.
\end{proof}

Since we define locally cartesian arrows in terms of the cartesian arrows of the canonical relative site, it is necessary to study the latter. The following proposition shows, in particular, that they coincide with the usual cartesian arrows

\begin{prop}\label{canonicalcartesian}
 Let $f:{\cal F}\to {\cal E}$ be a geometric morphism. Then

\begin{enumerate}[(i)]
    \item An arrow for the canonical relative site $\pi_f : (({\cal F}/f^*),J_f) \to \cal E$ is locally cartesian if and only if it is cartesian.

    \item For a stack $\mathbb D$ on a site $({\cal C},J)$, an arrow $(f,u)$ in $\mathcal{G}(\mathbb D)$ is cartesian if and only if for every covering $(g_i)_i$ for the topology $J$, we have that $(f,u)\circ(g_i,1)$ is cartesian.
    
    \item If $(e_i)_i$ is an epimorphic family in $\cal E$, and if $(e,g) : (E,F,\alpha) \to (E',F',\alpha ')$ is an arrow of the canonical relative site of $f$ for which $(e,g) \circ (\alpha^*(f^*e_i),e_i)$ is cartesian for every $i$, then $(e,g)$ is cartesian. 
\end{enumerate}
\end{prop}

\begin{proof}

(i) Let $(E,F,\alpha : E \to f^*F)$, $(E',F',\alpha' : E' \to f^*F')$ be two objects of the comma category, and $(e,g): (E,F,\alpha : E \to f^*F) \to (E',F',\alpha' : E' \to f^*F')$ a locally cartesian arrow. Let $(E'',F'',\alpha : E'' \to f^*F'')$ be another object, together with an arrow $(e',g') : (E'',F'',\alpha : E'' \to f^*F'') \to (E',F',\alpha' : E' \to f^*F')$ and an arrow $h : F'' \to F $ in $\cal F$. Since $(e,g)$ is locally cartesian, we have a covering $(e_i'',g_i'')_i$ of $(E'',F'',\alpha : E'' \to f^*F'')$ such that we can locally lift $h$ as in the diagram below:

\[\begin{tikzcd}
	E & {f^*F} \\
	{E'} & {f^*F'} \\
	{E''} & {f^*F''} \\
	{E''_i} & {f^*F''_i}
	\arrow["e"', from=1-1, to=2-1]
	\arrow["{f^*g}"', from=1-2, to=2-2]
	\arrow["{e'}", from=3-1, to=2-1]
	\arrow["{f^*g'}", from=3-2, to=2-2]
	\arrow["\alpha", from=1-1, to=1-2]
	\arrow["{\alpha'}", from=2-1, to=2-2]
	\arrow["{\alpha''}", from=3-1, to=3-2]
	\arrow["{e''_i}", from=4-1, to=3-1]
	\arrow["{f^*g''_i}"', from=4-2, to=3-2]
	\arrow["{\alpha''_i}", from=4-1, to=4-2]
	\arrow["{f^*h}"', bend left=-35, from=3-2, to=1-2]
	\arrow["{e_i}", bend right=-35, from=4-1, to=1-1]
\end{tikzcd}\]

Now we note that, for a local fibration in general, such a local lifting of an arrow $h$ gives us a locally compatible family:

\[\begin{tikzcd}
	& \bullet \\
	& \bullet \\
	\bullet && \bullet
	\arrow["{p(u_{ij})}"', from=3-1, to=3-3]
	\arrow["{p(v_i)}", from=3-1, to=2-2]
	\arrow["h"', from=2-2, to=1-2]
	\arrow["{p(v_j)}"', from=3-3, to=2-2]
	\arrow["{p(h_j)}"', bend right=45, from=3-3, to=1-2]
	\arrow["{p(h_i)}", bend left=45, from=3-1, to=1-2]
\end{tikzcd}\]

Indeed, if $u_{ij}$ is such that $v_j \circ u_{ij} = v_i $, then $h_i \circ u_{ij} =^K h_j$, by the local uniqueness of such a lifting in the definition of a locally cartesian arrow. 

But recall that, in our case, saying that the $(e_i'',g_i'')$ are covering is exactly saying that the $e_i''$ are jointly effective epimorphic. Hence, the $e_i$ form a locally compatible family for the sieve $(e_i'')_i$, but since the coverings are epimorphic ones this local compatibility is just a compatibility. As the covering are \emph{effective} epimorphic families, we can glue the $e_i$ together, giving us some $e'' : E'' \to E'$. This arrow complete $h$ into a morphism $(e,h) :(E'',F'',\alpha'' : E'' \to f^*F'') \to (E,F,\alpha : E \to f^*F)$ (which makes the square commute, as the precompositions of the two different paths with the epimorphic familiy  $(e_i'')_i$ coincide). The uniqueness of such a $e$ is immediate in virtue of the $(e_i)_i$ being jointly epimorphic.

(ii) If an arrow is cartesian, then its precomposition with any cartesian arrow is still cartesian. In the other direction, we know that $(f,u)(g_i,1) = (fg_i,(\mathbb D(g_i)(u))$ so that $(f,u)(g_i,1)$ is cartesian for every $i$, this means that the $\mathbb D(g_i)(u)$ are isomorphisms, i.e. $u$ is locally an isomorphism, that is (by the prestack property of $\mathbb D$), $u$ is an isomorphism.

(iii) Since we assume the $(e_i)_i$ to be covering, our thesis follows from (ii).
\end{proof}

\begin{remark}

Recall that, in general, if we have two consecutive commutative squares such that the lower one is a pullback:

\[\begin{tikzcd}
	\bullet & \bullet \\
	\bullet & \bullet \\
	\bullet & \bullet
	\arrow[from=2-1, to=2-2]
	\arrow[from=3-1, to=3-2]
	\arrow[from=2-1, to=3-1]
	\arrow[from=2-2, to=3-2]
	\arrow["\lrcorner"{anchor=center, pos=0.125}, draw=none, from=2-1, to=3-2]
	\arrow[from=1-1, to=2-1]
	\arrow[from=1-2, to=2-2]
	\arrow[from=1-1, to=1-2]
\end{tikzcd}\]
then the upper square is a pullback if and only if the outer rectangle also is. Note that, in the previous proposition, taking the geometric morphism $f$ to be the identity, point (iii) corresponds to: if we have a family of upper squares that are pullbacks and such that their two vertical arrows are jointly epimorphics, then the lower square is a pullback if and only if the outer rectangles are also pullbacks.

This constitutes some converse to the pullback lemma, with further hypothesis that we have some epimorphic family. This partial converse to the pullback lemma is true in a topos but not in general. 
\end{remark}

Point (i) of the precedent result being largely due to the canonicity of the topology on $\cal F$, the same proof method is applicable when we work with the canonical topology on any category:

\begin{prop}\label{prop_cartesiancanonicaltopology}
Let $({\cal D},J_{can}^{\cal D}) \to ({\cal C},J)$ be a comorphism of sites. The $J_{can}^{\cal D}$-cartesian arrows are exactly the cartesian ones.
\end{prop}

\begin{proof}
A cartesian arrow is in particular a locally cartesian arrow for any topology, so one side is obvious. For the other direction, consider such a commutative triangle with $f$ a locally cartesian arrow:

\[\begin{tikzcd}
	{p(d')} & {p(d)} \\
	{p(d'')}
	\arrow["h", from=2-1, to=1-1]
	\arrow["{p(f)}", from=1-1, to=1-2]
	\arrow["{p(g)}"', from=2-1, to=1-2]
\end{tikzcd}\]

If we have two liftings $h',h'' : d'' \to d'$ for $h$, that is $p(h') = p(h'') = h$ and $fh' = fh'' =g$, we have, in virtue of the local unicity of such a lifting for $f$, that $h'v_i = hv_i$ for some covering $(v_i)_i$ of $J^{\cal D}_{can}$. But these $v_i$ are jointly epimorphic by definition of the canonical topology, so $h = h'$. 

In order to prove the existence of a lifting for $h$, we will use the existence of a local lifting coming for the fact that $f$ is locally cartesian, and we will be able to glue all of them, because the covering sieves are \emph{effective} epimorphic ones. 

So, since $f$ is locally cartesian, we can locally lift $h$ as $h_i : d'_i \to d'$ on a sieve $S$ covering for $J_{can}^{\cal D}$, that is: $hp(v_i) = p(h_i)$ and $fh_i = gv_i$, for every $v_i \in S$. We see that the $h_i$ constitute a compatible family for the sheaf $y_{\cal D}(d')$ on the sieve $S$ covering $d''$: if we have $u_{ij}$ such that $v_iu_{ij}  =v_j$, then $h_iu_{ij}=h_j$. Indeed, we have that $fh_j = gv_j$, and hence $fh_j = gv_iu_{ij} = fh_iu_{ij}$, also $p(h_iu_{ij}) = hp(v_iu_{ij})$, and then $p(h_iu_{ij}) = hp(v_j) = p(h_j)$. By the uniqueness of a lifting applied to $h_iu_{ij}$ and $h_j$, we obtain the compatibility. Finally, since the sieve $S$ is an effective epimorphic one, by the sheaf property of $y_{\cal D}(d')$ we can glue all the $h_i$ into one $h'$ satisfying $p(h') = h$ and $fh' = g$. 
\end{proof}

In the light of \ref{canonicalcartesian}, we obtain the following important property:

\begin{prop}\label{etareflectsloccart}
Let $p:({\cal D},K) \to ({\cal C},J)$ be a comorphism of sites. Then the functor $\eta_{\cal D}$ reflects locally cartesian arrows.
\end{prop}

\begin{proof}
A locally cartesian arrow is by definition an arrow sent on a cartesian arrow by $\eta_{\cal D}$, which is in particular a $J_{C_p}$ cartesian arrow. And conversely, because of \ref{canonicalcartesian} the only $J_{C_p}$ cartesian arrows are the cartesian ones, which forces an arrow sent to a $J_{C_p}$-cartesian one to already being a locally cartesian one in $\cal D$.
\end{proof}

\subsection{Local fibrations}

As cartesian arrows give rise to the notion of fibration, locally cartesian arrows give rise to a localization of the notion of fibration:

\begin{defn}
A comorphism of sites $p:({\cal D},K) \to ({\cal C},J)$ is said to be a \emph{local fibration} if for any arrow $f : c \to p(d)$ there is a $J$-covering family $f_i:p(d_i) \to c$ such that $f \circ f_i = p(\hat{f_i})$ with $(\hat{f_i} : d_i \to d)_i$ a family of locally cartesian arrows. 
\end{defn}

\begin{remarks}
\begin{enumerate}[(a)]
    \item In a fibration, cartesian arrows are locally cartesian arrows since we can take for cover the maximal ones, and it is immediate that any relative site $p : (\mathcal{G}(\mathbb D),K) \to ({\cal C},J)$ is a local fibration.
    \item We will also see that a continuous functor between two toposes endowed with their canonical topologies is a local fibration if and only if it is a fibration.
    \item These examples show that fibrations are very particular cases of local fibrations.
\end{enumerate}
\end{remarks}

Recall that in a classical fibration we have an orthogonal factorization system:

\begin{prop}
Let $p :\cal G(\mathbb D) \to \cal C$ be a fibration. Any arrow $(f,u)$ of  $\cal G(\mathbb D)$ can be universally factorized as a cartesian arrow pre-composed with a vertical one: $(f,u)=(f,1)(1,u)$.
\end{prop}

As inspired by the horizontal-vertical factorization in an usual fibration, we can locally factorize every arrow with the help of locally cartesian ones:

\begin{prop}\label{locallyfacthorizontal}
Let $f : d' \to d$ be an arrow in a local fibration $p : ({\cal D},K) \to ({\cal C},J)$, there is a $K$-covering family $(v_i : p(d'_i) \to p(d'))_i$, a family $(\widehat{f_i} : d_i \to d)_i$ of locally cartesian arrows, arrows $a_i : d'_i \to d_i$, and arrows $s_i: p(d_i) \to p(d')$ such that $fv_i = \widehat{f_i}a_i$ and $p(f)s_i=p(\widehat{f_i})$.
\end{prop}

\begin{proof}

Let $f : d' \to d$ be any arrow in $\cal D$. We can send it in the canonical relative site with the help of the $\eta$ functor, where it is represented as:

\[\begin{tikzcd}
	{l_K(d')} & {l_K(d)} \\
	{C_p^*l_Jp(d')} & {C_p^*l_Jp(d)}
	\arrow["{l_K(f)}", from=1-1, to=1-2]
	\arrow["{\eta(d')}"', from=1-1, to=2-1]
	\arrow["{\eta(d)}", from=1-2, to=2-2]
	\arrow["{C_p^*l_Jp(f)}"', from=2-1, to=2-2]
\end{tikzcd}\]

Since $p(f)$ has for codomain an object coming from $p$, we can locally factor it as locally cartesian arrows: $p(f) \circ v_i = p(\hat{f_i})$ for $(v_i : p(d'_i) \to p(d'))_i$ a $K$-covering. We represent this situation in

\[\begin{tikzcd}
	{l_K(d'_i)} & {l_K(d')} & {l_K(d)} \\
	{C_p^*l_Jp(d'_i)} & {C_p^*l_Jp(d')} & {C_p^*l_Jp(d)}
	\arrow["{l_K(v_i)}", from=1-1, to=1-2]
	\arrow["{l_K(\widehat{f_i})}"{description}, bend left =30, from=1-1, to=1-3]
	\arrow["{\eta(d'_i)}"', from=1-1, to=2-1]
	\arrow["{l_K(f)}", from=1-2, to=1-3]
	\arrow["{\eta(d')}"', from=1-2, to=2-2]
	\arrow["{\eta(d)}", from=1-3, to=2-3]
	\arrow["{C_p^*l_J(v_i)}"', from=2-1, to=2-2]
	\arrow["{C_p^*l_Jp(\widehat{f_i})}"{description}, bend right = 30, from=2-1, to=2-3]
	\arrow["{C_p^*l_Jp(f)}"', from=2-2, to=2-3]
\end{tikzcd}\]

Then, since $(v_i)_i$ is a covering for $J$ on an object coming from $p$ which is a comorphism of sites, we can lift it into a covering $(v'_j)_j$ for $K$, as pictured in the following diagram:

\[\begin{tikzcd}
	& {l_K(d'_i)} & {l_K(d')} & {l_K(d)} \\
	{C_p^*l_Jp(d'_j)} & {C_p^*l_Jp(d'_i)} & {C_p^*l_Jp(d')} & {C_p^*l_Jp(d)}
	\arrow["{l_K(v_i)}", from=1-2, to=1-3]
	\arrow["{l_K(\widehat{f_i})}"{description}, bend left = 30, from=1-2, to=1-4]
	\arrow["{\eta(d'_i)}"', from=1-2, to=2-2]
	\arrow["{l_K(f)}", from=1-3, to=1-4]
	\arrow["{\eta(d')}", from=1-3, to=2-3]
	\arrow["{\eta(d)}", from=1-4, to=2-4]
	\arrow["{C_p^*l_J(a_j)}", from=2-1, to=2-2]
	\arrow["{C_p^*l_J(v'_j)}"', bend right = 18, from=2-1, to=2-3]
	\arrow["{C_p^*l_J(v_i)}", from=2-2, to=2-3]
	\arrow["{C_p^*l_Jp(\widehat{f_i})}"{description}, bend right = 30, from=2-2, to=2-4]
	\arrow["{C_p^*l_Jp(f)}"', from=2-3, to=2-4]
\end{tikzcd}\]

Since the $\widehat{f_i}$ are locally cartesian and we are in a situation where $p(\widehat{f_i})a_j=p(fv_j')$, we can locally lift the $a_j$ to $\cal D$:

\[\begin{tikzcd}
	&& {l_K(d'_i)} & {l_K(d')} & {l_K(d)} \\
	{C_p^*l_Jp(d'_{jk})} & {C_p^*l_Jp(d'_j)} & {C_p^*l_Jp(d'_i)} & {C_p^*l_Jp(d')} & {C_p^*l_Jp(d)}
	\arrow["{l_K(v_i)}", from=1-3, to=1-4]
	\arrow["{l_K(\widehat{f_i})}"{description}, bend left = 30, from=1-3, to=1-5]
	\arrow["{\eta(d'_i)}"', from=1-3, to=2-3]
	\arrow["{l_K(f)}", from=1-4, to=1-5]
	\arrow["{\eta(d')}", from=1-4, to=2-4]
	\arrow["{\eta(d)}", from=1-5, to=2-5]
	\arrow["{C_p^*l_Jp(v'_{jk})}"', from=2-1, to=2-2]
	\arrow["{C_p^*l_Jp(a_{jk})}"', bend left = 30, from=2-1, to=2-3]
	\arrow["{C_p^*l_J(a_j)}", from=2-2, to=2-3]
	\arrow["{C_p^*l_J(v'_j)}"', bend right = 18, from=2-2, to=2-4]
	\arrow["{C_p^*l_J(v_i)}", from=2-3, to=2-4]
	\arrow["{C_p^*l_Jp(\widehat{f_i})}"{description}, bend right = 30, from=2-3, to=2-5]
	\arrow["{C_p^*l_Jp(f)}"', from=2-4, to=2-5]
\end{tikzcd}\]

\noindent where $(v'_{ij})_{ij}$ is a covering for $K$, $p(a_{jk})=p(v'_{jk})a_j$ and $\widehat{f_i}a_{jk} = fv_j'v_{jk}'$.

Finally, the covering arrows $v'_jv'_{jk}$ together with the $a_{jk}$ and the mentioned  $\widehat{f_i}$ and $v_i$ satisfy the proposition conditions.
\end{proof}

Finally, local fibrations naturally come with their own notion of morphisms:

\begin{defn}
Let $p : ({\cal D},K) \to ({\cal C},J)$ and $p' : ({\cal D'},K') \to ({\cal C},J)$ be two local fibrations. We call a functor $A : ({\cal D},K) \to ({\cal D'},K')$ such that $p'A \simeq p$ a morphism of local fibrations when it sends locally cartesian arrows for $p$ to locally cartesian ones for $p'$.
\end{defn}

\section{Local fibrations and morphisms of relative toposes}\label{section4}

In this section we will study how to induce relative geometric morphisms between relative toposes presented through this notion of local fibration. In \cite{bartolicaramello} we saw a concrete characterization (Proposition 3.11.) of those lax morphisms of sites over a base site $(\cal C,J)$ inducing relative geometric morphisms. More precisely, we gave a concrete description of those morphisms of sites $A$ with natural transformation $\phi : p'A \Rightarrow p$, as pictured:

\[\begin{tikzcd}
	{(\cal D,K)} && {(\cal D',K')} \\
	& {(\cal C,J)}
	\arrow["A", from=1-1, to=1-3]
	\arrow[""{name=0, anchor=center, inner sep=0}, "p"', from=1-1, to=2-2]
	\arrow["{p'}", from=1-3, to=2-2]
	\arrow["\phi"{description}, shorten <=9pt, shorten >=9pt, Rightarrow, from=1-3, to=0]
\end{tikzcd}\]

\noindent such that, at topos-level, the geometric morphism $\Sh(A)$ together with the natural transformation $\widetilde{\phi}$ make the following triangle of geometric morphisms commutative:

\[\begin{tikzcd}
	{\widehat{\cal D}_K} && {\widehat{\cal D'}_{K'}} \\
	& {\widehat{\cal C}_J}
	\arrow[""{name=0, anchor=center, inner sep=0}, "{C_p}"', from=1-1, to=2-2]
	\arrow["{\Sh(A)}"', from=1-3, to=1-1]
	\arrow["{C_{p'}}", from=1-3, to=2-2]
	\arrow["\simeq"{description, pos=0.6}, shift left=2, draw=none, from=1-3, to=0]
	\arrow["{\widetilde{\phi}}"{description, pos=0.6}, shift right=2, draw=none, from=1-3, to=0]
\end{tikzcd}\]

With the help of this concrete characterization, we showed that in the more structured case, where the relative toposes are induced by relative sites, a sufficient condition holds: when a morphism of sites $A: (\cal G(\mathbb D),K) \to (\cal G(\mathbb D'),K')$ is also a morphism of fibrations, it induces a relative geometric morphism (Proposition 3.12 \cite{bartolicaramello}).

The aim of this section is to show that an analogous result holds in the case of local fibrations, and that, in this case, it is not only a sufficient condition but also a necessary one. Moreover, the proof we obtain here is more conceptual than the proof of Proposition 3.12 \cite{bartolicaramello}, which was more computational in nature.

Recall from \ref{etaextsubsection} that a morphism of sites $A$ between two relative sites induces a morphism of relative toposes if and only if its $\eta$-extension is itself a morphism of fibrations. The aim of our proof is to conceptually explain why the property of $A$ being a morphism of fibrations is transmitted to its $\eta$-extension $\widetilde{A}^*$.
 
Important properties of the $\eta$-extension are recalled in the subsection. In particular, we have that for a morphism of sites $A : (\mathcal{G}(\mathbb D),K) \to (\mathcal{G}(\mathbb D'),K')$ between two relative sites over $\cal C$ such that $p'A \simeq p$, we have a commutative square:

\[\begin{tikzcd}
	{\mathcal{G}(\mathbb D)} & {\mathcal{G}(\mathbb D')} \\
	{(\widehat{\mathcal{G}(\mathbb D)}_K/C_p^*)} & {(\widehat{\mathcal{G}(\mathbb D')}_{K'}/C_{p'}^*)}
	\arrow["{\widetilde{A}^*}"', from=2-1, to=2-2]
	\arrow["{\eta_{\mathbb D}}", from=1-1, to=2-1]
	\arrow["{\eta_{\mathbb D'}}"', from=1-2, to=2-2]
	\arrow["A", from=1-1, to=1-2]
\end{tikzcd}\]

In this light, for our case of relative sites, we may expect to be able to deduce that $\widetilde{A}^*$ preserves cartesian arrows from the fact that $A$ already does. To achieve this, a first idea would be to locally express every cartesian arrow of the canonical relative site $((\widehat{\mathcal{G}(\mathbb D)}_K/C_p^*),J_{C_p})$ using cartesian arrows coming from $\mathcal{G}(\mathbb D)$, and then to exploit a possible reflectivity of cartesian arrows through $\eta_{\mathbb D}$, together with the commutativity of the above square and the fact that both $A$ and $\eta_{\mathbb D}$ preserve cartesian arrows, to conclude. However, when examining the behavior of $\eta_{\mathbb D}$ with respect to cartesian arrows, we see that if it already preserves them, there is no need for it to reflect them. Nevertheless, as we saw in \ref{etareflectsloccart}, $\eta$ reflects \emph{locally} cartesian arrows; this suggests that the general framework of local fibrations is more appropriate for carrying out this conceptual proof.

The following two subsections present the two main steps of the proof, corresponding to two different types of localizations of cartesian arrows in the canonical stack. Our goal is not only to localize a cartesian arrow, but also to ensure that the image of such a cartesian arrow under $A$ remains cartesian, by relying on the cartesianness of the localizations. The first step is to show that an arrow is locally cartesian if and only if its pullbacks along a covering family are locally cartesian. The second step is to show that an arrow of a certain type is locally cartesian if and only if, in the associated canonical stack, it can be represented as a colimit of cartesian arrows coming from the site.

In the third subsection, we use these characterizations in order to deduce our main theorems: a morphism of sites between two local fibrations induces a morphism of relative toposes if and only if it is a morphism of local fibrations; and a relative Diaconescu's theorem for local fibrations.

In the final subsection, we apply our results to the specific case of fibrations, which are examples of local fibrations. In particular, we obtain a new proof of Diaconescu's theorem for relative sites, along with a general characterization of those morphisms of sites between relative sites that induce morphisms of relative toposes.

\subsection{Pullbacks of locally cartesian arrows}

Let $p : (\mathcal{D}, K) \to (\mathcal{C}, J)$ be a local fibration. We are going to show that every cartesian arrow in the canonical relative site of $C_p$ can, in a certain sense, be locally expressed as cartesian arrows coming from $\eta_{\cal D}$. This will allow us to deduce the following: if $A : (\mathcal{D}, K) \to (\mathcal{D}', K')$ is both a morphism of local fibrations and a morphism of sites, then its $\eta$-extension $\widetilde{A}^*$ is a morphism of fibrations. Indeed, since we have the commutation $\eta A \simeq \widetilde{A}^* \eta$, our goal is to deduce that a cartesian arrow $f$ in $(\widehat{\cal D}_K/C_p^*)$ is sent to a cartesian arrow $\widetilde{A}^*(f)$. We will show that this follows from the fact that its \textquotedblleft localizations\textquotedblright\ $\widetilde{A}^*(\eta(f_i)) \simeq \eta A(f_i)$ are cartesian, given that $A$ preserves locally cartesian arrows.

In order to obtain this local representation, we recall that $(\widehat{\cal D}_K/C_p^*)$ is a stack, of which we write the associated indexed category $S_{C_p}$. Let's consider an arrow $ \alpha : F \to C_p^*E$ and an arrow $f : E' \to E$ in $\widehat{\cal C}_J$, this gives us a cartesian arrow in $(\widehat{\cal D}_K/C_p^*)$, as pictured in the square :

\[\begin{tikzcd}
	{S_{C_p}(f)[\alpha]} & {C_p^*E'} \\
	F & {C_p^*E}
	\arrow["{\alpha'}", from=1-1, to=1-2]
	\arrow[from=1-1, to=2-1]
	\arrow["\lrcorner"{anchor=center, pos=0.125}, draw=none, from=1-1, to=2-2]
	\arrow["{C_p^*(f)}", from=1-2, to=2-2]
	\arrow["\alpha"', from=2-1, to=2-2]
\end{tikzcd}\]

As $\eta_{\cal D}$ is dense (see for example Proposition 2.8 \cite{bartolicaramello}), we are able to cover $[\alpha]$, seen as an object of $((\widehat{\cal D}_K/C_p^*),J_{C_p})$, with objects coming from $\cal D$: 
\[\begin{tikzcd}
	{S_{C_p}(f)[\alpha]} & {C_p^*E'} \\
	F & {C_p^*E} \\
	{l_K(d_i)} & {C_p^*l_J(pd_i)}
	\arrow["{\alpha'}", from=1-1, to=1-2]
	\arrow[from=1-1, to=2-1]
	\arrow["\lrcorner"{anchor=center, pos=0.125}, draw=none, from=1-1, to=2-2]
	\arrow["{C_p^*(f)}", from=1-2, to=2-2]
	\arrow["\alpha", from=2-1, to=2-2]
	\arrow["{g_i'}", from=3-1, to=2-1]
	\arrow["{\eta_{\cal D}(d_i)}"', from=3-1, to=3-2]
	\arrow["{C_p^*(g_i)}"', from=3-2, to=2-2]
\end{tikzcd}\]

Since  $(\widehat{\cal D}_K/C_p^*)$ is a cartesian fibration, we can take the pullbacks of our cartesian arrow along these covering arrows in the comma category, which give us again cartesian arrows (\ref{enumloccart} (ii)). This proves the following result:

\begin{prop}\label{pullbackcart}
Let $p : ({\cal D},K) \to ({\cal C},J)$ be a comorphism of sites. In the stack $(\widehat{\cal D}_K/C_p^*) \to \widehat{\cal C}_J$, for each cartesian arrow (here we use the notations for $(\widehat{\cal D}_K/C_p^*)$  seen as a Grothendieck construction, not as a comma category) $(f,1): [S_{\cal D}(\alpha)] \to [\alpha]$ we have a covering (for $J_{C_p}$) constituted of objects coming from $\cal D$ : $(\eta_{\cal D}(d_i) \to [\alpha])_i$ with the pullbacks of $(f,1)$ being cartesian arrows:

\[\begin{tikzcd}
	{P_i} & {\eta_{\cal D}(d_i)} \\
	{[S_{C_p}(\alpha)]} & {[\alpha]}
	\arrow["{(f_i,1)}", from=1-1, to=1-2]
	\arrow[from=1-1, to=2-1]
	\arrow["\lrcorner"{anchor=center, pos=0.125}, draw=none, from=1-1, to=2-2]
	\arrow[from=1-2, to=2-2]
	\arrow["{(f,1)}"', from=2-1, to=2-2]
\end{tikzcd}\]
\end{prop}\qed

In order to use the precedent localization, we need to introduce the following definition:

\begin{defn}
We say that $p:({\cal D},K) \to ({\cal C},J)$ is a \emph{cartesian} local fibration when $\cal D$ has finite limits and $p$ preserves them. In such a setting, we say that an arrow $f : d' \to d$ in $\cal D$ is $K$-cartesian if there exists a covering family $(u_i : d_i \to d)_i$ such that the pullbacks of $f$ along every $u_i$ is a locally cartesian arrow.
\end{defn}

The following proposition gives us a characterization of locally cartesian arrows with the help of its pullbacks along a covering: 

\begin{prop}\label{kloccartesian}
Let $p:({\cal D},K) \to ({\cal C},J)$ be a cartesian local fibration. An arrow is locally cartesian if and only if it is $K$-cartesian. 
\end{prop}

\begin{proof}
A locally cartesian arrow is obviously $K$-cartesian for the identity as a covering. Now assume that $f$ is $K$-cartesian, and let $g$ and $h$ be such $p(g) = p(f)\circ h$, the situation is as follows:

\[\begin{tikzcd}
	& {p(d_i')} & {p(d_i)} \\
	& {p(d')} & {p(d)} \\
	\\
	{p(d'')}
	\arrow["{p(f_i)}", from=1-2, to=1-3]
	\arrow["{p(u_i')}"', from=1-2, to=2-2]
	\arrow["\lrcorner"{anchor=center, pos=0.000125}, draw=none, from=1-2, to=2-3]
	\arrow["{p(u_i)}", from=1-3, to=2-3]
	\arrow["{p(f)}"', from=2-2, to=2-3]
	\arrow["h", from=4-1, to=2-2]
	\arrow["{p(g)}"', bend right = 20, from=4-1, to=2-3]
\end{tikzcd}\]

we pullback $g$ along the $u_i$, and since $p$ preserves finite limits, we now have:
\[\begin{tikzcd}
	{p(d''_i)} \\
	& {p(d_i')} & {p(d_i)} \\
	& {p(d')} & {p(d)} \\
	{p(d'')}
	\arrow[shorten <=6pt, dashed, from=1-1, to=2-2]
	\arrow["{p(g_i)}", bend left = 20, from=1-1, to=2-3]
	\arrow["\lrcorner"{anchor=center, pos=0.000125}, draw=none, from=1-1, to=3-3]
	\arrow["{p(u_i'')}"', from=1-1, to=4-1]
	\arrow["{p(f_i)}", from=2-2, to=2-3]
	\arrow["{p(u_i')}"', from=2-2, to=3-2]
	\arrow["\lrcorner"{anchor=center, pos=0.000125}, draw=none, from=2-2, to=3-3]
	\arrow["{p(u_i)}", from=2-3, to=3-3]
	\arrow["{p(f)}", from=3-2, to=3-3]
	\arrow["h"', from=4-1, to=3-2]
	\arrow["{p(g)}"', bend right = 20, from=4-1, to=3-3]
\end{tikzcd}\]
With the dashed arrow existing and making the diagram commutative because $p(g) = p(f) \circ h$. Now this dashed arrow can be locally lifted by the fact that $f_i$ are locally cartesian:

\[\begin{tikzcd}
	{p(d''_{ij})} \\
	{p(d''_i)} \\
	& {p(d'_i)} & {p(d_i)} \\
	& {p(d')} & {p(d)} \\
	{p(d'')}
	\arrow["{p(u_{ij}'')}"', from=1-1, to=2-1]
	\arrow["{p(h_{ij})}"{pos=0.3}, bend left = 12, from=1-1, to=3-2]
	\arrow[shorten <=6pt, dashed, from=2-1, to=3-2]
	\arrow["{p(g_i)}"{pos=0.7}, bend left = 18, from=2-1, to=3-3]
	\arrow["\lrcorner"{anchor=center, pos=0.000125}, draw=none, from=2-1, to=4-3]
	\arrow["{p(u_i'')}"', from=2-1, to=5-1]
	\arrow["{p(f_i)}", from=3-2, to=3-3]
	\arrow["{p(u_i')}"', from=3-2, to=4-2]
	\arrow["\lrcorner"{anchor=center, pos=0.000125}, draw=none, from=3-2, to=4-3]
	\arrow["{p(u_i)}", from=3-3, to=4-3]
	\arrow["{p(f)}", from=4-2, to=4-3]
	\arrow["h"', from=5-1, to=4-2]
	\arrow["{p(g)}"', bend right = 20, from=5-1, to=4-3]
\end{tikzcd}\]

The $u_i''$ are covering as pullbacks of the $u_i$ which are covering, hence the multicomposition of $u_{ij}''$ and $u_i''$ gives a covering, and the $u_i' \circ h_{ij}$ constitute local liftings of $h$ for this covering.

The local unicity goes as follows: if we have $h$ and $h'$ such that $p(h) = p(h')$ and $fh = fh'$, we pull them back along the $u_i$, giving two coverings $(v_i)_i$ and $(v_i')_i$, we take the intersection of these coverings, and, modulo a multicomposition, the local unicity of a litfting for the $f_i$ gives that $h$ and $h'$ are locally equals. 
\end{proof}

We are particularly interested in such arrows in the canonical stack of a relative topos, as we want to be able to locally express every cartesian arrow as locally cartesian arrows coming from its site. Since the canonical relative sites are cartesian fibrations, we can apply the previous result to obtain:

\begin{prop}\label{Jfcartimpliescart}
Let $f : {\cal F}\to {\cal E}$ be a relative topos. Any $J_f$-cartesian arrow in $(({\cal F}/f^*),J_f)$ is cartesian.
\end{prop}

\begin{proof}
Because of \ref{kloccartesian} we have that a $J_f$-cartesian arrow is locally cartesian, and because of \ref{canonicalcartesian} (i) we have that any locally cartesian arrow is cartesian.
\end{proof}

\subsection{Colimits of cartesian arrows}

In this subsection we present the second kind of localization we will use to locally characterize cartesian arrows in the canonical relative site. Indeed, the arrows $(f_i,1) : P_i \to \eta_{\cal D}(d_i)$ obtained in the previous subsection are not entirely coming from the site: only their codomains do. The following proposition shows how to localize them so that they arise from locally cartesian arrows of the site via $\eta_{\mathcal{D}}$:

\begin{prop}\label{localsurjcart}
Let $p : ({\cal D},K) \to ({\cal C},J)$ be a local fibration and a cartesian arrow $(f,g) : [\alpha] \to \eta_{\cal D}(d)$ in $(\widehat{\cal D}_K/C_p^*)$. There exist a covering for the (absolute) canonical topology on $(\widehat{\cal D}_K/C_p^*)$ constituted of cartesian arrows $((v_i,u_i) : \eta_{\cal D}(d_i) \to [\alpha])_i$, and locally cartesian arrows $(f'_i : d_i \to d)_i$ such that $(f,g) \circ (u_i,v_i) = \eta_{\cal D} (f'_i)$. 
\end{prop}

\begin{proof}
We picture our cartesian arrow $(f,g)$ in the following diagram:

\[\begin{tikzcd}
	E & {l(d)} \\
	{C_p^*F} & {C_p^*lp(d)}
	\arrow["f", from=1-1, to=1-2]
	\arrow["\alpha"', from=1-1, to=2-1]
	\arrow["\lrcorner"{anchor=center, pos=0.125}, draw=none, from=1-1, to=2-2]
	\arrow["{\eta_{\cal D}}", from=1-2, to=2-2]
	\arrow["{C_p^*(g)}"', from=2-1, to=2-2]
\end{tikzcd}\]

Then, we have a covering of $F$ in $\widehat{\cal C}_J$ for $J_{\widehat{\cal C}_J}^{can}$ by arrows $(u_i : l(c_i) \to F)_i$ such that $gu_i = l(g_i)$ with $g_i : c_i \to p(d)$. We pull back the covering $(C_p^*(u_i))_i$ along $\alpha$, which gives a covering $(v_i)_i$ of $E$:

\[\begin{tikzcd}
	{E_i} & E & {l(d)} \\
	{C_p^*l(c_i)} & {C_p^*F} & {C_p^*lp(d)}
	\arrow["{v_i}", from=1-1, to=1-2]
	\arrow["{\alpha_i}"', from=1-1, to=2-1]
	\arrow["\lrcorner"{anchor=center, pos=0.125}, draw=none, from=1-1, to=2-2]
	\arrow["f", from=1-2, to=1-3]
	\arrow["\alpha"', from=1-2, to=2-2]
	\arrow["\lrcorner"{anchor=center, pos=0.125}, draw=none, from=1-2, to=2-3]
	\arrow["{\eta_{\cal D}}", from=1-3, to=2-3]
	\arrow["{C_p^*(u_i)}"', from=2-1, to=2-2]
	\arrow["{C_p^*l(g_i)}"', bend right = 30, from=2-1, to=2-3]
	\arrow["{C_p^*(g)}"', from=2-2, to=2-3]
\end{tikzcd}\]

Now, by the fact that $p$ is a local fibration, we can locally lift the $g_i$ as locally cartesian arrows $\widehat{g_{ij}}$. This gives us the following commutative diagram:

\[\begin{tikzcd}
	{l(d_{ij})} & {E_i} & E & {l(d)} \\
	{C_p^*lp(d_{ij})} & {C_p^*l(c_i)} & {C_p^*F} & {C_p^*lp(d)}
	\arrow["{v_{ij}}", from=1-1, to=1-2]
	\arrow["{l(\widehat{g_{ij}})}", bend left = 30, from=1-1, to=1-4]
	\arrow["{\eta_{\cal D}(d_{ij})}"', from=1-1, to=2-1]
	\arrow["\lrcorner"{anchor=center, pos=0.125}, draw=none, from=1-1, to=2-2]
	\arrow["{v_i}", from=1-2, to=1-3]
	\arrow["{\alpha_i}"', from=1-2, to=2-2]
	\arrow["\lrcorner"{anchor=center, pos=0.125}, draw=none, from=1-2, to=2-3]
	\arrow["f", from=1-3, to=1-4]
	\arrow["\alpha"', from=1-3, to=2-3]
	\arrow["\lrcorner"{anchor=center, pos=0.125}, draw=none, from=1-3, to=2-4]
	\arrow["{\eta_{\cal D}(d)}", from=1-4, to=2-4]
	\arrow["{C_p^*l(u_{ij})}"', from=2-1, to=2-2]
	\arrow["{C_p^*lp(\widehat{g_{ij}})}"', bend right = 37, from=2-1, to=2-4]
	\arrow["{C_p^*(u_i)}"', from=2-2, to=2-3]
	\arrow["{C_p^*l(g_i)}"{description}, bend right = 28, from=2-2, to=2-4]
	\arrow["{C_p^*(g)}"', from=2-3, to=2-4]
\end{tikzcd}\]

\noindent where the pullback of the $\alpha_i$ along the $C_p^*l(u_{ij})$ are the $\eta_{\cal D}(d_{ij})$ because: the outer rectangle is a pullback, being the concatenation of three pullbacks, and hence the arrow on the left should also be the pullback of $\eta_{\cal D}(d)$ along $C_p^*lp(\widehat{g_{ij}})$, which is indeed $\eta_{\cal D}(d_{ij})$ by the very definition of $\widehat{g_{ij}}$ being locally cartesian.
\end{proof}

Since we are interested in recovering whether an arrow is cartesian using this kind of localization, we are naturally led to consider a sort of converse to the previous result: we would like to deduce that $(f,g)$ is cartesian from the fact that its precompositions with a covering, namely the $(g,f)(u_i,v_i) = \eta_{\mathcal{D}}(f_i')$, are cartesian. A first, weak version of this is:

\begin{prop}\label{precompkcart}
Let $p:({\cal D},K) \to ({\cal C},J)$ be a local fibration, $f : d \to d'$ and $(f_i : d_i \to d)_i$ arrows in $\cal D$ such that $p(f_i)_i$ is covering and $f\circ f_i$ is locally cartesian for every $i$. Then $f$ satisfies the local existence of liftings: let $h : p'(d'') \to p(d)$ in $\cal C$ and $g : d'' \to d'$ in $\cal D$ such that $p(f)h=p(g)$, then there exists a $K$-covering $(v_i : d''_i \to d'')$ and arrows $h_i : d''_i \to d$ such that $p(h_i)=hp(v_i)$ and $fh_i=gv_i$.
\end{prop}

\begin{proof}
We consider $g : d'' \to d'$ and $h : p(d'') \to p(d)$ such that $p(g) = p(f) \circ h$. Since the $p(f_i)$ are covering, we can pullback these arrows along $h$ and get another covering on $p(d'')$:

\[\begin{tikzcd}
	{c_{ij}} & {p(d'')} \\
	{p(d_i)} & {p(d)} & {p(d')}
	\arrow["{f'_{ij}}", from=1-1, to=1-2]
	\arrow["{h_{ij}}"', from=1-1, to=2-1]
	\arrow["h"', from=1-2, to=2-2]
	\arrow["{p(g)}", from=1-2, to=2-3]
	\arrow["{p(f_i)}"', from=2-1, to=2-2]
	\arrow["{p(f)}"', from=2-2, to=2-3]
\end{tikzcd}\]
Then, we locally lift the arrows $f_{ij}'$ into locally cartesian arrows $f_{ijk}$, and, as the $f_{ij}$ are covering, the multicompisition with their localizations is also a covering, which we can lift into a covering $(f_l)_l$ in $\cal D$, since $p$ is a comorphism:

\[\begin{tikzcd}
	{p(d_l)} \\
	& {p(d_{ijk})} \\
	&& {c_{ij}} & {p(d'')} \\
	&& {p(d_i)} & {p(d)} & {p(d')}
	\arrow["{a_l}"', from=1-1, to=2-2]
	\arrow["{p(f_l)}"{description}, bend left = 40, from=1-1, to=3-4]
	\arrow["{f'_{ijk}}"', from=2-2, to=3-3]
	\arrow["{p(f_{ijk})}"{description}, bend left =18, from=2-2, to=3-4]
	\arrow["{f'_{ij}}", from=3-3, to=3-4]
	\arrow["{h_{ij}}"', from=3-3, to=4-3]
	\arrow["h"', from=3-4, to=4-4]
	\arrow["{p(g)}", from=3-4, to=4-5]
	\arrow["{p(f_i)}"', from=4-3, to=4-4]
	\arrow["{p(f)}"', from=4-4, to=4-5]
\end{tikzcd}\]

Now we can apply the locally cartesian property of $p(f \circ f_i)$ to lift $h_{ij} \circ f'_{ijk} \circ a_l$, which gives us, by postcompositions with the $f_i$, local liftings of $h$.
\end{proof}

\begin{remark}
Notice that in general, we do not have the local uniqueness. For example, in the canonical stack of a topos $\widehat{\cal C}_J$ the arrow $(\nabla_c,1_c) : (c \coprod c, c, \nabla_c) \to (c,c,1_c)$ is not cartesian but it becomes cartesian when precomposed with the covering  $(i_1,1_c) : (c,c,1_c) \to (c \coprod c, c, \nabla_c)$, $(i_2,1_c) : (c,c,1_c) \to (c \coprod c, c, \nabla_c)$ (which give two identity morphisms). But we want to be able to recover locally cartesian arrows; not only arrows satisfying the local existence of a lifting, but also the local uniqueness. This hints that we need something stronger than an arrow that is locally cartesian on a covering: we need an arrow that is a colimit of locally cartesian ones.
\end{remark}

In order to have this colimit characterization, we need the following definition:

\begin{defn}\label{Dffacteta}
Let $p : ({\cal D},K) \to ({\cal C},J)$ be a local fibration, $d$ an object of $\cal D$, and a cartesian arrow $(f,g) : [\alpha] \to \eta_{\cal D}(d)$ in $(\widehat{\cal D}_K/C_p^*)$. We define the category ${\cal D}^{\mathbf{\eta-fact}}_{(f,g),d}$ as having for objects the triplets $(d',v',f')$ with $d'$ an object of $\cal D$, $(u',v') : \eta_{\cal D}(d') \to [\alpha]$ and $f' : d' \to d$ a locally cartesian arrow such that $(f,g)(u',v') = \eta_{\cal D}(f')$. We define as morphisms between two triplets $(d',(u',v'),f')$ and $(d'',(u'',v''),f'')$ the arrows $(x,y) : \eta_{\cal D}(d') \to \eta_{\cal D}(d'')  $ such that $(u'',v'')(x,y)=(u',v')$ and $\eta_{\cal D}(f'')(x,y)=\eta_{\cal D}(f')$. It comes with the obvious projection functors:

\[\begin{tikzcd}
	{(\widehat{\cal D}_K/C_p^*)/[\alpha]} & {(\widehat{\cal D}_K/C_p^*)} \\
	{{\cal D}_{(f,g),d}^{\mathbf{\eta-fact}}}
	\arrow["{\pi_{\alpha}}", from=1-1, to=1-2]
	\arrow["{p_{\alpha}^{\eta}}", from=2-1, to=1-1]
	\arrow["{p_{(f,g),d}^{\mathbf{\eta-fact}}}"', from=2-1, to=1-2]
\end{tikzcd}\]
\end{defn}

The following proposition shows us that every cartesian arrow $(f,g)$ in the canonical stack $(\widehat{\cal D}_K/C_p^*)$ is a colimit of cartesian arrows coming from the site:

\begin{prop}\label{caractcart} 
Let $p : ({\cal D},K) \to ({\cal C},J)$ be a local fibration, $d$ an object of $\cal D$, and a cartesian arrow $(f,g) : [\alpha] \to \eta_{\cal D}(d)$ in $(\widehat{\cal D}_K/C_p^*)$, as in the previous definition. The functor $p_{\alpha}^{\eta}$ is $J^{can}_{(\widehat{\cal D}_K/C_p^*)}$-cofinal, that is, $\colim p_{(f,g),d}^{\mathbf{\eta-fact}} \simeq [\alpha]$:

\[\begin{tikzcd}
	{\colim p_{(f,g),d}^{\mathbf{\eta-fact}}  \simeq [\alpha]} & {\eta_{\cal D}(d)} \\
	{p_{(f,g),d}^{\mathbf{\eta-fact}} (v',f',d') =\eta_{\cal D}(d')}
	\arrow["{(f,g)}", from=1-1, to=1-2]
	\arrow["{v'}", from=2-1, to=1-1]
	\arrow["{\eta_{\cal D}(f')}"', from=2-1, to=1-2]
\end{tikzcd}\]
\end{prop}

\begin{proof}
In this proof, we will denote by $v'$ any arrow of the form $(u',v') : \eta_{\mathcal{D}}(d') \to [\alpha]$ that appears as the second component of an object in ${\mathcal{D}}_{(f,g),d}^{\mathbf{\eta\text{-}fact}}$, in order to simplify the notation.

By \ref{localsurjcart} we know that the first condition of the cofinality conditions is satisfied: we have a covering constituted of objects coming from $p^{\eta}_{\alpha}$.

For the second condition, we take two objects $(d',v',f')$ and $(d'',v'',f'')$ of ${\mathcal{D}}_{(f,g),d}^{\mathbf{\eta\text{-}fact}}$, and since we have pullbacks in $(\widehat{\cal D}_K /C_p^*)$ we can restrict to check the cofinality condition for $[\beta] = (E,F,\beta : E \to C_p^*(F))$, with the two arrows $x' :[\beta] \to \eta_{\cal D}(d')$ and $x'':[\beta] \to \eta_{\cal D}(d'')$ giving the pullback:

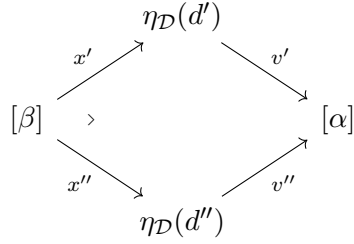
\[\begin{tikzcd}
	& {\eta_{\cal D}(d')} \\
	{[\beta]} && {[\alpha]} \\
	& {\eta_{\cal D}(d'')}
	\arrow["{v'}", from=1-2, to=2-3]
	\arrow["{x'}", from=2-1, to=1-2]
	\arrow["\lrcorner"{anchor=center, pos=0.125, rotate=45}, draw=none, from=2-1, to=2-3]
	\arrow["{x''}"', from=2-1, to=3-2]
	\arrow["{v''}"', from=3-2, to=2-3]
\end{tikzcd}\]

We know that $\eta_{\cal D}(f')$ is cartesian since $f'$ is locally cartesian. Moreover $(f,g)$ also is cartesian by hypothesis, so that by the two out of three property of cartesian arrows, we have that $v'$ (verifying $(f,g)v'=\eta_{\cal D}(f')$) also is \ref{enumloccart}. Since $x'$ is the pullback of $v'$, it also is cartesian, and similarly for $x''$. The situation is depicted in the following diagram:

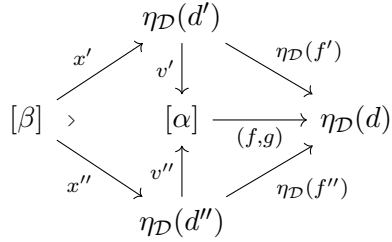
\[\begin{tikzcd}
	& {\eta_{\cal D}(d')} \\
	{[\beta]} & {[\alpha]} & {\eta_{\cal D}(d)} \\
	& {\eta_{\cal D}(d'')}
	\arrow["{v'}"', from=1-2, to=2-2]
	\arrow["{\eta_{\cal D}(f')}", from=1-2, to=2-3]
	\arrow["{x'}", from=2-1, to=1-2]
	\arrow["\lrcorner"{anchor=center, pos=0.125, rotate=45}, draw=none, from=2-1, to=2-2]
	\arrow["{x''}"', from=2-1, to=3-2]
	\arrow["{(f,g)}"', from=2-2, to=2-3]
	\arrow["{v''}", from=3-2, to=2-2]
	\arrow["{\eta_{\cal D}(f'')}"', from=3-2, to=2-3]
\end{tikzcd}\]

Now, since $x'$ and $\eta_{\cal D}(f')$ are cartesian, their composition also is. The codomain of their compisition being $\eta_{\cal D}(d)$, we can use the proposition \ref{localsurjcart} to locally lift $\eta_{\cal D}(f')v'x'$ as cartesian arrows $\eta_{\cal D}(\widehat{f_i})$ coming from $\cal D$:

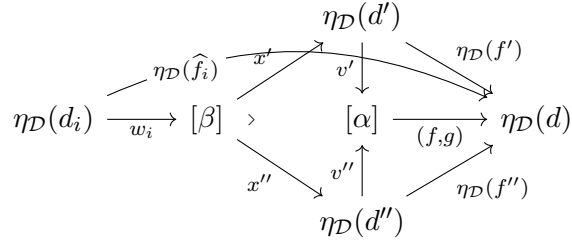
\[\begin{tikzcd}
	&& {\eta_{\cal D}(d')} \\
	{\eta_{\cal D}(d_i)} & {[\beta]} & {[\alpha]} & {\eta_{\cal D}(d)} \\
	&& {\eta_{\cal D}(d'')}
	\arrow["{v'}"', from=1-3, to=2-3]
	\arrow["{\eta_{\cal D}(f')}", from=1-3, to=2-4]
	\arrow["{w_i}"', from=2-1, to=2-2]
	\arrow["{\eta_{\cal D}(\widehat{f_i})}"{description, pos=0.2}, bend left = 24, from=2-1, to=2-4]
	\arrow["{x'}", from=2-2, to=1-3]
	\arrow["\lrcorner"{anchor=center, pos=0.125, rotate=45}, draw=none, from=2-2, to=2-3]
	\arrow["{x''}"', from=2-2, to=3-3]
	\arrow["{(f,g)}"', from=2-3, to=2-4]
	\arrow["{v''}", from=3-3, to=2-3]
	\arrow["{\eta_{\cal D}(f'')}"', from=3-3, to=2-4]
\end{tikzcd}\]

Repeating this argument for the lower path of the diagram, we obtain, at the cost of one more localization, intermediate objects $(d_i,v_i,\widehat{f_i})$ connecting $(d',v',f')$ and $(d'',v'',f'')$: this completes the cofinality condition.
\end{proof}

Note that in the definition of ${\cal D}^{\mathbf{\eta-fact}}_{(f,g),d}$, one does not require  $(f,g) : [\alpha] \to \eta_{\cal D}(d)$ to be a cartesian arrow. We showed that $(f,g)$ being cartesian implies the cofinality condition. Conversely, the following proposition says that, for an arbitrary arrow $(f,g)$ in $(\widehat{\cal D}_K /C_p^*)$, the cofinality condition implies that it is cartesian:

\begin{prop}\label{caractcartconverse}
Let $p : ({\cal D},K) \to ({\cal C},J)$ be a local fibration, $d$ an object of $\cal D$, and an arrow $(f,g) : [\alpha] \to \eta_{\cal D}(d)$ in $(\widehat{\cal D}_K/C_p^*)$. If the functor $p_{\alpha}^{\eta}$ is $J^{can}_{(\widehat{\cal D}_K/C_p^*)}$-cofinal, then $(f,g)$ is cartesian.
\end{prop}

\begin{proof} 
The functor $p_{\alpha}^{\eta}$ being $J^{can}_{(\widehat{\cal D}_K/C_p^*)}$-cofinal gives us that $[\alpha]$ is colimit of the diagram:

\[\begin{tikzcd}
	{\eta_{\cal D}(d_i)} \\
	{[\alpha]} & {\eta_{\cal D}(d)}
	\arrow["{v_i}"', from=1-1, to=2-1]
	\arrow["{\eta_{\cal D}(\widehat{f_i})}", from=1-1, to=2-2]
	\arrow["{(f,g)}"', from=2-1, to=2-2]
\end{tikzcd}\]

\noindent where we take all the $(d_i,v_i,\widehat{f_i})$, with $\widehat{f_i}$ locally cartesian. Hence we are in a situation of this form:

\[\begin{tikzcd}
	& {A_i} && {B_i} \\
	{\colim A_i} && {\colim B_i} \\
	A && B
	\arrow["{\eta_{\cal D}(d')}", from=1-2, to=1-4]
	\arrow[from=1-2, to=2-1]
	\arrow[from=1-2, to=3-1]
	\arrow["\lrcorner"{anchor=center, pos=0.125}, draw=none, from=1-2, to=3-3]
	\arrow[from=1-4, to=2-3]
	\arrow["{\eta_{\cal D}(\widehat{f_i})}", from=1-4, to=3-3]
	\arrow["\alpha"{description}, from=2-1, to=2-3]
	\arrow[from=2-1, to=3-1]
	\arrow[from=2-3, to=3-3]
	\arrow["{\eta_{\cal D}(d)}"', from=3-1, to=3-3]
\end{tikzcd}\]

The arrows $\eta_{\cal D}(\widehat{f_i})$ being cartesian, the $A_i$ are the pullbacks of the $B_i$ along $\eta_{\cal D}(d)$. So, by stability of colimits in a topos, we have that $\colim A_i$ is the pullback of $\colim B_i$ along $\eta_{\cal D}(d)$. This gives us that the front square is a pullback, that is, $(f,g)$ is cartesian.
\end{proof}

Even if we are not working at the topos-level, we are still able to define a category analogue to ${\cal D}^{\mathbf{\eta-fact}}_{(f,g),d}$ at the level of sites:

\begin{defn}
Let $p : ({\cal D},K) \to ({\cal C},J)$ be a local fibration, and $f : d' \to d$ an arrow in $\cal D$. We define ${\cal D}^{\mathbf{cart-fact}}_f$ to be the category having for objects the $(d'',v'',f'')$ where $d''$ is an object of $\cal D$, $f'' : d'' \to d$ is a locally cartesian arrow and $v'' : d'' \to d'$ is such that $fv''=f''$. The arrows between two triplets $(d'',v'',f'')$ and $(\overline{d''},\overline{v''},\overline{f''})$ are the arrows $u : d'' \to \overline{d''}$ such that $v'' = \overline{v''}u$ and $f'' = \overline{f''}u$. This category comes with the obvious projection functors depicted here:

\[\begin{tikzcd}
	{{\cal C}/p(d)} && {{\cal C}} \\
	{{\cal D}^{\mathbf{cart-fact}}_f} & {\cal D}
	\arrow["{\pi_{p(d)}}", from=1-1, to=1-3]
	\arrow["{\pi_f^{\cal C}}", from=2-1, to=1-1]
	\arrow["{\pi_f^{\mathbf{cart-fact}}}"', from=2-1, to=2-2]
	\arrow["p"', from=2-2, to=1-3]
\end{tikzcd}\]
\end{defn}

In the light of this definition, the previous result \ref{caractcart} about the colimit characterization of cartesian arrow holding for canonical stacks have a counterpart at the site-level, for local fibrations:

\begin{prop}\label{cartcarasite}
Let $p : ({\cal D},K) \to ({\cal C},J)$ be a local fibration, and $f : d' \to d$ an arrow in $\cal D$. We have the equivalence:

\begin{enumerate}[(i)]
    \item The arrow $f$ is locally cartesian.
    \item The functor $\pi_f^{\cal C}$ is $J$-cofinal.
    \item There exists a category ${{\cal D}^{\mathbf{cart-fact}}_f}'$ with a functor $\chi_{f} : {{\cal D}^{\mathbf{cart-fact}}_f}' \to {\cal D}^{\mathbf{cart-fact}}_f$:

    \[\begin{tikzcd}
	{{\cal C}/p(d)} \\
	{{\cal D}^{\mathbf{cart-fact}}_f} & {{\cal C}} \\
	{{{\cal D}^{\mathbf{cart-fact}}_f}'}
	\arrow["{\pi_{p(d)}}", from=1-1, to=2-2]
	\arrow["{\pi_f^{\cal C}}", from=2-1, to=1-1]
	\arrow[from=2-1, to=2-2]
	\arrow["{\chi_f}", from=3-1, to=2-1]
	\arrow["{\pi_{p(d)}\pi_f^{\cal C}\chi_f}"', from=3-1, to=2-2]
    \end{tikzcd}\]
    
    such that $\pi^{\cal C}_f \chi_f$ is $J$-cofinal.

    \item We have a colimit representation at the topos-level:
    
    $\colim_{{\cal D}^{\mathbf{cart-fact}}_f}lp\pi^{\mathbf{cart-fact}}_f \simeq lp(d)$.
\end{enumerate}
\end{prop}

\begin{proof}
The points (iv) and (ii) are equivalent in the light of the definition of cofinality, and the fact that $\colim_{{\cal C}/p(d)}l\pi_{p(d)} \simeq lp(d)$.

The fact that (i) $\Rightarrow$ (ii) is immediate to check, since the object $(d',1_{d'},f)$ is in the category ${\cal D}^{\mathbf{cart-fact}}_f$.

For the fact that (iii) or (ii) imply (i), with the same argument as in \ref{caractcartconverse}, taking $\eta_{\cal D}(f)$ as $(f,g) : [\alpha] \to \eta_{\cal D}(d)$, we obtain that $\eta_{\cal D}(f)$ is cartesian, that is, $f$ is locally cartesian.
\end{proof}

\subsection{Relative Diaconescu's theorem for local fibrations}

With these two successive localizations of cartesian arrows in a canonical stack, we are now able to prove the following important theorem:

\begin{thm}\label{loccartesianmorphsites}
Let $p : ({\cal D},K) \to ({\cal C},J)$ be a local fibration, $p' : ({\cal D'},K') \to ({\cal C},J)$ be a $(J,K')$-fibration, and $A: ({\cal D}, K) \to {(D', K')}$ a morphism of sites such that $p' A \simeq p$. Then, the following are equivalent:

\begin{enumerate}[(i)]
    \item The geometric morphism induced by $A$ is a morphism of relative toposes: $\Sh(A) : [C_p'] \to [C_p]$.
    \item The $\eta$-extension $\widetilde{A}^*$ is a morphism of fibrations. 
    \item The functor $A$ is a morphism of local fibrations. 
\end{enumerate}
\end{thm}

\begin{proof}
Recall that, as stated in Proposition 3.4. (iii) \cite{bartolicaramello}, $A$ induces a morphism of relative toposes if and only if $\widetilde{A}^*$ is a morphism of fibrations. 

Suppose that $\widetilde{A}^*$ is a morphism of fibrations, since $p'A \simeq p$, we have $\widetilde{A}^*\eta_{\cal D} \simeq \eta_{\cal D'} A$ (Proposition 3.4. (iv) \cite{bartolicaramello}). Let $f : d \to d'$ be a locally cartesian arrow, by definition $\eta_{\cal D} (f) $ is cartesian, and thus $\widetilde{A}^*\eta_{\cal D}(f) \simeq \eta_{\cal D'} A(f)$ also is: hence, by definition, $A(f)$ is a locally cartesian arrow, that is $A$ is a morphism of local fibrations.

Conversely, let us assume that $A$ is a morphism of local fibrations. We will show that a cartesian arrow $(f,g) : [\alpha] \to [\alpha']$ in $(\widehat{\cal D}_K / C_p^*l_J)$ is sent to a cartesian one by $\widetilde{A}^*$. For this, we will use our \textquotedblleft local \textquotedblright representation of the arrow $(f,g)$ in terms of cartesian arrows coming from $\eta_{\cal D}$, which we know to be preserved by $\widetilde{A}^*$. We recall how to proceed: first, by density of $\eta_{\cal D}$, we can cover $[\alpha']$ with objects $\eta_{\cal D}(d_i)$ coming from $\cal D$, and pull back the cartesian arrow $(f,g)$ along this covering, giving us cartesian arrows $(f_i,g_i)$ (cf. \ref{pullbackcart}):

\[\begin{tikzcd}
	{[\alpha]} & {[\alpha]'} \\
	{P_i} & {\eta_{\cal D}(d_i)}
	\arrow["{(f,g)}", from=1-1, to=1-2]
	\arrow["{u_i'}", from=2-1, to=1-1]
	\arrow["\lrcorner"{anchor=center, pos=0.125, rotate=90}, draw=none, from=2-1, to=1-2]
	\arrow["{(f_i,g_i)}"', from=2-1, to=2-2]
	\arrow["{u_i}"', from=2-2, to=1-2]
\end{tikzcd}\]

Let us assume that the $\widetilde{A}^*(f_i,g_i)$ are cartesian: since $\widetilde{A}^*$ is a morphism of sites it commutes with pullbacks and, moreover, it preserves coverings. Combining these two facts we have that $\widetilde{A}^*(f,g)$ is $J_{C_p}$-cartesian; but $J_{C_p}$-cartesian arrows are cartesian ones by \ref{Jfcartimpliescart}, so that $\widetilde{A}^*(f,g)$ is cartesian. Hence, we are reduced to show that $\widetilde{A}^*$ preserves these cartesian arrows $(f_i,g_i)$ having for codomain an object of the form $\eta_{\cal D}(d)$.

So let $(f',g') : P \to \eta_{\cal D}(d)$ be a cartesian arrow. Using \ref{caractcart} we know that we have a presentation of the domain $P$ of such an arrow as a colimit:

\[\begin{tikzcd}
	{\colim \eta_{\cal D}(d_{i})} & P & {\eta_{\cal D}(d)} \\
	& {\eta_{\cal D}(d_{i})}
	\arrow["\simeq"{marking, allow upside down}, draw=none, from=1-1, to=1-2]
	\arrow["{(f',g')}", from=1-2, to=1-3]
	\arrow["{v_{i}}", from=2-2, to=1-2]
	\arrow["{\eta_{\cal D}(\widehat{f_{i}})}"', from=2-2, to=1-3]
\end{tikzcd}\]

\noindent indexed by the category ${\cal D}_{(f',g'),d'}^{\mathbf{\eta-fact}}$ having for objects the $(d_{i},v_{i},\widehat{f_{i}})$ with $\widehat{f_{i}}$ being locally cartesian arrows. 

Now, since $\widetilde{A}^*$ preserves colimits and $\widetilde{A}^*\eta_{\cal D} \simeq \eta_{\cal D'}A$, we have the following diagram:

\[\begin{tikzcd}
	{\colim \eta_{\cal D'}(A(d_{i}))} & {\widetilde{A}^*(P)} & {\eta_{\cal D'}(A(d))} \\
	& {\eta_{\cal D'}(A(d_{i}))}
	\arrow["\simeq"{marking, allow upside down}, draw=none, from=1-1, to=1-2]
	\arrow["{\widetilde{A}^*(f',g')}", from=1-2, to=1-3]
	\arrow["{\widetilde{A}^*(v_{i})}", from=2-2, to=1-2]
	\arrow["{\eta_{\cal D'}(A(\widehat{f_{i}}))}"', from=2-2, to=1-3]
\end{tikzcd}\]

The $\eta_{\cal D'}A(\widehat{f_i})$ are cartesian arrows since $A$ preserves locally cartesian ones. Thus, we are in the situation of (iii) \ref{cartcarasite}. This gives us that $(f',g')$ is $J_{C_p}$-cartesian, i.e. it is cartesian. 

Hence, $\widetilde{A}^*$ preserves cartesian arrows and $A$ induces a morphism of relative toposes.
\end{proof}

This very general result naturally leads to a relative Diaconescu's theorem for local fibrations. Here is the definition of the good notion of flat functors for local fibrations:

\begin{defn}
Let $p : ({\cal D},K) \to ({\cal C},J)$ be a local fibration and $f : {\cal E} \to \widehat{\cal C}_J$ a relative topos. The category $\mathbf{Flat}_{\widehat{\cal C}_J}^{K,J}(({\cal D},K),({\cal E}/f^*l_J))$ is defined to be the category of continuous morphisms of local fibrations $A : ({\cal D},K) \to ({\cal E} / f^*l)$ such that $\widetilde{A}^*$ is a morphism of fibrations preserving finite limits fiberwise. They are called the  \emph{K-flat relative to} $\widehat{\cal C}_J$ functors.
\end{defn}

And here is the analogue of Diaconescu's theorem for local fibrations:

\begin{thm}\label{diaclocfib}
Let $p : ({\cal D},K) \to ({\cal C},J)$ be a local fibration and $f : {\cal E} \to \widehat{\cal C}_J$ a relative topos. The category $\mathbf{Flat}_{\widehat{\cal C}_J}^{K,J}(({\cal D},K),({\cal E}/f^*l_J))$ is equivalent to the category of morphisms of local fibrations which also are morphisms of sites. Thus, we have an equivalence:
\begin{center}
    $\mathbf{Flat}_{\widehat{\cal C}_J}^{K,J}(({\cal D},K),({\cal E} / f^*l_J)) \simeq \mathbf{Geom}_{\widehat{\cal C}_J}([f],[C_p])$
\end{center}

\noindent between those morphisms of relative toposes $a : [f] \to [C_p]$ and the morphisms of local fibrations which also are morphisms of sites $A : ({\cal D},K) \to ({\cal E} / f^*l_J)$.
\end{thm}

\begin{proof}
If $a : [f] \to [C_p]$ is a morphism of relative toposes, then $\widetilde{a^*}^*$ is a morphism of sites, but also of local fibrations, as it is a morphism of fibrations and that the local cartesian arrows of a canonical relative sites are the cartesian ones \ref{canonicalcartesian}. But $\eta_{\cal D}$ also is a morphism of local fibrations and a morphism of sites, hence the composite $\widetilde{a^*}^*\eta_{\cal D}$ is a K-flat relative to $\widehat{\cal C}_J$ morphisms. 

In the other direction, if $A : ({\cal D},K) \to (({\cal E} / f^*l_J),J_f)$ is a morphism of local fibrations and a morphism of sites, as a particular case of \ref{loccartesianmorphsites} we see that it induces a relative geometric morphism.

The fact that these two constructions are pseudoinverse is easily deducible from the density of the eta functors and the projections $\pi_{\cal D} : (\widehat{\cal D}_K/C_p^*) \to \widehat{\cal D}_K$, as fully detailed in the proof for fibrations (Theorem 3.15. \cite{bartolicaramello}). 
\end{proof}

\subsection{Fibrations seen as local fibrations}

As fibrations are special cases of local fibrations, it is natural to consider how results concerning local fibrations apply in this setting. We have seen in \ref{loccartesianmorphsites} that a morphism of sites between two local fibrations induces a morphism between their associated relative toposes if and only if it is a morphism of local fibrations. It is therefore useful to dispose of criteria for the preservation of locally cartesian arrows by morphisms of sites $A : (\mathcal{G}(\mathbb D),K) \to (\mathcal{G}(\mathbb D'),K')$ between two relative sites over some base site $({\cal C},J)$ such that $p'A \simeq p$. 

The following proposition shows that, given a functor $A$ between two fibrations, there are canonical comparison vertical arrows which measure how far $A$ is from being a morphism of fibrations:

\begin{prop}\label{defectmorphfib}
Let $p:\mathcal{G}(\mathbb D) \to {\cal C}$ and $p':\mathcal{G}({\mathbb D}') \to {\cal C}$ two fibrations over the same base category $\cal C$, and $A: \mathcal{G}(\mathbb D) \to \mathcal{G}(\mathbb D')$ a functor between them such that $p'A \simeq p$. 

For any arrow $f : c' \to c$ in $\cal C$, we have a canonical natural transformation 
\[
v_f : A_{c'}\mathbb D(f) \Rightarrow \mathbb D'(f)A_c
\]
as pictured in the diagram

\[\begin{tikzcd}
	{\mathbb D(c)} & {\mathbb D(c')} \\
	{\mathbb D'(c)} & {\mathbb D'(c')}
	\arrow["{\mathbb D(f)}", from=1-1, to=1-2]
	\arrow["{A_c}"', from=1-1, to=2-1]
	\arrow["{v_f}"{description}, Rightarrow, from=1-2, to=2-1]
	\arrow["{A_{c'}}", from=1-2, to=2-2]
	\arrow["{\mathbb D'(f)}"', from=2-1, to=2-2]
\end{tikzcd}\]

such that, for any object $x$ of ${\mathbb D}(c)$, the arrow
\[
v_f^x : A_{c'}(D(f)(x)) \Rightarrow \mathbb D'(f)(A_c(x))
\]
is the vertical part of $A(f,1_{{\mathbb D}(f)(x)}) : (A_{c'}(\mathbb D(f)(x)),c') \to (A_c(x),c) $.
\end{prop}

\begin{proof}
For every $x$ in $\mathbb D(c)$, we can factorize $A(f,1_{{\mathbb D}(f)(x)}) : (A_{c'}(\mathbb D(f)(x)),c') \to (A_c(x),c)$ as a vertical and a cartesian arrow; and, since  $p'A \simeq p$, we have that $p'A(f,1_{\mathbb D(f)(x)}) \simeq p(f,1_{\mathbb D(f)(x)})$ so that the cartesian part is given by $f$. We represent this factorization in the following diagram:

\[\begin{tikzcd}
	{(A_{c'}\mathbb D(f)(x),c')} & {(A_c(x),c)} \\
	{(\mathbb D'(f)A_{c}(x),c')}
	\arrow["{A(f,1_{\mathbb D(f)(x)})}", shift left, from=1-1, to=1-2]
	\arrow["{(1_{c'},v_f^x)}"', from=1-1, to=2-1]
	\arrow["{(f,1_{\mathbb D'(f)A_c(x)})}"'{pos=0.8}, from=2-1, to=1-2]
\end{tikzcd}\]

This vertical part gives us a vertical arrow for each object $x$ of $\mathbb D(c)$, as in the proposition $v_f^x : A_{c'}(\mathbb D(f)(x)) \to \mathbb D'(f)A_c(x)$. 

For the naturality in $x$, let $u : x \to x'$ be an arrow in $\mathbb D(c)$. We have the diagram bellow: 
\[\begin{tikzcd}
	& {(A_c(x'),c)} \\
	{(A_{c'}\mathbb D(f)(x'),c')} && {(\mathbb D'(f)A_{c}(x'),c')} \\
	{(A_{c'}\mathbb D(f)(x),c')} && {(\mathbb D'(f)A_{c}(x),c')} \\
	& {(A_c(x),c)}
	\arrow["{A(f,1_{\mathbb D(f)(x')})}", from=2-1, to=1-2]
	\arrow["{v_f^{x'}}"{pos=0.3}, from=2-1, to=2-3]
	\arrow["{(f,1_{\mathbb D'(f)A_c(x')})}"', from=2-3, to=1-2]
	\arrow["{(1_{c'},A_{c'}\mathbb D(f)(u))}", from=3-1, to=2-1]
	\arrow["{v_f^x}"'{pos=0.3}, from=3-1, to=3-3]
	\arrow["{A(f,1_{\mathbb D(f)(x)})}"', from=3-1, to=4-2]
	\arrow["{(1_{c'},\mathbb D'(f)A_{c}(u))}"', from=3-3, to=2-3]
	\arrow["{(f,1_{\mathbb D'(f)A_c(x)})}", from=3-3, to=4-2]
	\arrow["{(1_{c},A_c(u))}"{description}, from=4-2, to=1-2]
\end{tikzcd}\]

In order to check the naturality, we want to see if $v_f^{x'} \circ (1_{c'},A_{c'}\mathbb D(f)(u)) = (1_{c'},\mathbb D'(f)A_{c}(u))\circ v_f^x$, and, since $(f,1_{\mathbb D'(f)A_c(x')})$ is cartesian, we are reduced to show that these two arrows coincide when postcomposed by it, as well as their projections in $\cal C$. Their projections in $\cal C$ are equal since they only have a vertical non trivial part. For the postcomposition with $(f,1_{\mathbb D'(f)A_c(x')})$ we have: $ (f,1_{\mathbb D'(f)A_c(x')}) \circ v_f^{x'} \circ (1_{c'},A_{c'}\mathbb D(f)(u)) = A(f,1_{\mathbb D(f)(x')}) \circ (1_{c'},A_{c'}\mathbb D(f)(u))$ because of the commutation of the upper triangle by definition of $v_f^{x'}$. Then, $A(f,1_{\mathbb D(f)(x')}) \circ (1_{c'},A_{c'}\mathbb D(f)(u)) = (1_{c},A_c(u)) \circ A(f,1_{\mathbb D(f)(x)})$ because of the functoriality of $A$ together with the identity $(1,u)\circ (f,1) = (f,1)\circ(1,\mathbb D(f)(u))$ holding in any fibration. Then, we have $(1_c,A_c(u)) \circ A(f,1_{\mathbb D(f)(x)}) = (1_c,A_c(u)) \circ (f,1_{\mathbb D'(f)A_c(x)}) \circ v_f^x $ because of the commutation of the lower triangle (definition of $v_f^x$). Finally, since $(1_{c},A_c(u)) \circ (f,1_{\mathbb D'(f)A_c(x)}) =    (f,1_{\mathbb D'(f)A_c(x')}) \circ (1_{c'},\mathbb D'(f)A_{c}(u)) $, we conclude $(1_{c},A_c(u)) \circ (f,1_{\mathbb D'(f)A_c(x')}) \circ v_f^x = (f,1_{\mathbb D'(f)A_c(x')}) \circ (1_{c'},\mathbb D'(f)A_{c}(u)) \circ v_f^x$, so that the two postcompositions with $(f,1_{\mathbb D'(f)A_c(x')})$ are equal.
\end{proof}

It is interesting to notice that these canonical comparison vertical arrows, defined for different arrows in the base category, come with some compatibility conditions:

\begin{prop}\label{canonicalvertical}
In the same setting as the previous proposition, for two arrows $g: c' \to c$ and $f : c'' \to c'$ in the base category $\cal C$, we have the following compatibility property: $v_{gf}^x = \mathbb D'(f)(v^x_g) \circ v^f_{\mathbb D(g)(x)}$
\end{prop}

\begin{proof}
Making a slight abuse of notation, we identify $v_x^f$ with its vertical part. We have:

$(gf,1_{\mathbb D'(gf)(x)}) \circ (1_{c''},\mathbb{D}'(f)(v^x_g)) \circ (1_{c''},v_f^{\mathbb{D}(g)(x)})$

$ =(g,1_{\mathbb D(g)(x)}) \circ (1_{c'},v^x_g) \circ (f,1_{\mathbb D(gf)(x)}) \circ (1_{c''},v_f^{\mathbb{D}(g)(x)})$

$= A(gf,1_{\mathbb D(gf)(x)})$

\noindent where the first step is because of the composition law in a fibration, and the second one because of the definition of $v^x_g$ and $v_f^{\mathbb{D}(g)(x)}$ together with the functoriality of $A$.

Hence, these two vertical arrows are equal, by virtue of being vertical and having equal postcompositions with $(gf,1_{\mathbb D'(gf)(x)})$, a cartesian arrow.
\end{proof}

Now, using the explicit characterization of locally cartesian arrows in a relative site given in \ref{kcartesianforfib}, we can provide necessary and sufficient conditions for a functor to preserve them in the light of these canonical vertical arrows; this is the first point of the following proposition. This characterization is then simplified in the case where $A$ is continuous: in this case, $A$ preserves locally cartesian arrows if and only if its canonical comparison arrows (as in Proposition \ref{defectmorphfib}) are sent to arrows which are, locally, isomorphisms. In other words, the difference between being a morphism of fibrations or only a morphism of local fibrations is invisible at the topos level.

\begin{prop}
Let $A: (\mathcal{G}(\mathbb D),K) \to (\mathcal{G}(\mathbb D'),K')$ be a functor between two relative sites such that $p'A \simeq p$. It is a morphism of local fibrations if and only if the two following conditions are satisfied: 

\begin{enumerate}
    \item For every vertical arrow $(1,u)$ in $\mathcal{G}(\mathbb D)$ such that $l_K(1,u)$ is an iso, we have that $l_{K'}(A(1,u))$ is an iso.
    \item For all arrows $f : c' \to c$ in the basis, all the components $v^f_x$ of the canonical natural transformation $A_{c'}\mathbb D(f) \to \mathbb D'(f)A_c$ are such that $l_{K'}(v^f_x)$ is iso.
\end{enumerate}

The site-description of the condition for such vertical arrows to be isomorphisms at the topos-level is explicited in \ref{kcartesianforfib}.

If moreover $A$ is continuous, the first condition is immediately satisfied, thus $A$ being a morphism of local fibrations reduces to check the second condition.
\end{prop}

\begin{proof}
As stated in \ref{kcartesianforfib}, locally cartesian arrows in $(\mathcal{G}(\mathbb D),K)$ are those $(f,u)$ such that $l(1,u)$ is iso. For an arrow $(f,u) : (x',c') \to (x,c)$, the vertical part of $A(f,u)$ being $v_x^fA(1,u)$, our thesis follows.

If $A$ is continuous, $\Sh(A)^*l(1,u) = l(A(1,u))$, hence the first condition is automatic. 
\end{proof}

This characterization gives us that any morphism of fibrations which is continuous also is a morphism of local fibrations:

\begin{prop}\label{morphfibimpkfib}
Let $A: (\mathcal{G}(\mathbb D),K) \to (\mathcal{G}(\mathbb D'),K')$ be a morphism of fibrations and a continuous functor. Then, $A$ is a morphism of local fibrations.
\end{prop}

\begin{proof}
Since $A$ is continuous the previous proposition gives us that there is just a need for it to satisfy the second point. But $A$ is a morphism of fibrations if and only if all the $v^f_x$ are iso, so that they are automatically sent to  isos by $l_{K'}$.
\end{proof}

Then, since any morphism of sites and local fibrations induces a morphism of relative toposes, we have that:

\begin{cor}
Let $A: (\mathcal{G}(\mathbb D),K) \to (\mathcal{G}(\mathbb D'),K')$ be a morphism of relative sites. We have that $\Sh(A)$ is a morphism of relative toposes.
\end{cor}

\begin{proof}
As stated in \ref{morphfibimpkfib}, every morphism of sites which is also a morphism of fibrations is a morphism of local fibrations. Hence, by \ref{loccartesianmorphsites} it indeed induces a morphism of relative toposes.
\end{proof}

\begin{remark}
One can observe that, for a relative site $p : (\cal G(\mathbb D),K) \to (\cal C,J)$ and a relative topos $f : \cal E \to \widehat{\cal C}_J$, a continuous functor $A : (\cal G(\mathbb D),K) \to ((\cal E/f^*l_J),J_f)$ such that $\pi_f A \simeq p$ is a morphism of fibrations if and only if it is a morphism of local fibrations. Indeed, by \ref{morphfibimpkfib}, every morphism of fibrations is in particular a morphism of local fibrations; conversely, if $A$ is a morphism of local fibrations, then it sends cartesian arrows (which are in particular locally cartesian) to locally cartesian arrows of $((\cal E/f^*l_J),J_f)$, but these are cartesian arrows, hence it is a morphism of fibrations. This allows the relative Diaconescu theorem for local fibrations (\ref{diaclocfib}) to specialize to the one for fibrations (\cite{bartolicaramello}, Theorem 3.15).
\end{remark}

\section{Being a fibration at the topos-level}

Up to now, we have only been working under the assumption that $p : ({\cal D},K) \to ({\cal C},J)$ is a comorphism of sites. If, in addition, $p$ is continuous, it admits an extension $\Sh(p)^* : \widehat{\cal D}_K \to \widehat{\cal C}_J$ of $p$ between the corresponding sheaf toposes. In this section, we aim to compare the conditions under which $p$ is a local fibration with those under which $\Sh(p)^*$ is a fibration. To do so, we will first reformulate the notion of cartesian arrow and the existence of cartesian liftings in terms of a left adjoint. Then, this new formulation will allow to draw a link between contiuous comorphisms of sites infucing fibrations at the topos-level, and locally connected morphisms. Finally, we give a general characterization of those continuous comorphisms inducing fibrations at the topos-level.

\subsection{Fibrations and left adjoint projection functors}\label{section5}

Recall that a continuous comorphism of sites $p : (\cal D,K) \to (\cal C,J)$ induces a pair of adjoint functors \( C_p^* \dashv \Sh(p)^* \), where \( C_p^* \) is the inverse image functor and \( \Sh(p)^* \) is its exceptional left adjoint. Since the functor $\Sh(p)^* : \widehat{\cal D}_K \to \widehat{\cal C}_J$ constitutes an extension to the toposes of the functor $p : (\cal D,K) \to (\cal C,J)$, as pictured in the following commutative square:

\[\begin{tikzcd}
	{\widehat{\cal C}_J} & {\widehat{\cal D}_K} \\
	{(\cal C,J)} & {(\cal D,K)}
	\arrow["{\Sh(p)^*}"', from=1-2, to=1-1]
	\arrow["{l_J}", from=2-1, to=1-1]
	\arrow["\simeq"{description}, draw=none, from=2-1, to=1-2]
	\arrow["{l_K}"', from=2-2, to=1-2]
	\arrow["p", from=2-2, to=2-1]
\end{tikzcd}\]

\noindent it is very natural to wonder when it will be a fibration. 

Now, if we want to study the fibration-theoretic aspects of \( \Sh(p)^* \), we should first consider the general and abstract case where we just have a pair of adjoint functors $p_* \vdash p^*$. In this subsection, we characterize the cartesian arrows and the existence of cartesian liftings for a projection functor which is a left adjoint.

\begin{prop}\label{cartesianwithunit}
Let $p^* : {\cal D} \to {\cal C}$ be a functor having a right adjoint $p_* \vdash p^*$ with its unit $\eta$. An arrow $f : d' \to d$ in $\cal D$ is cartesian for $p^*$ if and only if the following square is a pullback:

\[\begin{tikzcd}
	{d'} & d \\
	{p_*p^*(d')} & {p_*p^*(d)}
	\arrow["{\eta_{d'}}"', from=1-1, to=2-1]
	\arrow["{\eta_{d}}", from=1-2, to=2-2]
	\arrow["f", from=1-1, to=1-2]
	\arrow["{p_*p^*(f)}"', from=2-1, to=2-2]
\end{tikzcd}\]
\end{prop}

\begin{proof}
The square is a pullback exactly when there is a bijection between arrows $h : d'' \to d'$ and pairs of arrows $(g : d'' \to d, h' : d'' \to p_*p^*(d'))$ such that $ \eta_{d} g =  p_*p^*(f) h'$, that is, by transposition along the adjunction: $p^*(g) = p^*(f)^th'$, which is exactly the unique lifting property of a cartesian arrow $f$.
\end{proof}

Because of the definition of the canonical functor $\eta_{\cal D}$ in the case of a continuous comorphism of sites (Remarks 2.3. (d) \cite{bartolicaramello}), the previous proposition gives us a characterization of locally cartesian arrows in this case:

\begin{prop}
Let $p : (\mathcal{C},J) \to (\mathcal{D},K)$ be a continuous comorphism of sites. Then an arrow $f : d \to d'$ is locally cartesian if and only if $l_K(f)$ is cartesian for $\Sh(p)^* : \Sh(\mathcal{D},K) \to \Sh(\mathcal{C},J)$.
\end{prop}

\begin{proof}

Recall that, when $p$ is continuous, $\eta_{\cal D}$ is given at an object $d$ of $\cal D$ by the unit of the weak morphism of toposes (adjunction) $\Sh(p)_* \vdash \Sh(p)^*$ at the object $l_K(d)$. The precedent proposition says that $l_K(f)$ is cartesian exactly when $\eta_{\cal D}(f)$ is represented by a pullback, that is, $f$ is locally cartesian by definition.  
\end{proof}

Inspired by the previous reformulation of cartesian arrows for a left adjoint, we have the same kind of characterization for a left adjoint to be a fibration:

\begin{prop}\label{liftingwithunit}
Let $p^* : {\cal D} \to {\cal C}$ be a functor having a right adjoint $p_* \vdash p^*$ with unit $\eta$. An arrow $f: c \to p^*(d)$ admits a cartesian lifting if and only if the following pullback exists and the arrow $q^t : p^*(F') \to c$ is an isomorphism:

\[\begin{tikzcd}
	{F'} & F \\
	{p_*(c)} & {p^*p_*(d)}
	\arrow["{\hat{f}}", from=1-1, to=1-2]
	\arrow["q"', from=1-1, to=2-1]
	\arrow["\lrcorner"{anchor=center, pos=0.125}, draw=none, from=1-1, to=2-2]
	\arrow["{\eta(F)}", from=1-2, to=2-2]
	\arrow["{p_*(f)}"', from=2-1, to=2-2]
\end{tikzcd}\]
In particular, $p^*$ is a fibration if and only if this is true for all such arrows $f : c \to p^*(d)$.
\end{prop}

\begin{proof}
If the arrow admits a cartesian lifting, then \ref{cartesianwithunit} gives us that the pullback exists and is given by the unit of the adjunction, which is transposed as the identity arrow.

For the converse direction, we saw in \ref{cartesianwithunit} that an arrow $f$ is cartesian if and only if its associated naturality square is cartesian. Thus, if we want to obtain a cartesian lifting of an arrow $f : c \to p^*(d)$, we should look at the pullback square:

\[\begin{tikzcd}
	{F'} & F \\
	{p_*(c)} & {p^*p_*(d)}
	\arrow["{\hat{f}}", from=1-1, to=1-2]
	\arrow["q"', from=1-1, to=2-1]
	\arrow["\lrcorner"{anchor=center, pos=0.125}, draw=none, from=1-1, to=2-2]
	\arrow["{\eta(F)}", from=1-2, to=2-2]
	\arrow["{p_*(f)}"', from=2-1, to=2-2]
\end{tikzcd}\]

If $q^t$ is an isomorphism, we have:

\[\begin{tikzcd}
	{d'} & {d'} & d \\
	{p_*p^*(d')} & {p_*(c)} & {p_*p^*(d)}
	\arrow["{1_{d'}}", from=1-1, to=1-2]
	\arrow["{\eta(d')}"', from=1-1, to=2-1]
	\arrow["\lrcorner"{anchor=center, pos=0.125}, draw=none, from=1-1, to=2-2]
	\arrow["{\widehat{f}}", from=1-2, to=1-3]
	\arrow["q"', from=1-2, to=2-2]
	\arrow["\lrcorner"{anchor=center, pos=0.125}, draw=none, from=1-2, to=2-3]
	\arrow["{\eta(d)}", from=1-3, to=2-3]
	\arrow["{p_*(q^t)}"', from=2-1, to=2-2]
	\arrow["\simeq", from=2-1, to=2-2]
	\arrow["{p_*p^*(\widehat{f})}"', bend right = 25, from=2-1, to=2-3]
	\arrow["{p_*(f)}"', from=2-2, to=2-3]
\end{tikzcd}\]

so that $f \circ q^t = p^*(\hat{f})$ and the outer rectangle is a pullback, that is, $\widehat{f}$ is a cartesian lifting for $f$.

Conversely, if $p^*$ is a fibration, we have an isomorphism $\sigma$ and a cartesian arrow $\hat{f}$ such that: 

\[\begin{tikzcd}
	{p^*(d')} & c & {p^*(d)}
	\arrow["\sigma", draw=none, from=1-1, to=1-2]
	\arrow["\simeq"', from=1-1, to=1-2]
	\arrow["{p^*(\hat{f})}"', bend right = 25, from=1-1, to=1-3]
	\arrow["f", from=1-2, to=1-3]
\end{tikzcd}\]

Since $\hat{f}$ is cartesian, this gives the following pullback squares:

\[\begin{tikzcd}
	{d'} & {\overline{d}} & d \\
	{p_*p^*(d')} & {p_*(c)} & {p_*p^*(d)}
	\arrow["{\sigma'}", from=1-1, to=1-2]
	\arrow["\simeq"', from=1-1, to=1-2]
	\arrow["{\hat{f}}", bend left = 28, from=1-1, to=1-3]
	\arrow["{\eta(d')}"', from=1-1, to=2-1]
	\arrow["\lrcorner"{anchor=center, pos=0.125}, draw=none, from=1-1, to=2-2]
	\arrow["{f'}", from=1-2, to=1-3]
	\arrow["q"', from=1-2, to=2-2]
	\arrow["\lrcorner"{anchor=center, pos=0.125}, draw=none, from=1-2, to=2-3]
	\arrow["{\eta(d)}", from=1-3, to=2-3]
	\arrow["{p_*(\sigma)}"', from=2-1, to=2-2]
	\arrow["\simeq", draw=none, from=2-1, to=2-2]
	\arrow["{p_*p^*(\hat{f})}"', bend right = 28, from=2-1, to=2-3]
	\arrow["{p_*(f)}"', from=2-2, to=2-3]
\end{tikzcd}\]

\noindent where $f'$ is the pullback of $p_*(f)$ and $\sigma'$ is the pullback of $p_*(\sigma)$.

Finally we have $(\sigma'q)^t = q^tp_*(\sigma')$, but also $(q\sigma')^t = (p_*(\sigma)\eta(F'))^t = \sigma$, so that $q^t = \sigma p_*(\sigma')^{-1}$ is an iso.
\end{proof}

\subsection{Locally connected morphisms as fibrations}\label{locallyconnectedsection}

In this subsection, we recall the definition of a locally connected morphism. We then show that the previous formulation of the existence of cartesian liftings, when applied to the left adjoint $\Sh(p)^*$, is closely related to local connectedness.

\begin{defn}
We say that an essential geometric morphism $f : {\cal F} \to {\cal E}$ (a geometric morphism of which the inverse image $f^*$ admits a further left adjoint $f_!$) is locally connected when its essential image $f_!$ induces an indexed left adjoint $(f_!) : S_{f} \to S_{\cal E}$ given at an object $E$ by $(f_!)_E : ({\cal F}/f^*(E)) \to ({\cal E}/E)$ sending $[\alpha : F \to f^*(E)]$ to $[\alpha^t: f_!(F) \to E]$. 
\end{defn}

\begin{remark}
For an essential geometric morphism we always have such a family of functors between the fibers, but it is not always the case that it is \emph{indexed}. Asking for $(f_!)$ to be indexed is exactly asking for every arrow $h : E' \to E$ in $\cal E$ to induce a commutative square (up to iso):

\[\begin{tikzcd}
	{{\cal F}/f^*(E)} & {{\cal F}/f^*(E')} \\
	{{\cal E}/E} & {{\cal E}/E'}
	\arrow["{S_f(h)}", shift left, from=1-1, to=1-2]
	\arrow["{(f_!)_E}"', from=1-1, to=2-1]
	\arrow["{(f_!)_{E'}}", from=1-2, to=2-2]
	\arrow["{S_{\cal E}(h)}"', from=2-1, to=2-2]
\end{tikzcd}\]
\end{remark}

The following important remark explains how the existence of cartesian liftings is related to the locally connectedness: 

\begin{remarks}\label{frobrempourrelevcart}

\begin{enumerate}
    \item For a continuous comorphism of sites $p : ({\cal D},K) \to ({\cal C},J)$ and an arrow $f : F \to \Sh(p)^*(E)$ in $\widehat{\cal C}_J$, we have a square of functors which does not commute in general, but commutes (in particular) when $C_p$ is locally connected:
\[\begin{tikzcd}
	{S_{C_p}(\Sh(p)^*E)} & {S_{C_p}(F)} \\
	{S_{\widehat{\cal C}_J}(\Sh(p)^*E)} & {S_{\widehat{\cal C}_J}(F)}
	\arrow["{S_{C_p}(f)}", from=1-1, to=1-2]
	\arrow["{(\Sh(p)^*)_{\Sh(p)^*E}}"', from=1-1, to=2-1]
	\arrow["{(\Sh(p)^*)_F}", from=1-2, to=2-2]
	\arrow["{S_{\widehat{\cal C}_J}(f)}"', from=2-1, to=2-2]
\end{tikzcd}\]

In the light of \ref{liftingwithunit}, asking the existence of a cartesian lifting for $f$ is exactly asking that the previous square commutes at the object $[\eta_p(E)]$, that is:

$$(\Sh(p)^*)_FS_{C_p}(f)([\eta_p(E)]) \simeq S_{\widehat{\cal C}_J}(f)(\Sh(p)^*)_{\Sh(p)^*E}([\eta_p(E)])$$

Indeed, the lower path of the square, when evaluated at $[\eta_p(E)]$, gives us first the transposition of $\eta_p(E)$ which is the identity at $E$, since $\eta_p$ is the unit of the adjunction $C_p^* \vdash Sh(p)^*$, and then we pull back the identity which gives us again the identity. The upper path of the square, when evaluated at $[\eta_p(E)]$, gives us the transpose of the arrow $q$ in the following pullback:

\[\begin{tikzcd}
	{F'} & E \\
	{C_p^*F} & {C_p^*\Sh(p)^*E}
	\arrow[from=1-1, to=1-2]
	\arrow["q"', from=1-1, to=2-1]
	\arrow["\lrcorner"{anchor=center, pos=0.125}, draw=none, from=1-1, to=2-2]
	\arrow["{\eta_p(E)}", from=1-2, to=2-2]
	\arrow["{C_p^*(f)}"', from=2-1, to=2-2]
\end{tikzcd}\]
    This is the arrow used in \ref{liftingwithunit}, where we state that $f$ admits a lifting if and only if it is an isomorphism; that is, if and only if the two paths of the square yield isomorphic objects when evaluated at $[\eta_p(E)]$.
    \item The same reflexion holds for \ref{cartesianwithunit}: for a continuous comorphism of sites $p : ({\cal D},K) \to ({\cal C},J)$ and an arrow $f : E' \to E$ in $\widehat{\cal D}_K$, the fact that the following square commutes (up to iso) on the object $[\eta_p(E)]$
    \[\begin{tikzcd}
	{S_{C_p}(\Sh(p)^*E)} & {S_{C_p}(\Sh(p)^*E')} \\
	{S_{\widehat{\cal C}_J}(\Sh(p)^*E)} & {S_{\widehat{\cal C}_J}(\Sh(p)^*E')}
	\arrow["{S_{C_p}(\Sh(p)^*f)}", shift left=2, from=1-1, to=1-2]
	\arrow["{(\Sh(p)^*)_{\Sh(p)^*E}}"', from=1-1, to=2-1]
	\arrow["{(\Sh(p)^*)_{\Sh(p)^*E'}}", from=1-2, to=2-2]
	\arrow["{S_{\widehat{\cal C}_J}(\Sh(p)^*f)}"', shift right=2, from=2-1, to=2-2]
    \end{tikzcd}\]

\noindent is equivalent to $f$ being a cartesian for the left adjoint $\Sh(p)^*$.
\end{enumerate}
\end{remarks}

The previous remark gives us that the essential image of any locally connected geometric morphism admits cartesian liftings: it forms a fibration.

\begin{prop}\label{locallyconnimpliesfib}
Let $f : {\cal F} \to {\cal E}$ be a locally connected geometric morphism. Its essential image $f_! : {\cal F} \to {\cal E}$ is a fibration.     
\end{prop}

\qed

Recall from \cite{CaramelloZanfa}, Definition 2.11.1, that every fibration over a base site can be endowed with its associated Giraud topology: this is the smallest topology making the projection a comorphism of sites. We denote it by \( p : (\mathcal{G}(\mathbb{D}), J_{\mathbb{D}}) \to (\mathcal{C}, J) \). Corollary 4.58 in \cite{denseness} recalls that the projection functor of such a \textquotedblleft trivial \textquotedblright relative site is continuous, and states that the induced essential geometric morphism is in fact locally connected. Combining this result with the previous proposition, we obtain:

\begin{prop}\label{fibimplfibtopos}
Let \( p : (\mathcal{G}(\mathbb{D}), J_{\mathbb{D}}) \to (\mathcal{C}, J) \) be a trivial relative site. The essential image $C_{p!} : \widehat{\mathcal{G}(\mathbb{D})}_{J_{\mathbb{D}}} \to \widehat{\cal C}_J$ of $C_p$ is a fibration.
\end{prop}
\qed

Since the projection functor of a trivial relative site is known to induce a locally connected geometric morphism, it is natural to ask whether a similar result holds for local fibrations. Taking a \emph{continuous} local fibration: a local fibration having its projection functor not only being a comorphism but also a continuous functor, we have that the induced geometric morphism is locally connected:

\begin{prop}\label{locconnlocfib}
Let $p : ({\cal D},K) \to ({\cal C},J)$ be continuous local fibration. Then, $C_p$ is a locally connected geometric morphism. 
\end{prop}

\begin{proof}
To see this, we are going to verify that it satisfies the conditions of the point (ii) of the Theorem 4.55. in \cite{denseness}, which gives characterization of those continuous comorphisms inducing locally connected morphisms; we use the same notations.

First condition (point (i) of the point (ii)): let $h : Q \to Q'$ be a morphism of presheaves on $\cal C$, and $(d,x)$ an object $d$ of $\cal D$ with an element $x$ of $Q'(pd)$. Let $(c,z,g)$ be the data of an object $c$ of $\cal C$, an element $z$ of $Q(c)$ and an arrow $g : c \to pd$ such that $Q'(g)(x) = h_{c}(z)$ (an object of the category $\mathcal{B}^h_{(x,d)}$). We are going to show that we can locally reach it with objects in the image of $\xi^h_{(x,d)} : \mathcal{A}^h_{(x,d)} \to \mathcal{B}^h_{(x,d)}$. 

We can locally lift $g$ into locally cartesian arrows: $gv_i = p(\widehat{g_i})$. And then, we just take the objects $(d_i,Q(v_i)(z),\widehat{g_i})$ of $\mathcal{A}^h_{(x,d)}$. Indeed, we have that the following diagram commutes:

\[\begin{tikzcd}
	{Q(pd_i)} & {Q(c)} \\
	{Q'(pd_i)} & {Q'(c)} & {Q'(pd)}
	\arrow["{Q'(g)}"', from=2-3, to=2-2]
	\arrow["{Q'(v_i)}"', from=2-2, to=2-1]
	\arrow["{Q'(\widehat{g_i})}", bend left = 36, from=2-3, to=2-1]
	\arrow["{h_{pd_i}}"', from=1-1, to=2-1]
	\arrow["{h_c}"', from=1-2, to=2-2]
	\arrow["{Q(v_i)}"', from=1-2, to=1-1]
\end{tikzcd}\]

and hence, $h_{pd_i}Q(v_i)(z) = Q'(v_i)h_c(z)$ (by naturality of $h$), but $h_c(z)  = Q'(g)(x)$, so $h_{pd_i}Q(v_i)(z) = Q'(v_ig)(x) = Q'p(\widehat{g_i})(x)$, which achieve to prove that the

\noindent $(d_i,Q(v_i)(z),\widehat{g_i})$ are objects of $\mathcal{A}^h_{(x,d)})$. Finally, as required, we have locally reached $b_{(x,d)}^h(c,z,g)$:

\[\begin{tikzcd}
	& {c=b_{(x,d)}^h(c,z,g)} \\
	{p(d_i)} & {p(d_i) = b^h_{(x,d)}(\xi_{(x,d)}(d_i,Q(v_i)(z),\widehat{g_i}))}
	\arrow["{v_i}", from=2-1, to=1-2]
	\arrow["{1_{pd_i}}"', from=2-1, to=2-2]
	\arrow["{ b^h_{(x,d)}(v_i)}"', from=2-2, to=1-2]
\end{tikzcd}\]

For the second condition ((ii) of point (ii) of the theorem), let's take two objects $(d',y,f)$ and $(d'',y',f')$ of the category $\mathcal{A}^h_{(x,d)})$ and two arrows $\alpha  : c \to b^h_{(x,d)}(\xi_{(x,d)}(d',y,f))$ and $\alpha'  : c \to b^h_{(x,d)}(\xi_{(x,d)}(d'',y',f'))$ such that we have a zig-zag:

\[\begin{tikzcd}
	& {b^h_{(x,d)}(\xi_{(x,d)}(d',y,f))} \\
	& {b^h_{(x,d)}(c_1,z_1,g_1)} \\
	c \\
	& {b^h_{(x,d)}(c_n,z_n,g_n)} \\
	& {b^h_{(x,d)}(\xi_{(x,d)}(d'',y',f'))}
	\arrow["\alpha", from=3-1, to=1-2]
	\arrow["{\alpha'}"', from=3-1, to=5-2]
	\arrow["{b^h_{(x,d)}(u_1)}"', from=2-2, to=1-2]
	\arrow["{b^h_{(x,d)}(u_n)}", from=4-2, to=5-2]
	\arrow[""{name=0, anchor=center, inner sep=0}, squiggly, no head, from=4-2, to=2-2]
	\arrow["{\alpha_1}"{description}, from=3-1, to=2-2]
	\arrow["{\alpha_n}"{description}, from=3-1, to=4-2]
	\arrow[shorten >=10pt, squiggly, no head, from=3-1, to=0]
\end{tikzcd}\]

From the existence of this zig-zag we want to obtain, locally, a zig-zag coming not only from $b^h_{(x,d)}$, but from $b^h_{(x,d)}\xi_{(x,d)}$. Notice that such a zig-zag can always be chosen to be such that the direction of the consecutive arrows are opposite. Now we distinguish two cases: the first one is when the arrow $u_1$ is pointing in the direction indicated in the precedent diagram, and the second one will be when it points to the opposite. We will show that, whatever the direction, we can always (locally) turn $\alpha_1$ as coming from $b^h_{(x,d)}\xi_{(x,d)}$, and step by step this will give us the whole wished zig-zag. So let $u_1$ be as in the zig-zag diagram, we have:

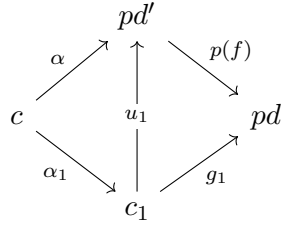
\[\begin{tikzcd}
	& {pd'} \\
	c && pd \\
	& {c_1}
	\arrow["{\alpha_1}"', from=2-1, to=3-2]
	\arrow["\alpha", from=2-1, to=1-2]
	\arrow["{p(f)}", from=1-2, to=2-3]
	\arrow["{g_1}"', from=3-2, to=2-3]
	\arrow["{u_1}"{description}, from=3-2, to=1-2]
\end{tikzcd}\]

We locally lift $u_1$ into locally cartesian arrows: $u_1v_i = p(\widehat{u_1^i})$, and also pullback the covering $(v_i)_i$ along $\alpha_1$:

\[\begin{tikzcd}
	& {pd'} \\
	c && pd \\
	& {c_1} \\
	{c_{ij}} & {p(d'_i)}
	\arrow["{\alpha_1}"', from=2-1, to=3-2]
	\arrow["\alpha", from=2-1, to=1-2]
	\arrow["{p(f)}", from=1-2, to=2-3]
	\arrow["{g_1}"', from=3-2, to=2-3]
	\arrow["{u_1}"', from=3-2, to=1-2]
	\arrow["{v_i}"', from=4-2, to=3-2]
	\arrow["{p(\widehat{u_1^i})}"{description, pos=0.6}, bend left = 24, from=4-2, to=1-2]
	\arrow["{v_{ij}}", from=4-1, to=2-1]
	\arrow["{w_{ij}}"', from=4-1, to=4-2]
\end{tikzcd}\]

And the situation is now:

\[\begin{tikzcd}
	&& {p(d')=b^h_{(x,d)}\xi_{(x,d)}(d',y,f)} \\
	{c_{ij}} && {p(d_i) = b^h_{(x,d)}\xi_{(x,d)}(d_i,Q(\widehat{v_1^i})(y),f\widehat{u_1^i})} \\
	&& {c_2 = b^h_{(x,d)}(c_2,z_2,g_2)}
	\arrow["{\alpha v_{ij}}", from=2-1, to=1-3]
	\arrow["{w_{ij}}"{description}, from=2-1, to=2-3]
	\arrow["{\alpha_2v_{ij}}"', from=2-1, to=3-3]
	\arrow["{p(\widehat{u_1^i})=b^h_{(x,d)}\xi_{(x,d)}(\widehat{u_1^i})}"', from=2-3, to=1-3]
	\arrow["{u_2v_i = b^h_{(x,d)}(u_2v_i)}", from=2-3, to=3-3]
\end{tikzcd}\]

we are reduced to the second case.

For the second case, let's assume that the first arrow $u_1$ is in the opposite direction: $u_1 : pd' \to c_1$.  

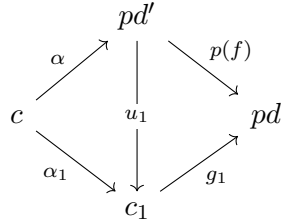
\[\begin{tikzcd}
	& {pd'} \\
	c && pd \\
	& {c_1}
	\arrow["{\alpha_1}"', from=2-1, to=3-2]
	\arrow["\alpha", from=2-1, to=1-2]
	\arrow["{p(f)}", from=1-2, to=2-3]
	\arrow["{g_1}"', from=3-2, to=2-3]
	\arrow["{u_1}"{description}, from=1-2, to=3-2]
\end{tikzcd}\]

This time, we can locally lift $g_1$ as locally cartesian arrows: $g_1v_i = p(\widehat{g_i})$, and then pull back the covering $(v_i)_i$ along $u_1$:

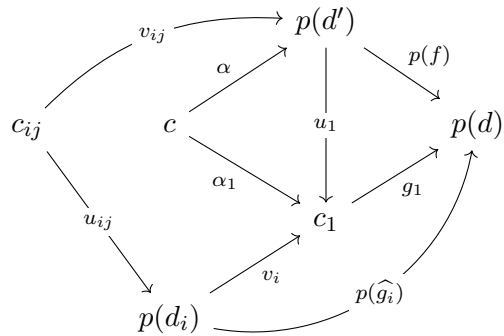
\[\begin{tikzcd}
	&& {p(d')} \\
	{c_{ij}} & c && {p(d)} \\
	&& {c_1} \\
	& {p(d_i)}
	\arrow["\alpha", from=2-2, to=1-3]
	\arrow["{\alpha_1}"', from=2-2, to=3-3]
	\arrow["{p(f)}", from=1-3, to=2-4]
	\arrow["{g_1}"', from=3-3, to=2-4]
	\arrow["{u_1}"{description}, from=1-3, to=3-3]
	\arrow["{v_i}"', from=4-2, to=3-3]
	\arrow["{v_{ij}}"{description}, bend left = 24, from=2-1, to=1-3]
	\arrow["{u_{ij}}"{description}, from=2-1, to=4-2]
	\arrow["{p(\widehat{g_i})}"{description}, bend right = 45, from=4-2, to=2-4]
\end{tikzcd}\]

Now we can locally lift the $u_{ij}$ as locally cartesian arrows: $u_{ij}v_{ijk} = p(\widehat{u_{ijk}})$. And then, by the locally cartesian property of the $\widehat{g_i}$, since $p(\widehat{g_i})v_{ijk}u_{ij} = p(f\widehat{u_{ijk}})$, we can also locally lift $v_{ijk}u_{ij}$, as in the diagram:

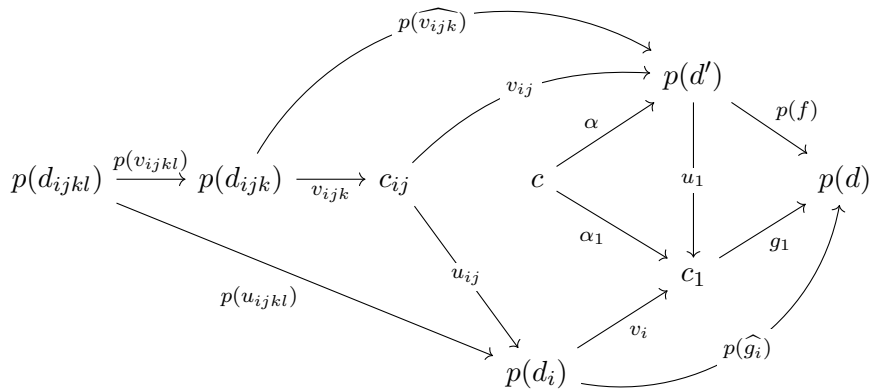
\[\begin{tikzcd}
	&&&& {p(d')} \\
	{p(d_{ijkl})} & {p(d_{ijk})} & {c_{ij}} & c && {p(d)} \\
	&&&& {c_1} \\
	&&& {p(d_i)}
	\arrow["\alpha", from=2-4, to=1-5]
	\arrow["{\alpha_1}"', from=2-4, to=3-5]
	\arrow["{p(f)}", from=1-5, to=2-6]
	\arrow["{g_1}"', from=3-5, to=2-6]
	\arrow["{u_1}"{description}, from=1-5, to=3-5]
	\arrow["{v_i}"', from=4-4, to=3-5]
	\arrow["{v_{ij}}"{description}, bend left = 24, from=2-3, to=1-5]
	\arrow["{u_{ij}}"{description}, from=2-3, to=4-4]
	\arrow["{p(\widehat{g_i})}"{description}, bend right =45, from=4-4, to=2-6]
	\arrow["{v_{ijk}}"', from=2-2, to=2-3]
	\arrow["{p(\widehat{v_{ijk}})}"{description}, bend left = 45, from=2-2, to=1-5]
	\arrow["{p(v_{ijkl})}", from=2-1, to=2-2]
	\arrow["{p(u_{ijkl})}"', from=2-1, to=4-4]
\end{tikzcd}\]

Since $p$ is continuous it preserves covering, hence the $p(v_{ijkl})$ are covering arrows, and we can multicompose to obtain a covering $(v_{ij}v_{ijk}p(v_{ijkl}))_{ijkl}$. We pullback this covering along $\alpha$, and the situation is now:



\begin{small}
\[\begin{tikzcd}
	& {p(d')=b^h_{(x,d)}\xi_{(x,d)}(d',y,f)} \\
	&& {p(d_{ijkl})} \\
	{c_{ijklm}} & {c_1=b^h_{(x,d)}(c_1,z_1,g_1)} & {b^h_{(x,d)}\xi_{(x,d)}(d_{ijkl},Q(\widehat{v_{ijk}}v_{ijkl})(y),f\widehat{v_{ijk}}v_{ijkl})} \\
	\\
	& {p(d_i)=b^h_{(x,d)}\xi_{(x,d)}(d'_i,Q(v_i)(y),f\widehat{g_i})}
	\arrow["{b^h_{(x,d)}(u_1)}", from=1-2, to=3-2]
	\arrow[equals, from=2-3, to=3-3]
	\arrow[from=3-1, to=1-2]
	\arrow[from=3-1, to=3-2]
	\arrow[bend right = 10, from=3-1, to=3-3]
	\arrow[from=3-1, to=5-2]
	\arrow["{b^h_{(x,d)}\xi_{(x,d)}(\widehat{v_{ijk}}v_{ijkl})}"', from=3-3, to=1-2]
	\arrow["{b^h_{(x,d)}\xi_{(x,d)}(u_{ijkl})}", from=3-3, to=5-2]
	\arrow["{b^h_{(x,d)}(v_i)}"', from=5-2, to=3-2]
\end{tikzcd}\]
\end{small}

in which it is lengthy but straightforward to check the commutativity of all the wished arrows.

But recall that the $v_i$ are covering, and hence it suffices to pull them back along $u_2$ to get back to the precedent case. Applying alternatively these two cases gives us a zig-zag coming from $\xi_{(x,d)}^h$, hence the conditions are satisfied and $C_p$ is locally connected.
\end{proof}

The fact that a contiuous local fibration induces a locally connected morphism allows to deduce that the essential image it induces is a fibration:

\begin{prop}\label{continuouslocfibinducesfib}
Let $p : ({\cal D},K) \to ({\cal C},J)$ be a continuous local fibration. Then, $\Sh(p)^*$ is a fibration.
\end{prop}

\begin{proof}
By \ref{locconnlocfib} we have that $\Sh(p)^*$ is indexed, so that the reformulation condition \ref{frobrempourrelevcart} is satisfied, that is $\Sh(p)^*$ is a fibration by \ref{liftingwithunit}.
\end{proof}

This allows us to obtain a characterization of locally connected geometric morphisms in these terms:

\begin{prop}
Let $f : {\cal F}\to {\cal E}$ be an essential geometric morphism. It is locally connected if and only if its essential image $f_! : {\cal F} \to {\cal E}$ is a fibration.
\end{prop}

\begin{proof}
The first direction is given by \ref{locallyconnimpliesfib}.
In the other direction, we assume $f_!$ to be a fibration. It is a general fact that, for an essential geometric morphism $f$, its essential image $f_!$ is a continuous comorphism of sites such that $C_{f_!} \simeq f$. It is a fibration, so in particular a local fibration, and it is moreover continuous: we can apply \ref{locconnlocfib}: $C_{f_!} \simeq f$ is locally connected.
\end{proof}

Also we can deduce that, in the case of toposes with their canonical topologies, continuous local fibrations coincide with fibrations:

\begin{prop}\label{prop:canonicaltopologylocalfibration}
A colimit-preserving (i.e. continuous) comorphism of sites $p : ({\cal F},J_{\cal F}^{can}) \to ({\cal E},J_{\cal E}^{can})$ between toposes is a local fibration if and only if it is a fibration.
\end{prop}

\begin{proof}
By \ref{continuouslocfibinducesfib} we have that $\Sh(p)^*$ is a fibration. But, since $p$ preserves colimits, we have $\Sh(p)^* \simeq p$, from where follows the proposition. 
\end{proof}

As previously mentioned \ref{continuouslocfibinducesfib}, a fibration endowed with its Giraud topology $p : (\cal D,J_{\cal D}) \to (\cal C,J)$—or more generally, a continuous local fibration—induces an essential image $\Sh(p)^* : Gir_J(\cal D) \to \widehat{\cal C}_J$ that is itself a fibration (see \ref{continuouslocfibinducesfib}). In the following proposition, we begin by giving general properties for some functors relating the local fibration, its relative topos and its canonical stack:

\begin{prop}\label{etacontlocfib}
Let $p : ({\cal D},K) \to ({\cal C},J)$ be a continuous local fibration. 

\begin{enumerate}[(i)]
    \item The functor $\eta_{\widehat{\cal D}_K} : \widehat{\cal D}_K \to (\widehat{\cal D}_K/C_p^*)$ is fully faithful and reflects cartesian arrows.
    \item The canonical (absolute) functor $l_K : (\cal D,K) \to \widehat{\cal D}_K $ reflects locally cartesian arrows. In particular, it is a morphism of local fibrations.
    \item If moreover $\cal D$ is a fibration, $l_K$ is a morphism of fibrations. 
    .
    \item There is a canonical natural transformation $\overline{\alpha}$

\[\begin{tikzcd}
	{\widehat{\cal D}_K} && {(\widehat{\cal D}_K/C_p^*)} \\
	& {\widehat{\cal C}_J}
	\arrow["{{C_p}_!}"', from=1-1, to=2-2]
	\arrow["{\pi_{\cal D}}"', from=1-3, to=1-1]
	\arrow[""{name=0, anchor=center, inner sep=0}, "{\pi_{C_p}}", from=1-3, to=2-2]
	\arrow["{\overline{\alpha}}"', shorten <=12pt, shorten >=12pt, Rightarrow, from=1-1, to=0]
\end{tikzcd}\]
    which is given, at an object $(E,F,\alpha : E \to C_p^*F)$, by the transpose of $\alpha$ : $\overline{\alpha} : {C_p}_!E \to F = \pi_{\cal D}\pi_{C_p}(E,F,\alpha)$.
\end{enumerate}
\end{prop}

\begin{proof}
For the first point, since an arrow $u : F \to F'$ is sent by $\eta_{\widehat{\cal D}_K}$ to $(u,{C_{p}}_!(u)) : \eta_{\widehat{\cal D}_K}(F) \to \eta_{\widehat{\cal D}_K}(F')$, it is faithful since the first component is uniquely determined by the antecedent arrow $u$. Moreover, if we have an arrow  $(u,v) : \eta_{\widehat{\cal D}_K}(F) \to \eta_{\widehat{\cal D}_K}(F')$ as depicted:

\[\begin{tikzcd}
	F & {F'} \\
	{C_p^*{C_p}_!(F)} & {C_p^*{C_p}_!(F')}
	\arrow["u", from=1-1, to=1-2]
	\arrow["{\eta_{F}}"', from=1-1, to=2-1]
	\arrow["{\eta_{F'}}", from=1-2, to=2-2]
	\arrow["{C_p^*(v)}"', from=2-1, to=2-2]
\end{tikzcd}\]

\noindent it is forced that $v = {C_p}_!(u)$ because, by the naturality of the adjunction $C_p^* \vdash {C_p}_!$, $u = (\eta_{F'}u)^t =  {C_p}_!(u)$.

And $\eta_{\widehat{\cal D}_K}$ reflects locally cartesian arrows because it reflects locally cartesian arrows, but locally cartesian arrows in $(\widehat{\cal D}_K,J_{\widehat{\cal D}_K}^{can})$ are cartesian, since the topology is the canonical one \ref{prop_cartesiancanonicaltopology}; and locally cartesian arrows in $(\widehat{\cal D}_K/C_p^*)$ also are cartesian ones because of \ref{canonicalcartesian}.

For (ii), it is immediate that $\eta_{\widehat{\cal D}_K}l_K \simeq \eta_{\cal D}$. Hence, an arrow $f$ in $\cal D$ is locally cartesian if and only if $\eta_{\cal D}(f)$ is cartesian, that is, if and only if $\eta_{\widehat{\cal D}_K}l_K(f)$ is cartesian. But by (i) we know that $\eta_{\widehat{\cal D}_K}$ reflects cartesian arrows, so that $f$ is cartesian if and only if $l_K(f)$ is.

For (iii), we know that a cartesian arrow is in particular a locally cartesian one; and, since locally cartesian ones in $\widehat{\cal D}_K$ are cartesian because of \ref{canonicalcartesian}, the previous point gives us the preservation of cartesian arrows.

The last point is immediate.
\end{proof}

Since we know that, when $p$ is a fibration endowed with its Giraud topology, ${C_p}_! : Gir_J(\cal D) \to \widehat{\cal C}_J$ is a fibration, we may be interested in studying the glueing property of this fibration:

\begin{prop}
Let $p : (\cal D,J_{\cal D}) \to (\cal C,J)$ be a fibration with its trivial topology. The fibration ${C_p}_! : Gir_J(\cal D) \to \widehat{\cal C}_J$ is a stack.
\end{prop}

\begin{proof}
From point (i) of the previous proposition, we know that $\eta_{\cal D}$ is a fully faithful morphism of fibrations; hence $Gir_J(\cal D)$ is a subobject of a stack: it is a prestack. Indeed, let $E$ be an object of $\widehat{\cal C}_J$, for $S$ a covering sieve on it we consider the category of descent datum on $S$, that is the morphisms of fibrations of the form $m$:

\[\begin{tikzcd}
	{\int S} && {Gir_J(\cal D)} \\
	& {\widehat{\cal C}_J}
	\arrow["m", from=1-1, to=1-3]
	\arrow["{\pi_S}"', from=1-1, to=2-2]
	\arrow["{{C_p}_!}", from=1-3, to=2-2]
\end{tikzcd}\]

We have the following commutative diagram:
\[\begin{tikzcd}
	{\mathbf{Fib}_{\widehat{\cal C}_J}(\int S,Gir_J({\cal D}))} & {\mathbf{Fib}_{\widehat{\cal C}_J}(\int S,(Gir_J({\cal D})/C_p^*))} \\
	\\
	{\mathbf{Fib}_{\widehat{\cal C}_J}((\widehat{\cal C}_J/E),Gir_J({\cal D}))} & {\mathbf{Fib}_{\widehat{\cal C}_J}((\widehat{\cal C}_J/E),(Gir_J({\cal D})/C_p^*))}
	\arrow[hook, from=1-1, to=1-2]
	\arrow["{\eta_{Gir_J({\cal D})} \circ -}"{description}, shift left=4, draw=none, from=1-1, to=1-2]
	\arrow["{\rotatebox{270}{$\simeq$}}"{description}, draw=none, from=1-2, to=3-2]
	\arrow["{(- \circ m)}", from=3-1, to=1-1]
	\arrow[hook, from=3-1, to=3-2]
	\arrow["{\eta_{Gir_J({\cal D})} \circ -}"{description}, shift right=4, draw=none, from=3-1, to=3-2]
\end{tikzcd}\]

Here, the postcomposition by $\eta_{Gir_J(\cal D)}$ is fully faithful, since $\eta_{Gir_J(\cal D)}$ is a fully faithful morphism of fibrations. The equivalences on the right comes from the very definition of $(Gir_J({\cal D})/C_p^*)$ being a stack. Since all the other functors in this diagram are equivalences or fully faithful functors, this forces $(- \circ m)$ to be fully faithful, that is, ${C_p}_!$ is a prestack.

It remains to show that $(-\circ m)$ is not only fully faithful, but also essentially surjective. To this end, we consider a descent datum for $Gir_J(\cal D)$ on the covering sieve $S$: this consists of a family of objects $F_f$ in $Gir_J(\cal D)$, indexed by the arrows $f : E_f \to E$ in $S$, such that ${C_p}_!(F_f) \simeq E_f$, together with some isomorphisms and compatibility conditions which we do not make explicit here for the sake of conciseness. 

The crucial point is that, since $(Gir_J({\cal D})/C_p^*)$ is a stack, when we take the image through $\eta_{Gir_J(\cal D)}$ of this descent datum, we have a gluing $(F,E, \alpha : F \to C_p^*(E))$ for the $\eta(F_f)$ in $(Gir_J({\cal D})/C_p^*)$. What we want to show is that this gluing is in the essential image of $\eta_{Gir_J(\cal D)}$.

First, we know that ${C_p}_!(F_f) \simeq E_f$ are covering $E$ for the canonical topology: hence, by the colimit-preservation of ${C_p}_!$, we know that $E \simeq \colim F_f$. We denote $\colim F_f$ by $F'$: the gluing of the $\eta(F_f)$ is of the form $(F,{C_p}_!(F'), \alpha : F \to C_p^*{C_p}_!(F'))$. Also, we denote by $u_f : F_f \to F'$ the arrows constituting the cocone of this colimit.

Now, recall that $(F,{C_p}_!(F'), \alpha : F \to C_p^*{C_p}_!(F'))$ being a gluing for the $\eta(F_f)$ exactly means that the squares of the following form are pullbacks:

\[\begin{tikzcd}
	F & {C_p^*{C_p}_!(F')} \\
	{F_f} & {C_p^*{C_p}_!(F_f)}
	\arrow["\alpha", from=1-1, to=1-2]
	\arrow["{u_f'}", from=2-1, to=1-1]
	\arrow["\lrcorner"{anchor=center, pos=0.125, rotate=90}, draw=none, from=2-1, to=1-2]
	\arrow["{\eta(F_f)}"', from=2-1, to=2-2]
	\arrow["{C_p^*{C_p}_!(u_f)}"', from=2-2, to=1-2]
\end{tikzcd}\]

But, since $F'$ is a colimit of the $F_f$ and the functors $C_p^*$ and ${C_p}_!$ preserve them, the right part of these pullbacks constitute a colimit cocone. Moreover, since colimits are stable under pullbacks, we have that $F$ is isomorphic to $F'$, and the $u_f'$ are (modulo this isomorphism) the $u_f$. Hence, these pullback squares are of the form:

\[\begin{tikzcd}
	{F'} & {C_p^*{C_p}_!(F')} \\
	{F_f} & {C_p^*{C_p}_!(F_f)}
	\arrow["\alpha", from=1-1, to=1-2]
	\arrow["{u_f}", from=2-1, to=1-1]
	\arrow["\lrcorner"{anchor=center, pos=0.125, rotate=90}, draw=none, from=2-1, to=1-2]
	\arrow["{\eta(F_f)}"', from=2-1, to=2-2]
	\arrow["{C_p^*{C_p}_!(u_f)}"', from=2-2, to=1-2]
\end{tikzcd}\]

Finally, we know that when we replace $\alpha$ by $\eta(F')$ this square commutes (naturality of $\eta$ which is the unit of the adjunction $C_p^* \vdash {C_p}_!$). By the universal property of the colimit $F' \simeq \colim F_f$, it is the only arrow making all of these diagram commute: $\alpha = \eta(F')$. This shows that the gluing is of the form $(F',{C_p}_!(F'),\eta(F'))$: it belongs in the essential image of $\eta$, that is, $Gir_J(\cal D)$ is a stack.
\end{proof}

\begin{remark}
Let's consider a topos $(\cal E,J_{\cal E}^{can})$ as a base site together with a presheaf $P$ on it. By taking the Giraud topos of $\int P \to (\cal E,J_{\cal E}^{can})$ we obtain the étale topos $(\cal E/a_{J_{\cal E}^{can}}(P)) \to \cal E$ of which the essential image is the projection on the first component. This fibration is equivalent to $\int a_{J_{\cal E}^{can}}(P)$, and the essential image is just the projection of this fibration: the Giraud topos of a presheaf together with its essential image gives us the sheafification of the presheaf.

Moreover, if we work with a base site $(\cal C,J)$ and a presheaf $P$ on it, we can form the following bipullback:

\[\begin{tikzcd}
	{\int a_J(P)} & {(\widehat{\cal C}_J/a_J(P))} \\
	{({\cal C},J)} & {\widehat{\cal C}_J}
	\arrow[from=1-1, to=1-2]
	\arrow["{\pi_{a_J(P)}}"', from=1-1, to=2-1]
	\arrow["\lrcorner"{anchor=center, pos=0.125}, draw=none, from=1-1, to=2-2]
	\arrow["{{C_{\pi_P}}_!}", from=1-2, to=2-2]
	\arrow["{l_J}"', from=2-1, to=2-2]
\end{tikzcd}\]

\noindent which gives us the sheafification of the presheaf $P$. 
\end{remark}

\subsection{A characterization for inducing a fibration at the topos-level}

The previous subsection established that a continuous local fibration $p : ({\cal D},K) \to ({\cal C},J)$ induces a fibration $\Sh(p)^* : \widehat{\cal D}_K \to \widehat{\cal C}_J$. In this subsection, we are interested in the extent to which a general continuous comorphism of sites $p : ({\cal D},K) \to ({\cal C},J)$ induces a fibration $\Sh(p)^* : \widehat{\cal D}_K \to \widehat{\cal C}_J$.

We begin by observing a property of continuous local fibrations that will guide our analysis. For every arrow $f : c \to p(d)$ in the base category ${\cal C}$ of a local fibration $p : ({\cal D},K) \to ({\cal C},J)$, there exists a covering family $(p(d_i) \to c)_i$ locally factorizing $f$ through arrows of the form $p(\widehat{f_i})$ for arrows $\widehat{f_i}$ in ${\cal D}$. This naturally leads to the question: can $c$ be described not only as being covered by such objects, but also as a colimit of objects coming from ${\cal D}$, at the topos-level ? The following definition will be the starting point in order to obtain results about such colimit representations:

\begin{defn}\label{Dffact}
Let $p : {\cal D} \to {\cal C}$ be a functor. For every object $c$ of ${\cal C}$ together with an object $d$ of ${\cal D}$ and an arrow $f : c \to p(d)$, we can define the category ${\cal D}^{\mathbf{fact}}_{f,d}$ as having for objects the triplets $(d',v' : p(d') \to c,f' :d' \to d)$ such that $fv' = p(f')$, and arrows between two such triplets $(d',v',f')$ and $(d'',v'',f'')$ the arrows $u : d' \to d''$ in ${\cal D}$ such that $v''p(u) = v'$ and $f''u=f'$. 

We have a projection functor: $\pi_d^f : {\cal D}^{\mathbf{fact}}_{f,d} \to {\cal C}/c$ sending a triplet $(d',v',f')$ to $(p(d'),v')$. We denote the following composite $p^f_d$:

\[\begin{tikzcd}
	{{\cal C}/c} & {{\cal C}} \\
	{{\cal D}^{\mathbf{fact}}_{f,d}}
	\arrow["{\pi_c}", from=1-1, to=1-2]
	\arrow["{\pi_d^f}", from=2-1, to=1-1]
	\arrow["{p^f_d}"', from=2-1, to=1-2]
\end{tikzcd}\]
\end{defn}

\begin{remark}
It will be useful to notice that the category ${\cal D}^{\mathbf{fact}}_{f,d}$ is the category of elements of the following pullback presheaf:

\[\begin{tikzcd}
	F & {{\cal D}(-,d)} \\
	{{\cal C}(-,c)} & {{\cal C}(-,p(d))}
	\arrow[from=1-1, to=1-2]
	\arrow[from=1-1, to=2-1]
	\arrow["\lrcorner"{anchor=center, pos=0.125}, draw=none, from=1-1, to=2-2]
	\arrow["{\eta_p}", from=1-2, to=2-2]
	\arrow["{f\circ -}"', from=2-1, to=2-2]
\end{tikzcd}\]
\end{remark}

What we want to do now is: if $p : ({\cal D},K) \to ({\cal C},J)$ is a local fibration, when we have some arrow $f : c \to p(d) $ and we precompose it with covering arrows $(v_i : p(d_i) \to c)_{i \in I}$ to obtain locally cartesian arrows $(\widehat{f_i} : d_i \to d)_{i \in I}$, are we able to represent $l(c)$, in $\widehat{\cal C}_J$, as a colimit coming from $p$.

Following the previous discussion, in order to obtain the object $l(c)$ as a colimit in the topos, we will use the notion of cofinality. Let $c$ be an object of a site $({\cal C},J)$. We know that $l(c) \simeq \colim l\pi_c$ where $\pi_c : {\cal C}/c \to {\cal C}$ is the obvious projection. We may want to represent the object $l(c)$ as the colimit of a diagram factorizing through ${\cal C}/c$, saying that we can fully characterize $c$ by not taking \emph{all} the arrows going into it, but just \emph{some} of those arrows, possibly indexed by another category. This situation is a particular case of the Proposition 2.21. \cite{denseness}, namely:

\begin{prop}
Let $c$ be an object of a site $({\cal C},J)$, and $\cal D$ a category with a functor $p : {\cal D} \to {\cal C}/c$. We have that $l(c) \simeq \colim l\pi_c \simeq \colim l\pi_cp$ exactly when these two conditions are satisfied:

\begin{enumerate}
    \item There is a set of arrows of the form $p(d_i)$ covering $c$.
    \item Whenever we have two arrows $x : c' \to \dom (p(d_i))$ and $x' : c' \to \dom (p(d_j))$ such that $p(d_i)x = p(d_j)x'$, there is a covering $x_i : c'_i \to c'$ such that $xx_i$ and $x'x_i$ are in the same connected component of $(c'_i / \pi_cp)$.
\end{enumerate}
\end{prop}

\begin{proof}
All the simplifications are due to the fact that, in $(x / \pi_c)$, two objects $v : x \to \pi_c([f : c' \to c])$ and $v' : x \to \pi_c([g : c'' \to c])$ are in the same connected component if and only if $gv'=fv$. 
\end{proof}

Here we can apply this result when we have an arrow $f : c \to p(d)$ in a local fibration, in order to present $l(c)$ as a colimit of objects coming from $\cal D$:

\begin{prop}\label{prop:cofinalitycolimit}
Let $p : ({\cal D},K) \to ({\cal C},J)$ be a continuous local fibration. Recall that we have the following commutative triangle:

\[\begin{tikzcd}
	{{\cal C}/c} & {\cal C} \\
	{{\cal D}^{\mathbf{fact}}_{f,d}}
	\arrow["{\pi_c}", from=1-1, to=1-2]
	\arrow["{p^f_d}"', from=2-1, to=1-2]
	\arrow["{\pi^f_d}", from=2-1, to=1-1]
\end{tikzcd}\]
We have that $\pi^f_d$ is  $J$-cofinal, that is: $l(c) \simeq \colim lp^f_d$.
\end{prop}

\begin{proof}
We have to verify the two conditions recalled in of the previous proposition.

(i): Let $x : \overline{c} \to c' = \pi_c(g : c'  \to c)$. We can cover $c$ by the $(v_i: p(d_i) \to c)_i$ from where the first point of the previous proposition is fulfilled.

(ii): Let $(d',f',v')$ and $(d'',f'',v'')$ be two objects of ${\cal D}^{\mathbf{fact}}_{f,d}$, and two arrows $u' : \overline{c} \to p(d')$, $u'' : \overline{c} \to p(d'')$ such that $v'u' = v''u''$. We can locally lift the arrow $fv'u' = fv''u''$ into locally cartesian arrows, as in the following diagram:

\[\begin{tikzcd}
	&& {p(d')} \\
	{p(\overline{d_i})} & {\overline{c}} && c & {p(d)} \\
	&& {p(d'')}
	\arrow["{x'}", from=2-2, to=1-3]
	\arrow["{x''}"', from=2-2, to=3-3]
	\arrow["{v'}", from=1-3, to=2-4]
	\arrow["{v''}"', from=3-3, to=2-4]
	\arrow["f"', from=2-4, to=2-5]
	\arrow["{p(\widehat{g_i})}"', bend right = 12, from=2-1, to=2-5]
	\arrow["{u_i}", from=2-1, to=2-2]
\end{tikzcd}\]

Then, we can do the same for $x'u_i$ and $x''u_i$: we have coverings $(u'_{ij})_j$ and $(u''_{ij'})_{j'}$ such that $u_iu'_{ij} = p(\widehat{f'_{ij}})$ and $u_iu''_{ij'} = p(\widehat{f''_{ij'}})$, with the $\widehat{f'_{ij}}$ and $\widehat{f''_{ij'}}$ locally cartesian arrows. We explicit one of these two situations in the diagram:

\[\begin{tikzcd}
	&&& {p(d')} \\
	{p(d'_{ij})} & {p(\overline{d_i})} & {\overline{c}} && c & {p(d)}
	\arrow["{x'}", from=2-3, to=1-4]
	\arrow["{v'}", from=1-4, to=2-5]
	\arrow["f"', from=2-5, to=2-6]
	\arrow["{p(\widehat{g_i})}"', bend right = 12, from=2-2, to=2-6]
	\arrow["{u_i}", from=2-2, to=2-3]
	\arrow["{p(\widehat{f'_{ij}})}", bend left = 15, from=2-1, to=1-4]
	\arrow["{u'_{ij}}"', from=2-1, to=2-2]
	\arrow["{p(f')}", bend left = 15, from=1-4, to=2-6]
\end{tikzcd}\]

Now we use the locally cartesian property of $\widehat{g_i}$ to locally lift the $u'_{ij}$: we have coverings $(u'_{ijk} : d'_{ijk} \to d'_{ij})_k$ and arrows $(x'_{ijk} : d'_{ijk} \to \overline{d_i})_k$ such that $\widehat{g_i}x'_{ijk} = f'\widehat{f'_{ij}}u'_{ijk}$ and $u'_{ij}p(u'_{ijk}) = p(x'_{ijk})$. On the other side, we can do the same for the $u''_{ij'}$:  we have coverings $(u''_{ij'k'} : d'_{ij'k'} \to d''_{ij'})_k$ and arrows $(x''_{ij'k'} : d''_{ij'k'} \to \overline{d_i})_{k'}$ such that $\widehat{g_i}x''_{ij'k'} = f''\widehat{f''_{ij'}}u''_{ij'k'}$ and $u''_{ij'}p(u''_{ij'k'}) = p(x''_{ij'k'})$. Again, one of these two symmetric situations is depicted in the following diagram:

\[\begin{tikzcd}
	&&& {p(d')} \\
	{p(d'_{ij})} & {p(\overline{d_i})} & {\overline{c}} && c & {p(d)} \\
	{p(d'_{ijk})}
	\arrow["{x'}", from=2-3, to=1-4]
	\arrow["{v'}", from=1-4, to=2-5]
	\arrow["f"', from=2-5, to=2-6]
	\arrow["{p(\widehat{g_i})}"',  bend right = 12, from=2-2, to=2-6]
	\arrow["{u_i}", from=2-2, to=2-3]
	\arrow["{p(\widehat{f'_{ij}})}", bend left = 15, from=2-1, to=1-4]
	\arrow["{u'_{ij}}"', from=2-1, to=2-2]
	\arrow["{p(f')}", bend left = 15, from=1-4, to=2-6]
	\arrow["{p(u'_{ijk})}", from=3-1, to=2-1]
	\arrow["{p(x'_{ijk})}"', from=3-1, to=2-2]
\end{tikzcd}\]

Since $p$ is continuous it preserves coverings, and we have that $(p(u'_{ijk}))_k$ and $(p(u''_{ij'k'}))_{k'}$ are covering. At this point, we have two different coverings which allow us to locally connect $(d',f',v')$ to $(\overline{d_i}, \widehat{g_i}, v'x'u_i)$ on one side, and $(d'',f'',v'')$ to $(\overline{d_i}, \widehat{g_i}, v'x'u_i)$ on the other side. Indeed, we have :

\[\begin{tikzcd}
	& {\overline{c}} & {p(d')=p^f_d(d',f',v')} \\
	{p(d'_{ijk})} && {p(d'_{ijk})=p^f_d(d'_{ijk},\widehat{g_i}x'_{ijk},v'x'u_ip(x'_{ijk}))} \\
	&& {p(\overline{d_i})=p^f_d(d_i,\widehat{g_i},v'x'u_i)}
	\arrow["{u_iu'_{ij}p(u'_{ijk})}", from=2-1, to=1-2]
	\arrow["{x'}", from=1-2, to=1-3]
	\arrow["{p(\widehat{f'_{ij}}u'_{ijk})=p^f_d(\widehat{f'_{ij}}u'_{ijk})}"', from=2-3, to=1-3]
	\arrow["{p(x'_{ijk})=p^f_d(x'_{ijk})}", from=2-3, to=3-3]
	\arrow["{p(x'_{ijk})}"', from=2-1, to=3-3]
	\arrow["{1_{p(d'_{ijk})}}", from=2-1, to=2-3]
\end{tikzcd}\]

an similarly for the situation with $x''$.

Finally, we can achieve to locally connect $(d',f',v')$ and $(d'',f'',v'')$ (using the $(\overline{d_i},\widehat{f_i},v_i)$ as intermediaries) by taking the intersection of the two coverings $(u'_{ij}p(u'_{ijk}))_{jk}$ and $(u''_{ij'}p(u''_{ij'k'}))_{j'k'}$.
\end{proof}

Actually, this cofinality condition (satisfied for a continuous local fibration) determines entirely when the essential image of a continuous comorphism of sites $\Sh(p)^*$ is a fibration:

\begin{prop}
Let $p : (\mathcal{C},J) \to (\mathcal{D},K)$ be a continuous comorphism of sites. Then, $\Sh(p)^* \simeq {C_p}_!$ is a fibration if and only if, for every $f : c \to p(d)$, the functor $\pi_d^f :{\cal D}_{f,d}^{\mathbf{fact}} \to {\cal C}/c$ is $J$-cofinal.
\end{prop}

\begin{proof}

As an immediate application of \ref{liftingwithunit}, for an arrow $f : c \to p(d)$, we have a cartesian lifting of $l(f)$ if and only if, in the following pullback:
\[\begin{tikzcd}
	F & {l(d)} \\
	{C_p^*l(c)} & {C_p^*lp(d)}
	\arrow["{\widehat{f}}", from=1-1, to=1-2]
	\arrow["q"', from=1-1, to=2-1]
	\arrow["\lrcorner"{anchor=center, pos=0.125}, draw=none, from=1-1, to=2-2]
	\arrow["{\eta(d)}", from=1-2, to=2-2]
	\arrow["{C_p^*l(f)}"', from=2-1, to=2-2]
\end{tikzcd}\]

\noindent we have that $q^t$ is an isomorphism. 

We can present $F$ as a colimit of generators, indexed by its category of elements. In order to have a manageable presentation, we can do it at the presheaf-level (before applying the sheafification): evaluated at an object $\overline{d}$, we have: $F(\overline{d}) = \{ (g : \overline{d} \to d,v : p(\overline{d}) \to c) | fv=p(g) \}$, so that $\int F= {\cal D}_{f,d}^{\mathbf{fact}}$. Since $\Sh(p)^*$ preserves colimits, we have $\Sh(p)^*(F) \simeq \colim_{\int F} lp(d)$, and the arrow $q^t : \Sh(p)^*(F) \to l(c)$ is given by the canonical arrow induced by the obvious universal property of the colimit: this arrow being an isomorphism is equivalent to $\pi^f_d$ being $J$-cofinal; so that this cofinality condition is equivalent to $\Sh(p)^*$ admitting cartesian liftings for arrows coming from the site $(\cal D,K)$.

Actually, we can deduce the existence of a cartesian lifting for an arbitrary arrow $f : E \to \Sh(p)^*(F)$ from this to hold on the arrows of the form $l(f) : l(c) \to lp(d)$. Indeed, the transpose $q^t$ of the pullback of $\eta(F)$ along $C_p^*(f)$ being an isomorphism can be reformulated as: the following square commutes when applied on the object $[\eta(F)]$ (cf. \ref{frobrempourrelevcart}):

\[\begin{tikzcd}
	{S_{C_p}(\Sh(p)^*F)} & {S_{C_p}(E)} \\
	{S_{\widehat{\cal C}_J}(\Sh(p)^*F)} & {S_{\widehat{\cal C}_J}(E)}
	\arrow["{S_{C_p}(f)}", from=1-1, to=1-2]
	\arrow["{(\Sh(p)^*)_{\Sh(p)^*(F)}}"', shift left=2, from=1-1, to=2-1]
	\arrow["{(\Sh(p)^*)_E}", from=1-2, to=2-2]
	\arrow["{S_{\widehat{\cal C}_J}(f)}"', from=2-1, to=2-2]
\end{tikzcd}\]

It exactly asks that the naturality squares occurring in the definition of $C_p$ being a locally connected morphism commute for each arrow in the basis of the form $f : E \to \Sh(p)^*(F)$ when evaluated on the object $[\eta(F)]$. In \cite{denseness} Proposition 4.53. it is shown that we can reduce to check that naturality squares of this kind commutes on the generators. Recalling that $\eta(F)$ is a colimit of generators, we can use the same arguments, but only using the fact that these naturality squares commute on the required $\eta(l(d))$, to deduce that the naturality square related to $f : E \to \Sh(p)^*(F)$ commutes on $\eta(F)$.
\end{proof}

\subsection{Totally connected morphisms as fibrations}
As shown in subsection \ref{locallyconnectedsection}, locally connected morphisms can be characterized as those essential geometric morphisms whose essential image is a fibration. This provides the following new characterization of totally connected morphisms, derived from our reformulation of the notion of locally connected morphism:

\begin{prop}
A geometric morphism $f : \cal F \to \cal E$ is totally connected if and only if it is essential and its essential image $f_! : \cal F \to \cal E$ is a cartesian fibration.
\end{prop}
\qed

\begin{remark}
A totally connected topos comes with a canonical point. Indeed, for $t : \cal T \to \cal E$ a totally connected topos, the fact that its essential image $t_!$ is cartesian provides a pair of geometric morphisms: obviously $(t^* \dashv t_*) : \cal T \to \cal E$, but also $(t_! \dashv t^*) : \cal E \to \cal T$ which we will denote by $p_t$. The fact that $t_!$ is a cartesian fibration gives us that $t_!t^* \simeq 1_{\cal E}$: one can easily check that the right adjoint of the projection functor of a cartesian fibration is given by sending an object $E$ of the base category to the terminal object of the fiber over $E$, so that their composition is the identity. Hence, those two geometric morphisms gives an adjoint retraction in the bicategory $\mathbf{Top}$ of toposes: we have $tp_t \simeq 1_{\cal E}$ and $p_t \vdash t$. We have a canonical $\cal E$-point of the relative topos $[t]$, as pictured:

\[\begin{tikzcd}
	{\cal E} && {\cal T} \\
	& {\cal E}
	\arrow[""{name=0, anchor=center, inner sep=0}, "{p_t}", from=1-1, to=1-3]
	\arrow["{1_{\cal E}}"', from=1-1, to=2-2]
	\arrow["t", from=1-3, to=2-2]
	\arrow["\simeq"{description}, draw=none, from=0, to=2-2]
\end{tikzcd}\]
\end{remark}

We have some interesting examples of totally connected morphisms arising from our study:

\begin{prop}
Let $p : \cal G(\mathbb C) \to (\cal C,J)$ be a cartesian fibration over a cartesian site $(\cal C,J)$. Its giraud topos $C_p : Gir_J(\mathbb C) \to \widehat{\cal C}_J$ is a totally connected relative topos. It has for canonical $\widehat{\cal C}_J$-point the geometric morphism $\Sh(p):\widehat{\cal C}_J \to Gir_J(\mathbb C)$.
\end{prop}

\begin{proof}
We know from \ref{fibimplfibtopos} that $C_p$ is locally connected (the essential image ${C_p}_!$ is a fibration); but since the fibration is cartesian, the projection functor $p$ is finite-limits preserving, hence ${C_p}_!$ also preserves finite limits: $C_p$ is totally connected. Since $p$ moreover preserves coverings (the topology on $\cal G(\mathbb C)$ is the Giraud one), it is not only a comorphism but also a morphism of sites: it induces two adjoint geometric morphisms $\Sh(p) \vdash C_p$. 
\end{proof}

For another example of totally connected relative topos, we look at the canonical stack of a relative topos: given a relative topos $f : \cal F \to \cal E$, the canonical geometric morphism $\rho_f : (\cal F/f^*) \to \cal E$ of \ref{rhodef} is locally connected since we know its essential image to be the projection $\pi_f : (\cal F/f^*) \to \cal E$, which is a fibration. Moreover this morphism is actually \emph{totally connected}, as an immediate consequence of $\pi_f : (\cal F/f^*) \to \cal E$ preserving finite limits:

\begin{prop}
Let $f : \cal F \to \cal E$ be a geometric morphism. Then the relative topos $\rho_f : (\cal F/f^*) \to \cal E$ defined by its canonical stack is a \emph{totally connected} morphism.
\end{prop}
\qed

\begin{remark}
In light of the previous proposition, a simple yet informative example of a totally connected topos is the following: given a topological space $X$, there exists a unique geometric morphism $\mathbf{!}_X : \Sh(X) \to \mathbf{Set}$. The canonical stack associated with this geometric morphism is the comma category $(\Sh(X)/\mathbf{!}_X^*)$. One can show that this new topos is equivalent to the topos of sheaves on a topological space $\tilde{X}$ defined as follows: the points of $\tilde{X}$ are those of $X$ together with an additional point $\epsilon$; an open subset of $\tilde{X}$ is either the empty set or a set of the form $V \cup { \epsilon }$, where $V$ is an open subset of $X$. This topological space has a unique dense point: $\epsilon$ belongs to every one of its open sets. One may intuitively think of this point as connecting all points of the space, as it is \og close \fg to all of them in the topological sense.
\end{remark}

We just saw that, to any relative topos $[f : \cal F \to \cal E]$, we can associate a totally connected topos $[\rho_f : (\cal F/f^*) \to \cal E]$. This assignment is moreover functorial: to each morphism of relative toposes $g : [f] \to [f'] $, we saw that $\widetilde{g^*} : [\rho_f] \to [\rho_{f'}]$ also was a morphism of relative toposes; it is immediate to see that this construction respects the identity and compositions, also with natural transformations. 

In \ref{etaextensionproperties}, we showed that for two relative toposes $f : \cal F \to \cal E$ and $f' : \cal F' \to \cal E$, and a geometric morphism $g : \cal F \to \cal F'$, the morphism $g$ is a morphism of \emph{relative} toposes if and only if $\widetilde{g^*}^* : (\cal F'/f'^*) \to (\cal F/f^*)$ is a morphism of \emph{fibrations}. This suggests that the appropriate notion of morphism between totally connected relative toposes should involve a \emph{morphism of fibrations} component.

We may thus define a \og free totally connected relative topos \fg functor
$S_{-} : \mathbf{Top}/\cal E \to \mathbf{TotConn}/\cal E$,
where the target category has as objects the totally connected geometric morphisms $t : \cal T \to \cal E$, and as morphisms the relative geometric morphisms $g : [t] \to [t’]$ such that $g^* : [t'_!] \to [t_!]$ is also a morphism of \emph{fibrations}.

The functor $S_{-}^{\cal E}$ assigns to each relative topos $[f]$ the totally connected topos $[\rho_f]$ given by its canonical relative stack, and sends a relative geometric morphism $g : [f] \to [f']$ to the $\eta$-extension of its inverse image:
$\widetilde{g^*} : (\cal F/f^*) \to (\cal F'/f'^*)$.

In picture, for two relative toposes and a morphism between them: 

\[\begin{tikzcd}
	{\cal F} && {\cal F'} \\
	& {\cal E}
	\arrow["g", from=1-1, to=1-3]
	\arrow["f"', from=1-1, to=2-2]
	\arrow[""{name=0, anchor=center, inner sep=0}, "{f'}", from=1-3, to=2-2]
	\arrow["\simeq"{description, pos=0.6}, draw=none, from=1-1, to=0]
\end{tikzcd}\]

\noindent we associate the following two totally connected toposes, together with the induced morphism of totally connected toposes between them:

\[\begin{tikzcd}
	{(\cal F/f^*)} && {(\cal F'/f'^*)} \\
	& {\cal E}
	\arrow["{\widetilde{g^*}}", from=1-1, to=1-3]
	\arrow["{\rho_f}"', from=1-1, to=2-2]
	\arrow[""{name=0, anchor=center, inner sep=0}, "{\rho_{f'}}", from=1-3, to=2-2]
	\arrow["\simeq"{description, pos=0.6}, draw=none, from=1-1, to=0]
\end{tikzcd}\]

\begin{prop}
Let $\cal E$ be a base topos. The functor $S_{-} : {\mathbf{Top}/\cal E} \to {\mathbf{TotConn}/\cal E}$ previously defined is left adjoint of the obvious forgetful functor $U_{\cal E}^{\mathbf{TotConn}} : {\mathbf{TotConn}/\cal E} \to {\mathbf{Top}/\cal E}$, as pictured:

\[\begin{tikzcd}
	{\mathbf{Top}/\cal E} && {\mathbf{TotConn}/\cal E}
	\arrow[""{name=0, anchor=center, inner sep=0}, "{S_{-}^{\cal E}}", bend left = 18, from=1-1, to=1-3]
	\arrow[""{name=1, anchor=center, inner sep=0}, "{U_{\cal E}^{\mathbf{TotConn}}}", bend left = 18, from=1-3, to=1-1]
	\arrow["\dashv"{anchor=center, rotate=-90}, draw=none, from=0, to=1]
\end{tikzcd}\]

The unit is given, at a relative topos $[f : \cal F \to \cal E]$, by the geometric embedding $j_f : \cal F \hookrightarrow (\cal F/f^*)$. The counit is given, at a totally connected relative topos $t : \cal T \to \cal E$, by the morphism of totally connected toposes $\epsilon_t : (\cal T/t^*) \twoheadrightarrow \cal T$, the direct and inverse images of which are respectively the obvious projection and the fully faithful morphism of fibrations $\eta_t$ (hence, it is moreover connected).
\end{prop}

\begin{proof}
Let $t : \cal T \to \cal E$ be a totally connected relative topos. The fact that  $\eta_t : \cal T \to (\cal T/t^*)$ is left adjoint to the projection is immediate, in the light of the adjunction $t^* \vdash t_!$; also, it preserves finite limits since they are computed componentwise in $(\cal T/t^*)$, and it acts on the first component as the identity and on the second as $t_!$ (which is finite-limit preserving by definition). Moreover, this inverse image $\epsilon_t^* \simeq \eta_t : \cal T \to (\cal T/t^*)$ is indeed a morphism of fibrations, because of \ref{cartesianwithunit} (actually, it not only preserves but also reflect cartesian arrows); so that $\epsilon_t$ is a morphism of totally connected relative toposes. The naturality of this counit is immediate, as the naturality of the unit $j$ (\ref{characcanonicalstack}).

Now let $t : \cal T \to \cal E$ be a totally connected relative topos and $f : \cal F \to \cal E$ a relative topos. We want to build the equivalence:

\[\begin{tikzcd}
	{\mathbf{Top}/\cal E([f],[t])} && {\mathbf{TotConn}/\cal E([\rho_f],[t])}
	\arrow["\simeq"{description}, draw=none, from=1-1, to=1-3]
\end{tikzcd}\]

For the first part of it, let $g : [f] \to [t]$ be a morphism of relative toposes, as pictured:

\[\begin{tikzcd}
	{\cal F} && {\cal T} \\
	& {\cal E}
	\arrow["g", from=1-1, to=1-3]
	\arrow["f"', from=1-1, to=2-2]
	\arrow[""{name=0, anchor=center, inner sep=0}, "t", from=1-3, to=2-2]
	\arrow["\simeq"{description, pos=0.6}, draw=none, from=1-1, to=0]
\end{tikzcd}\]

We can take its $\eta$-extension: $\widetilde{g^*} : (\cal F/f^*) \to (\cal T/t^*)$ and postcompose it with the counit $\widetilde{g^*}\epsilon_t : (\cal F/f^*) \to \cal T$ as pictured:

\[\begin{tikzcd}
	{(\cal F/f^*)} & {(\cal T/t^*)} & {\cal T} \\
	& {\cal E}
	\arrow["{\widetilde{g^*}}", from=1-1, to=1-2]
	\arrow["{\rho_f}"', from=1-1, to=2-2]
	\arrow["{\epsilon_t}", from=1-2, to=1-3]
	\arrow["{\rho_t}"{pos=0.3}, from=1-2, to=2-2]
	\arrow["t", from=1-3, to=2-2]
\end{tikzcd}\]

This composition is a morphism of totally connected toposes as: the diagram commutes because $g$ is a relative geometric morphism if and only if $\widetilde{g^*}$ also is \ref{characcanonicalstack}, and the inverse image of $\epsilon_t\widetilde{g^*}$ is $\widetilde{g^*}^*\eta_t$ which is a morphism of fibrations as $\eta_t$ is and $\widetilde{g^*}^*$ is if and only if $g$ is a morphism of relative toposes \ref{etaextensionproperties}.

For the other part of the equivalence: let $g' : [\rho_f] \to [t]$ be a morphism of totally connected toposes. We just have to precompose it with the unit $j_f$ in order to obtain a morphism of relative toposes:

\[\begin{tikzcd}
	{\cal F} & {(\cal F/f^*)} & {\cal T} \\
	& {\cal E}
	\arrow["{j_f}", hook, from=1-1, to=1-2]
	\arrow["f"', from=1-1, to=2-2]
	\arrow["{g'}", from=1-2, to=1-3]
	\arrow["{\rho_f}"{description}, from=1-2, to=2-2]
	\arrow["t", from=1-3, to=2-2]
\end{tikzcd}\]

There remains to prove that these two constructions are pseudoinverse to each other. If we start with a relative geometric morphism $g : [f] \to [t]$, then we apply the first construction to obtain $\epsilon_t\widetilde{g^*}$, and we apply to it the second construction so that we finally have $\epsilon_t\widetilde{g^*}j_f$. Hence, we want to show that the following square of geometric morpshims is commutative:

\[\begin{tikzcd}
	{(\cal F/f^*)} && {(\cal T/t^*)} \\
	{\cal F} && {\cal T}
	\arrow["{\widetilde{g^*}}", from=1-1, to=1-3]
	\arrow["{\epsilon_t}", from=1-3, to=2-3]
	\arrow["{j_f}", hook, from=2-1, to=1-1]
	\arrow["g"{description}, from=2-1, to=2-3]
\end{tikzcd}\]

When we look at the inverse images, this square of geometric morphisms is given by:

\[\begin{tikzcd}
	{(\cal F/f^*)} && {(\cal T/t^*)} \\
	{\cal F} && {\cal T}
	\arrow["{\pi_{\cal F}}"', from=1-1, to=2-1]
	\arrow["{\widetilde{g^*}^*}"', from=1-3, to=1-1]
	\arrow["{\eta_t}"', from=2-3, to=1-3]
	\arrow["{g^*}"{description}, from=2-3, to=2-1]
\end{tikzcd}\]

\noindent so that an easy computation gives its commutativity.

In the other direction, let $g' : [\rho_f] \to [t]$ be a morphism of totally connected toposes. We apply the second construction, that is, we precompose it with the unit to obtain a morphism of relative toposes $g'j_f : [f] \to [t]$. Then, we recover a morphism of totally connected relative toposes by applying the second construction: $\epsilon_t\widetilde{g'^*}\widetilde{j_f^*}$, where we used the functoriality of the $\eta$-extension (functoriality of the left adjoint $S_{-}^{\cal E}$). 

At this point, it is easy to notice that $\epsilon_tj_t \simeq 1_{\cal T}$. Moreover, it is lengthy but straightforward to show that $\widetilde{j_f^*} \simeq j_{\rho_f}$. The situation is the following:

\[\begin{tikzcd}
	& {\cal T} \\
	{\cal T} & {(\cal T/t^*)} \\
	{(\cal F/f^*)} & {((\cal F/f^*)/\tau_f)}
	\arrow["{j_t}", hook, from=2-1, to=2-2]
	\arrow["{\epsilon_t}"', two heads, from=2-2, to=1-2]
	\arrow["{g'}", from=3-1, to=2-1]
	\arrow["{\widetilde{j_f^*}\simeq j_{\rho_f}}"', hook, from=3-1, to=3-2]
	\arrow["{\widetilde{g'^*}}"', from=3-2, to=2-2]
\end{tikzcd}\]

By the fact that $\widetilde{j_f^*} \simeq j_{\rho_f}$ together with the naturality of $j$, we know that the square commutes; so that, when we postcompose with $\epsilon_t$, we have $\epsilon_t\widetilde{g'^*}\widetilde{j_f^*} \simeq \epsilon_tj_tg'$. Since $\epsilon_tj_t \simeq 1_{\cal T}$, we finally obtain $\epsilon_t\widetilde{g'^*}\widetilde{j_f^*} \simeq g'$, so that these two constructions indeed are pseudoinverse parts of an equivalence. The proof of the naturality of these constructions is left to the reader.
\end{proof}

It is well known that every geometric morphism can be factorized as a closed embedding followed by a totally connected morphism: it is the factorization of $f : \cal F \to \cal E$ as $\cal F \xrightarrow{j_f} (\cal F/f^*) \xrightarrow{\rho_f} \cal E$. Here, we notice that this factorization enjoys moreover the exceptional following property:

\begin{prop}
Let $f,f' : \cal F \to \cal E$ be two geometric morpshisms and $\alpha : f \Rightarrow f'$ a geometric transformation between them, as pictured:

\[\begin{tikzcd}
	{\cal F} && {\cal E}
	\arrow[""{name=0, anchor=center, inner sep=0}, "f", bend left = 22, from=1-1, to=1-3]
	\arrow[""{name=1, anchor=center, inner sep=0}, "{f'}"', bend right = 22, from=1-1, to=1-3]
	\arrow["\alpha", shorten <=3pt, shorten >=3pt, Rightarrow, from=0, to=1]
\end{tikzcd}\]

This situation induces a geometric morphism $\cal F/\alpha : (\cal F/f^*) \to (\cal F/f'^*)$ such that $\cal F/\alpha \circ j_f \simeq j_{f}$, together with a geometric transformation $\overline{\alpha} : \rho_f \circ \cal F/\alpha \Rightarrow \rho_{f'}$, as depicted in:

\[\begin{tikzcd}
	& {(\cal F/f^*)} \\
	{\cal F} && {\cal E} \\
	& {(\cal F/f'^*)}
	\arrow[""{name=0, anchor=center, inner sep=0}, "{\rho_f}", two heads, from=1-2, to=2-3]
	\arrow[""{name=1, anchor=center, inner sep=0}, "{\cal F/\alpha}"{description, pos=0.4}, from=1-2, to=3-2]
	\arrow[""{name=2, anchor=center, inner sep=0}, "{j_f}", hook, from=2-1, to=1-2]
	\arrow["{j_{f'}}"', hook', from=2-1, to=3-2]
	\arrow["{\rho_{f'}}"', two heads, from=3-2, to=2-3]
	\arrow["\simeq"', draw=none, from=2, to=1]
	\arrow["{\overline{\alpha}}"{description}, shorten <=6pt, shorten >=6pt, Rightarrow, from=3-2, to=0]
\end{tikzcd}\]
\end{prop}

\begin{proof}
Indeed, the geometric morphism $\cal F/\alpha$ has for inverse image the postcomposition by $\alpha_E$: to an object $(F,E,u:F\to f^*(E))$ of $(\cal F/f^*)$ we associate $(F,E,\alpha_Eu : F\to f'^*(E)$; and for direct image the pullback along $\alpha_E$: to an object $(F',E,u':F\to f'^*(E))$ we associate the object $(F,E,u:F\to f^*(E))$ depicted in the following pullback:

\[\begin{tikzcd}
	F & {F'} \\
	{f^*(E)} & {f'^*(E)}
	\arrow[from=1-1, to=1-2]
	\arrow["u"', from=1-1, to=2-1]
	\arrow["\lrcorner"{anchor=center, pos=0.125}, draw=none, from=1-1, to=2-2]
	\arrow["{u'}", from=1-2, to=2-2]
	\arrow["{\alpha_E}"', from=2-1, to=2-2]
\end{tikzcd}\]

The commutation of the left triangle of geometric morphisms of the proposition is immediate, also for the existence of the geometric transformation $\overline{\alpha}$ which is directly derived from $\alpha$.
\end{proof}

\section{Indexed weak Diaconescu's theorem}\label{section6}

In this section, we generalize the relative Diaconescu theorem from \cite{bartolicaramello}. In the absolute setting, continuous functors $A : ({\cal C},J) \to {\cal E}$ correspond to weak geometric morphisms $\Sh(A)_* \dashv \Sh(A)^* : \widehat{\cal C}_J \to {\cal E}$ — that is, to mere pairs of adjoint functors. The relative setting, however, brings to light two distinct notions of weak geometric morphism: \emph{relative weak geometric morphisms} and \emph{indexed weak geometric morphisms}. These two definitions call for further clarification, which we do in the first subsection.

Our main focus will be on \emph{indexed weak geometric morphisms}, but it is worth noting that \emph{relative weak geometric morphisms} can also be described as particular morphisms of opfibrations — a description we provide in the second subsection.

In the third subsection, we examine how \emph{indexed weak geometric morphisms} can be induced from local fibrations when used as sites of presentation; these will correspond to continuous morphisms of local fibrations. This construction yields a weak indexed version of Diaconescu’s theorem, first for local fibrations and then for ordinary fibrations. Finally, we apply this theorem to fibrations and deduce that the Giraud topos of a fibration is canonically equivalent to the one of its stackification.

\subsection{Distinction between indexed and relative weak geometric morphisms}

There are two points of view that coincide in the case of geometric morphisms, but differ in the context of weak geometric morphisms: the indexed setting and the relative setting. We begin by presenting them in the case of geometric morphisms. Given two relative toposes $f : {\cal E} \to {\cal F}$ and $f' : {\cal E'} \to {\cal F}$, we can define two associated categories:

\begin{itemize}
    \item The category having for objects the geometric morphisms $g : {\cal E} \to {\cal E'}$ such that $f'g \simeq f$, i.e. \emph{relative geometric morphisms}.
    \item The category having for objects the indexed adjunctions (indexed functors being component-wise adjoint functors) $(g_*) \vdash (g^*) : S_{f'} \to S_{f}$ between the canonical stacks of $f$ and $f'$ such that the fiberwise components $(g^*)_F$ preserve finite limits, i.e. \emph{indexed geometric morphisms}.
\end{itemize}

These two categories are equivalent: if we have an indexed geometric morphism we just take the adjoint pair at the terminal object $((g_*)_{\mathbf{1}_{\cal F}} \vdash (g^*)_{\mathbf{1}_{\cal F}})$ to obtain a relative geometric morphism,

\[\begin{tikzcd}
	{S_f} & {S_{f'}} && {{\cal E}/f^*(\mathbf{1}_{\cal F}) \simeq {\cal E}} & {{\cal E}' \simeq {\cal E'}/f'^*(\mathbf{1}_{\cal F})}
	\arrow[""{name=0, anchor=center, inner sep=0}, "{(g_*)}", shift left=2, from=1-1, to=1-2]
	\arrow[""{name=1, anchor=center, inner sep=0}, "{(g^*)}", shift left=2, from=1-2, to=1-1]
	\arrow[squiggly, maps to, from=1-2, to=1-4]
	\arrow[""{name=2, anchor=center, inner sep=0}, "{(g_*)_{\mathbf{1}_{\cal F}}}", shift left=2, from=1-4, to=1-5]
	\arrow[""{name=3, anchor=center, inner sep=0}, "{(g^*)_{\mathbf{1}_{\cal F}}}", shift left=2, from=1-5, to=1-4]
	\arrow["\dashv"{anchor=center, rotate=90}, draw=none, from=1, to=0]
	\arrow["\dashv"{anchor=center, rotate=90}, draw=none, from=3, to=2]
\end{tikzcd}\]

and if we have a relative geometric morphism, its $\eta$-extension will be an indexed left adjoint, giving us:

\[\begin{tikzcd}
	{\cal E} && {\cal E'} && {{\cal E}/{\cal F}} && {{\cal E'}/{\cal F}} \\
	& {\cal F} &&&& {\cal F}
	\arrow["g", from=1-1, to=1-3]
	\arrow["f"', from=1-1, to=2-2]
	\arrow[squiggly, maps to, from=1-3, to=1-5]
	\arrow["{f'}", from=1-3, to=2-2]
	\arrow[""{name=0, anchor=center, inner sep=0}, "{\widetilde{g^*}_*}"', shift right=2, from=1-5, to=1-7]
	\arrow["{\pi_f}"', from=1-5, to=2-6]
	\arrow[""{name=1, anchor=center, inner sep=0}, "{\widetilde{g^*}^*}"', shift right=3, from=1-7, to=1-5]
	\arrow["{\pi_{f'}}", from=1-7, to=2-6]
	\arrow["\dashv"{anchor=center, rotate=-90}, draw=none, from=1, to=0]
\end{tikzcd}\]

The fact that the inverse image of a geometric morphism preserves finite limits is crucial in the construction of this equivalence between \emph{relative} and \emph{indexed} geometric morphisms: the preservation of finite limits by the inverse image forces its extension to preserve cartesian arrows, as they are represented by cartesian squares. 

Now for the weak geometric morphisms, we recall the absolute definition from Definition 4.1.\cite{denseness}:

\begin{defn}
A weak geometric morphism  $f : {\cal E} \to {\cal E'}$  consists of an adjoint pair $f_* \vdash f^* : {\cal E'} \to {\cal E}$ (the inverse image does not need to preserve finite limits).
\end{defn}

In general, we define the two following different notions of weak geometric morphisms:

\begin{defn}\label{relativeandindexed}
Let $f : \cal F \to \cal E$ and $f' : \cal F' \to \cal E$ be two relative toposes. We call:

\begin{enumerate}
    \item a weak geometric morphism $a : \cal F \to \cal F'$ a \emph{relative weak geometric morphism} if we have a commutative triangle $f'a \simeq f$.
    \item a pair of \emph{indexed adjoint functors} $ (a_*) \vdash (a^*) : S_{f'} \to S_f$ an \emph{indexed weak geometric morphism}. 
\end{enumerate}
\end{defn}

As we will see, this (absolute) definition of weak geometric morphisms is a definition of (absolute) indexed weak geometric morphism; even for toposes over $\mathbf{Set}$, we will see that the notions of indexed and relative weak geometric morphisms already do not coincide. Indeed, even if we do not require that a relative weak geometric morphism commutes with any kind of limits, the preservation of the terminal object is automatic, since $a^*f^*(\mathbf{1}_{\cal F}) = f'^*(\mathbf{1}_{\cal F})$, and $f^*$ and $f'^*$ are preserving it. Any \emph{relative} weak geometric morphism is forced to preserve the finite limits coming from the basis, which is not the case of the absolute weak geometric morphisms as defined previously: they correspond to the indexed weak geometric morphisms (over $\mathbf{Set}$), as any functor between categories is always $\mathbf{Set}$-indexed. So the relevant notion for a weak Diaconescu's theorem over some base is not the notion of relative weak geometric morphism, but the one of indexed weak geometric morphisms.

In order to understand these two different kinds of weak morphisms, we need some comments about the canonical stacks: for a relative topos $f: \cal E \to F$ we recall that its canonical stack is not only a fibration (for the action of pulling back), but also an opfibration: we denote by $\Sigma_g : {\cal E}/f^*F \to {\cal E}/f^*F'$, for an arrow $g : F \to F'$ in $\cal F$, the functor which acts by post-composition with $f^*(g)$. It is the left adjoint to the pullback, and it is hence making the canonical stack an opfibration. For example, in $\mathbf{Set}$ we have that the functor $\Sigma_{!_I} : \mathbf{Set}/I \to \mathbf{Set}$ sends an $I$-indexed set $(S_i)_I \to I$ to its coproduct indexed by $I$, that is $\coprod_IS_i$. This covariant action of arrows from the base-category has to be understood as providing some kind of coproduct shaped after the base category; we do not only have \emph{absolute} coproducts. In particular, being a morphism of opfibrations, that is, commuting with these functors $\Sigma_g$, can be interpreted as a commutation with these coproducts shaped after the basis. This is the case for indexed left adjoints for example: as absolute left adjoints commute with colimits, indexed left adjoints commute with colimits (component-wise) also with those functors $\Sigma_g$, that is, they also are morphisms of opfibrations.

\subsection{The case of relative weak geometric morphisms}

In this subsection, we focus on relative weak geometric morphisms and describe them as a particular kind of morphisms of opfibrations. The two following categories are the general setting in order to understand relative weak geometric morphisms:

\begin{defn}
Let $f : \cal E \to F$ and $f' : \cal E' \to F$ be two relative toposes over $\cal F$. We define:

\begin{enumerate}
    \item the category $\mathbf{LaxWmor}([f],[f'])$ of lax weak geometric morphisms consisting of pairs $(a, \phi)$ where $a:{\cal E}\to {\cal E}'$ is a weak morphism of toposes, and $\phi:a^{\ast}\circ f'^{\ast}\imp f^{\ast}$ a geometric transformation.
    \item the category $\mathbf{OpfibWeak}(S_{f'},S_{f})$ of fiberwise weak geometric morphisms (pairs of adjoint functors) such that the left adjoints yield a morphism of \emph{op}fibrations
\end{enumerate}
\end{defn}

These two categories are equivalent:

\begin{prop}
Let $f : \cal E \to F$ and $f' : \cal E' \to F$ be two relative toposes over $\cal F$. We have an equivalence of categories:

$$\mathbf{OpfibWeak}(S_{f'},S_{f}) \simeq \mathbf{LaxWmor}([f],[f'])$$.   
\end{prop}

\begin{proof}
The equivalence works as follows:

If we have such a pair $(a,\phi)$ we can build the morphism of opfibrations which is fiberwise given by $a_F^{\phi} : {\cal E'}/f'^*(F) \to {\cal E}/f^*(F)$ sending an object $[\alpha' : E' \to f'^*(F)]$ to $[\phi_F a^*(\alpha') : a^*(E) \to a^*f'^*(F) \to f^*(F)]$. The fact that it is indeed a morphism of opfibrations results from the naturality of $\phi$. Its right adjoint is given by sending an object $\alpha : E \to f^*(F)$ to the following pullback living in ${\cal E'}/f'^*(F)$:

\[\begin{tikzcd}
	{E'} && {a_*(E)} \\
	{f'^*(F)} & {a_*a^*f'^*(F)} & {a_*f^*(F)}
	\arrow[from=1-1, to=1-3]
	\arrow["{\alpha'}"', from=1-1, to=2-1]
	\arrow["\lrcorner"{anchor=center, pos=0.125}, draw=none, from=1-1, to=2-2]
	\arrow["{a_*(\alpha)}", from=1-3, to=2-3]
	\arrow["{\zeta_{f'^*(F)}^a}"', from=2-1, to=2-2]
	\arrow["{a_*(\phi_F)}"', from=2-2, to=2-3]
\end{tikzcd}\]

\noindent where $\zeta_{f'^*(F)}^a$ denotes the unit of $a_* \vdash a^*$ at the object $f'^*(F)$.

On the other hand, if we have a morphism of opfibrations from $S_{f'}$ to $S_f$, since ${\cal E}/ f^*(\mathbf{1}_{\cal F}) \simeq {\cal E}$, we can associate to every morphism of opfibrations $a$ the weak geometric morphism $(a_{\mathbf{1_{\cal F}}})_* \vdash (a_{\mathbf{1_{\cal F}}})^* : \cal E \to \cal E'$. The components of $\phi$ are obtained as the images of the terminal objects of the ${\cal E}/f^*(F)$ through the $(a_F)^*$.

When we take a pair $(a,\phi)$ and apply the two constructions, it is immediate that we recover $(a,\phi)$.

Now starting with a morphism of opfibrations $a$, taking $a_{\mathbf{1}_{\cal F}}$ and then getting back to the associated morphism of opfibrations, by the mere definitions we see that the adjunction induced at the fiber over $\mathbf{1}_{\cal F}$ is the the same as in the beginning. To see that it also coincides on the other fibers, we exploit the commutation of:

\[\begin{tikzcd}
	{{\cal E}/f^*F} && {{\cal E'}/f'^*F} \\
	{\cal E} && {\cal E'} \\
	\\
	\\
	\arrow["{\Sigma_{!_F}}"', from=1-1, to=2-1]
	\arrow["{\Sigma'_{!_F}}", from=1-3, to=2-3]
	\arrow["{((a_{\mathbf{1}_{\cal F}})^*)_F^*}", from=1-1, to=1-3]
	\arrow["{(a_{\mathbf{1}_{\cal F}})^*}"', from=2-1, to=2-3]
\end{tikzcd}\]

Indeed, we can recover the information by sending 

\[\begin{tikzcd}
	E & {f^*F} \\
	{f^*F}
	\arrow["\alpha"', from=1-1, to=2-1]
	\arrow["1", from=1-2, to=2-1]
	\arrow["\alpha", from=1-1, to=1-2]
\end{tikzcd}\]

from the top left to the bottom right of the square: this forces the image of $[\alpha]$ to coincide with the image of $[\alpha]$ by $(a^*)_F$.
\end{proof}

\begin{remark}

\begin{enumerate}[(i)]
    \item Note that in $\mathbf{OpfibWeak}(S_f, S_{f'})$, we do not assume the existence of an \emph{indexed} adjunction: we are merely considering componentwise pairs of adjoint functors. However, the fact that the inverse image is a morphism of opfibrations is equivalent to the direct image being a morphism of fibrations.
    \item The fact that we work with morphisms of opfibrations ensures that the component at the terminal object $\mathbf{1}$ commutes with arbitrary coproducts: we just need to use the commutation with the $\Sigma_{!_I}$ for the arrows $!_I : \coprod_I \mathbf{1} \to \mathbf{1}$.
    \item Recall that the colimits in the fibers ${\cal F}/f^*(F)$ of $S_f$ are computed as if they were in the terminal fiber, that is, $\cal F$; there is no choice of a structural arrow over $f^*(F)$ for the colimit, it is uniquely determined by the fact that it is a colimit.
\end{enumerate}
\end{remark}

We can specify this proposition in order to have a characterization of weak relative geometric morphisms. 

\begin{defn}
Let $f : \cal E \to F$ and $f' : \cal E' \to F$ be two relative toposes over $\cal F$. By $\mathbf{OpfibTerm}(S_{f'},S_{f})$ we denote the category of \emph{terminal-preserving morphisms of opfibrations}, that is fiberwise pairs of adjoint functors such that the left adjoints yield a morphism of \emph{op}fibrations preserving the terminal objects fiberwise. Also, we denote $\mathbf{Wmor}/{\cal F}([f],[f'])$ the category of weak morphisms of toposes between $[f]$ and $[f']$ over $\cal F$. 
\end{defn}

\begin{prop}
In the setting of the previous definition, we have an equivalence of categories: $\mathbf{OpfibTerm}(S_{f'},S_f) \simeq \mathbf{Wmor}/{\cal F}([f],[f'])$.
\end{prop}

\begin{proof}
This follows directly from the previous equivalence with relative weak morphisms. Given a pair $(a,\phi)$, the associated morphism of opfibrations maps the terminal object $[f'^*(F) \to f'^*(F)]$ to the object $\phi_F : a^*f'^*(F) \to f^*(F)$. It defines a relative weak morphism if and only if each $\phi_F$ is an isomorphism, which holds exactly when terminal objects are preserved fiberwise.
\end{proof}

We can restrict this equivalence even further, recovering the previously mentioned equivalence of Theorem B.3.1.5 in \cite{elephant}.

\begin{prop}
Let $f : \cal E \to F$ and $f' : \cal E' \to F$ be two relative toposes over $\cal F$. We have an equivelance between    $\mathbf{FibOpfibCart}(S_{f'},S_f) $ and $  \mathbf{Geom}/{\cal F}([f],[f'])$, where the $\mathbf{FibOpfibCart}$ stands for the fibrations and opfibrations morpisms which also preserve finite limits.
\end{prop}

\begin{proof}
Indeed, if we restrict the previous equivalence to geometric morphisms over the basis, we obtain that the $\eta$-extension is not only a morphism of opfibrations, but also a morphism of fibrations and finite limits preserving because of \ref{etaextensionproperties}. In the over direction, taking the component of a morphism of fibrations, opfibrations, and finite-limits preserving at $\mathbf{1_{\cal F}}$ ensures us that the adjunction we obtain is a geometric morphism.
\end{proof}

\subsection{Indexed weak Diaconescu's theorem}

Having discussed relative weak geometric morphisms, we now turn to the notion of \emph{indexed} weak geometric morphisms, which will serve as the framework for a weak version of Diaconescu’s theorem.

We begin by providing a site-characterization of continuous functors that induce indexed weak geometric morphisms, in the most general setting:

\begin{prop}\label{sitecharaccontinuousfunctorsimpliesmorphfib}
Let $p : ({\cal D},K) \to ({\cal C},J)$ and $p' : ({\cal D}',K') \to ({\cal C},J)$ be two comorphism of sites, and $A : ({\cal D},K) \to ({\cal D}',K')$ be a continuous functor together with $\phi :p'A \Rightarrow p$ a natural transformation. The pair $(A,\phi)$ induces an indexed weak geometric morphism having $\widetilde{A}^*$ for inverse image if and only if the two following conditions are satisfied:

\begin{enumerate}[(i)]
    \item For every arrow $f : c' \to c$ in $\cal C$, object $d$ of $\cal D$ together with an arrow $u : p(d) \to c$, and every pair $(u' : p'(d') \to c',g' : d' \to A(d))$ such that $fu'=u\phi_{d}p'(g')$, there exist: a covering $(f_i' : d'_i \to d')_i$ for $K'$ together with triplets $(x_i : d_i' \to A(\overline{d_i}),\overline{g_i} : \overline{d_i} \to d,\overline{u_i} : p(\overline{d_i}) \to c')$ such that: $u'p'(f_i') = \overline{u_i}\phi_{\overline{d_i}}p'(x_i)$ and $g'f_i' = A(\overline{g_i})x_i$.
    \item For every two triplets $(x_1 : d' \to A(\overline{d_1}),g_1 : \overline{d_1} \to d,\overline{u_1} : p(\overline{d_1}) \to c')$ and $(x_2 : d' \to A(\overline{d_2}),g_2 : \overline{d_2} \to d,\overline{u_2} : p(\overline{d_2}) \to c')$ such that: $\overline{u_1}\phi_{\overline{d_1}}p'(x_1) = \overline{u_2}\phi_{\overline{d_2}}p'(x_2)$ and $A(\overline{g_1})x_1 =A(\overline{g_2})x_2$, there exist a covering $(f_i' : d_i' \to d')_i$ for $K'$ such that, denoting $P^f_{(d,u)}$ the presheaf which is the pullack:

\[\begin{tikzcd}
	{P^f_{(d,u)}} & {{\cal D}(-,d)} \\
	{{\cal C}(p-,c)} & {{\cal C}(p-,c')}
	\arrow[from=1-1, to=1-2]
	\arrow[from=1-1, to=2-1]
	\arrow["\lrcorner"{anchor=center, pos=0.125}, draw=none, from=1-1, to=2-2]
	\arrow["{ev_u}", from=1-2, to=2-2]
	\arrow["{f \circ -}"', from=2-1, to=2-2]
\end{tikzcd}\]
    we have that  $(x_1f'_i : d'_i \to A(\overline{d_1}),g_1 : \overline{d_1} \to d,\overline{u_1} : p(\overline{d_1}) \to c')$ and $(x_2f_i' : d'_i \to A(\overline{d_2}),g_2 : \overline{d_2} \to d,\overline{u_2} : p(\overline{d_2}) \to c')$ are in the same connected component of $(d_i' / A\pi^f_{(d,u)})$ with $\pi^f_{(d,u)}$ being the projection functor $\pi^f_{(d,u)} : \int P^f_{(d,u)} \to {\cal D}$.
\end{enumerate}
\end{prop}

\begin{proof}

We want to check when $\widetilde{A}^*$ is a morphism of fibrations. That is, for every arrow $f : c' \to c$ in $\cal C$, we want to see if the following canonical natural transformation is an isomorphism, making the square commutative (up to a canonical iso):

\[\begin{tikzcd}
	{S_{C_p}(c)} & {S_{C_p}(c')} \\
	{S_{C_{p'}}(c)} & {S_{C_{p'}}(c')}
	\arrow["{S_{C_p}(f)}", from=1-1, to=1-2]
	\arrow["{\widetilde{A}^*_c}"', from=1-1, to=2-1]
	\arrow["{\nu_f}"{description}, Rightarrow, from=1-2, to=2-1]
	\arrow["{\widetilde{A}^*_{c'}}", from=1-2, to=2-2]
	\arrow["{S_{C_{p'}}(f)}"', from=2-1, to=2-2]
\end{tikzcd}\]

First, we notice that the functors involved in this square are all preserving colimits, so we can restrict to check that $\nu^f$ is an iso on the generators of $S_{C_p}(c) = \widehat{\cal D}_K/C_p^*l(c)$, namely the: $a_K(ev_u) : a_K({\cal D}(-,d)) \to a_K({\cal C}(p-,c))$ where $ev_u$ is given by the Yoneda lemma, evaluating the identity of $d$ at the arrow $u : p(d) \to c$. For a better understanding, we first study these arrows $\nu^f_{(d,u)}$ at the presheaf-level, and then we will apply the usual criteria to characterize exactly when it is an isomorphism after sheafification. 

So let's take the topologies to be the trivial ones, and compute the two paths of the square, in order to explicit the arrows $\nu^f_{(d,u)}$:

The upper path gives: first, we have the pullback 

\[\begin{tikzcd}
	{P^f_{(d,u)}} & {{\cal D}(-,d)} \\
	{{\cal C}(p-,c')} & {{\cal C}(p-,c)}
	\arrow[from=1-1, to=1-2]
	\arrow[from=1-1, to=2-1]
	\arrow["\lrcorner"{anchor=center, pos=0.125}, draw=none, from=1-1, to=2-2]
	\arrow["{ev_u}", from=1-2, to=2-2]
	\arrow["{f\circ -}"', from=2-1, to=2-2]
\end{tikzcd}\]

which, evaluated at an object $\overline{d}$ of $\cal D$, is given by $P^f_{(d,u)}(\overline{d}) = \{ (g : \overline{d} \to d, \overline{u}' : p(\overline{d}) \to c') ; up(g) = f \overline{u}' \}$. Then, we want to apply $\widetilde{A}^*_{c'}$. Since it preserves colimits, we can present $P^f_{(d,u)}$ as the colimit of its category of elements, and we obtain, after evaluation at an object $d'$ of $\cal D'$: $\widetilde{A}^*_{c'}(P^f_{(d,u)})(d') = \{ [x : d' \to A(\overline{d}), (g : \overline{d} \to d, \overline{u}' : p(\overline{d}) \to c')]; up(g) = f \overline{u}'\}$. Here, $[x : d' \to A(\overline{d}), (g : \overline{d} \to d, \overline{u}' : p(\overline{d}) \to c')]$ denotes the class of such an arrow in $\colim_{\int P^f_{(d,u)}}{\cal D}(d',A(\pi_{P^f_{(d,u)}}(g,\overline{u}')))$, with $\pi^f_{(d,u)} : \int P^f_{(d,u)} \to \cal D$ the canonical projection functor. For the sake of readability, we denote the presheaf thus obtained $U^f_{(d,u)}$.

The lower path gives: first, in the light of the good behavior of $\widetilde{A}^*_c$ on the generators, we have $\widetilde{A}^*_c(ev_u : {\cal D}(-,d) \to {\cal C}(p-,c)) = ev_{u\phi_d} : {\cal D'}(-,A(d)) \to {\cal C}(p'-,c)$. Then, we pull it back:

\[\begin{tikzcd}
	{P'^f_{(d,u)}} & {{\cal D'}(-,A(d))} \\
	{{\cal C}(p'-,c')} & {{\cal C}(p'-,c)}
	\arrow[from=1-1, to=1-2]
	\arrow[from=1-1, to=2-1]
	\arrow["\lrcorner"{anchor=center, pos=0.125}, draw=none, from=1-1, to=2-2]
	\arrow["{ev_{u\phi_d}}", from=1-2, to=2-2]
	\arrow["{f \circ -}"', from=2-1, to=2-2]
\end{tikzcd}\]

When evaluated at an object $d'$ of $\cal D'$, we have $P'^f_{(d,u)}(d') = \{ (x : d' \to A(d), u' : p'(d') \to c') ; u\phi_dp'(x) = fu' \}$. For the sake of readability, we denote the presheaf thus obtained $L^f_{(d,u)}$. 

Now, we can exhibit the canonical arrow $\nu^f_{(d,u)} : U^f_{(d,u)} \to L^f_{(d,u)}$: at an object $d'$ of $\cal D'$: it sends an element $[x : d' \to A(\overline{d}), (g : \overline{d} \to d, \overline{u}' : p(\overline{d}) \to c')]$ of $U^f_{(d,u)}(d')$ (so we have $up(g) = f \overline{u}'$) to the element $(A(g)x : d' \to A(d), \overline{u}'\phi_{\overline{d}}p'(x) : p'(d') \to c')$ of $L^f_{(d,u)}(d')$.

This morphism of presheaves being locally epi is the first condition of the proposition, and the second condition is the locally mono one.
\end{proof}

Using this characterization, we can show that, when working with local fibrations as presentations for our relative toposes, a continuous functor between them induces an indexed weak geometric morphism if and only if it is a morphism of local fibrations:

\begin{thm}\label{morphlocfibcont}
Let $p : ({\cal D},K) \to ({\cal C},J)$ and $p' : ({\cal D}',K') \to ({\cal C},J)$ be two local fibrations, and $A : ({\cal D},K) \to ({\cal D}',K')$ be a continuous functor such that $p'A \simeq p$. We have the equvalence:
\begin{enumerate}[(i)]
    \item The $\eta$-extension $\widetilde{A}^*$ is a morphism of fibrations.
    \item The pair previously defined $\widetilde{A}_* \vdash \widetilde{A}^*$ is an \emph{indexed} adjunction, that is $\widetilde{A}^*$ is the inverse image of a weak indexed geometric morphism. 
    \item The continuous functor $A$ is a morphism of local fibrations.
\end{enumerate}
\end{thm}

\begin{proof}

The equivalence between the two first points comes from the fact that $\widetilde{A}^*$ is a morphism of opfibrations, hence its right adjoint (which we know to exist) is automatically indexed. So $\widetilde{A}^*$ is an \emph{indexed} left adjoint if and only if it is a morphism of fibrations.

The implication $(i) \Rightarrow (iii)$ comes from the fact that $\widetilde{A}^*\eta_{\cal D} \simeq \eta_{\cal D'}A$. Indeed, if $\widetilde{A}^*$ is a morphism of fibrations and we take a locally cartesian arrow $f$ in $\cal D$, we have by definition that $\eta_{\cal D}(f)$ is cartesian, hence $\widetilde{A}^*\eta_{\cal D}(f)$ is cartesian, from where $\eta_{\cal D'}A(f)$ is cartesian, that is, by definition, $A(f)$ is locally cartesian.

Let's prove $(iii) \Rightarrow (i)$. In order to have lighter notations, we assume the isomorphism $\phi : p'A \simeq p$ to be the equality. There is no substantial difference when the same proof is applied with $\phi$ a non-trivial iso. We will use the precedent characterization, and the notations of its proof:

Let's first show that it is locally epi. We take an element $(x : d' \to A(d), u' : p'(d') \to c')$ of $L^f_{(d,u)}(d')$, so we have the equality $up'(x) = fu'$. We will locally bring this element to elements coming from $U^f_{(d,u)}$. To do so, we send $x$ in the basis and, since $p'A = p$, we have that $p'(x) : p'(d') \to p(d)$ has a codomain coming from $p$. By the fact that $p$ is a local fibration, we can lift this arrow into locally cartesian ones: $p'(x)w_i = p(\widehat{x_i})$ with the $w_i$ being $J$-covering. Let's picture the diagram on which we will be proceeding:

\[\begin{tikzcd}
	{p'(d')} & {p(d)} \\
	{p(d_i)}
	\arrow["{p'(x)}", from=1-1, to=1-2]
	\arrow["{w_i}", from=2-1, to=1-1]
	\arrow["{p(\widehat{x_i})}"', from=2-1, to=1-2]
\end{tikzcd}\]

Now, since $p'$ is a comorphism and the $w_i$ are covering $p'(d')$, we can lift them into a covering of $d'$: we have a $K'$-covering $(w'_j : d'_j \to d')_j$ of $d'$, and for each $j$ an arrow $a_j : p'(d'_j) \to p(d_i)$ for some $i$ such that the following diagrams commute:

\[\begin{tikzcd}
	& {p'(d')} & {p(d)} \\
	& {p(d_i)} \\
	{p'(d'_j)}
	\arrow["{p'(x)}", from=1-2, to=1-3]
	\arrow["{w_i}"', from=2-2, to=1-2]
	\arrow["{p(\widehat{x_i})}"', from=2-2, to=1-3]
	\arrow["{p(w'_j)}", from=3-1, to=1-2]
	\arrow["{a_j}"', from=3-1, to=2-2]
\end{tikzcd}\]

Again, since $p'A = p$ we have $p(\widehat{x_i}) = p'(A(\widehat{x_i}))$. But $A$ is a morphism of local fibrations, so it sends locally cartesian arrows to locally cartesian ones. Thus, we can lift the $a_j$ along the locally cartesian arrows $A(\widehat{x_i})$: we have $K'$-coverings $(w'_{jk} : d'_{jk} \to d'_j)_k$ and arrows $a_{jk}$ such that $p'(w'_{jk})a_j = p'(a_{jk})$ and $A(\widehat{x_i})a_{jk} = xw'_kw'_{jk}$. In picture:

\[\begin{tikzcd}
	& {p'(d')} & {p'A(d)} \\
	& {p'A(d_i)} \\
	{p'(d'_j)} \\
	{p'(d'_{jk})}
	\arrow["{p'(x)}", from=1-2, to=1-3]
	\arrow["{w_i}"', from=2-2, to=1-2]
	\arrow["{p'A(\widehat{x_i})}"', from=2-2, to=1-3]
	\arrow["{p(w'_j)}", from=3-1, to=1-2]
	\arrow["{a_j}"{description}, from=3-1, to=2-2]
	\arrow["{p'(a_{jk})}"', from=4-1, to=2-2]
	\arrow["{p'(w'_{jk})}", from=4-1, to=3-1]
\end{tikzcd}\]

Finally, it is lengthy but straightforward to check that the $[a_{jk} : d'_{jk} \to A(d_i), (\widehat{x_i} : d_i \to d,u'w_i : p(d_i) \to c')]$ are well defined elements of  $U^f_{(d,u)}(d'_{ij})$, and that their images under $\nu^f_{(d,u)}(d'_{ij})$ are the $(xw'_jw'_{jk} : d'_{jk} \to A(d), u'p'(w'_jw'_{jk}) : p'(d'_{jk}) \to c')$, which are the localizations of the element that we wanted to locally reach.

Now we want to check the local mono condition. We will use an intermediate step, namely: let  $[\overline{x} : d' \to A(\overline{d}), (g : \overline{d} \to d, \overline{u}' : p(\overline{d}) \to c)]$ with $ up(g) = f \overline{u}'$ be an element of $U^f_{(d,u)}(d')$: locally, this element is equal to elements having their component in $g$ being locally cartesian. Indeed, by \ref{locallyfacthorizontal} we take such a factorization $gw_i = \widehat{g_i}a_i$ with arrows $s_i : p(d_i) \to p(\overline{d})$ such that $p(g)s_i=p(\widehat{g_i})$. Since $A$ preserves coverings, we can pullback the covering $(A(w_i))_i$ along $\overline{x}$ to obtain a $K'$-covering $(w_{ij})_{ij}$ of $d'$:

\[\begin{tikzcd}
	d' & {A(\overline{d})} & {A(d)} \\
	{d'_{ij}} & {A(\overline{d_i})} & {A(d_i)}
	\arrow["{\overline{x}}", from=1-1, to=1-2]
	\arrow["{A(g)}", from=1-2, to=1-3]
	\arrow["{w_{ij}}", from=2-1, to=1-1]
	\arrow["{x_{ij}}"', from=2-1, to=2-2]
	\arrow["{A(w_i)}"', from=2-2, to=1-2]
	\arrow["{A(a_i)}"', from=2-2, to=2-3]
	\arrow["{A(\widehat{g_i})}"', from=2-3, to=1-3]
\end{tikzcd}\]

This gives us in particular the following commutative diagram:

\[\begin{tikzcd}
	& {A(\overline{d})} \\
	{d'_{ij}} & {A(\overline{d_i})} \\
	& {A(d_i)}
	\arrow["{\overline{x}w_{ij}}", from=2-1, to=1-2]
	\arrow["{x_{ij}}"', from=2-1, to=2-2]
	\arrow["{A(a_i)x_{ij}}"', from=2-1, to=3-2]
	\arrow["{A(w_i)}"', from=2-2, to=1-2]
	\arrow["{A(a_i)}", from=2-2, to=3-2]
\end{tikzcd}\]

Moreover, the middle arrow comes with $gw_i$ and $\overline{u}'p(w_i)$, and the bottom arrow comes with $\widehat{g_i}$ and $\overline{u}'s_i$. It is just a matter of checking to see that these give three elements of $U^f_{(d,u)}(d'_{ij})$, which are in fact equals, as shows the connection on the previous diagram. Hence, we have that (locally) every element is equal to elements of $U^f_{(d,u)}$ having locally cartesian arrows for their component in $g$.

So, to check the locally mono condition, we know that we are reduced to verify it with elements having a locally cartesian arrow for their component in $g$: so let $[\overline{x} : d' \to A(\overline{d}), (g:\overline{d} \to d,\overline{u}' : p(\overline{d}) \to c')]$ and $[\overline{x}' : d' \to A(\overline{d}'), (g':\overline{d}' \to d,\overline{u}'' : p(\overline{d}') \to c')]$ be two elements of $U^f_{(d,u)}(d')$, with $g$ and $g'$ being  locally cartesian, such that their images through $\nu^f_{(d,u)}(d') $ are equals, that is the two following squares commute:

\[\begin{tikzcd}
	& {A(\overline{d})} \\
	d' && {A(d)} \\
	& {A(\overline{d}')}
	\arrow["{A(g)}", from=1-2, to=2-3]
	\arrow["{\overline{x}}", from=2-1, to=1-2]
	\arrow["{\overline{x'}}"', from=2-1, to=3-2]
	\arrow["{A(g')}"', from=3-2, to=2-3]
\end{tikzcd}\]

\[\begin{tikzcd}
	& {p'A(\overline{d})=p(\overline{d})} \\
	{p'(d')} && {c'} \\
	& {p'A(\overline{d})=p(\overline{d})}
	\arrow["{\overline{u}'}", from=1-2, to=2-3]
	\arrow["{p'(\overline{x})}", from=2-1, to=1-2]
	\arrow["{p'(\overline{x}')}"', from=2-1, to=3-2]
	\arrow["{\overline{u}''}"', from=3-2, to=2-3]
\end{tikzcd}\]

We can project the first square into the basis, and since it commutes and $p'A(d)=p(d)$, we can lift the arrow thus obtained into locally cartesian arrows:

\[\begin{tikzcd}
	&& {p'A(\overline{d})} \\
	{p(d_i)} & {p'(d')} && {p(d)} \\
	&& {p'A(\overline{d}')}
	\arrow["{p'A(g)}", from=1-3, to=2-4]
	\arrow["{w_i}", from=2-1, to=2-2]
	\arrow["{p(\widehat{k_i})}"{pos=0.6}, bend right = 12, from=2-1, to=2-4]
	\arrow["{p'(\overline{x})}", from=2-2, to=1-3]
	\arrow["{p'(\overline{x'})}"', from=2-2, to=3-3]
	\arrow["{p'A(g')}"', from=3-3, to=2-4]
\end{tikzcd}\]

Since the $w_i$ are $J$-covering and $p'$ is a comorphism, we can lift it to a covering in $\cal D'$:

\[\begin{tikzcd}
	&&& {p'A(\overline{d})} \\
	{p'(d'_j)} & {p(d_i)} & {p'(d')} && {p(d)} \\
	&&& {p'A(\overline{d}')}
	\arrow["{p'A(g)}", from=1-4, to=2-5]
	\arrow["{a_j}", from=2-1, to=2-2]
	\arrow["{p'(w'_j)}"', bend right = 20, from=2-1, to=2-3]
	\arrow["{w_i}", from=2-2, to=2-3]
	\arrow["{p(\widehat{k_i})}"{pos=0.6}, bend right =15, from=2-2, to=2-5]
	\arrow["{p'(\overline{x})}", from=2-3, to=1-4]
	\arrow["{p'(\overline{x'})}"', from=2-3, to=3-4]
	\arrow["{p'A(g')}"', from=3-4, to=2-5]
\end{tikzcd}\]

Since the $\widehat{k_i}$ are locally cartesian and $A$ preserves them, we have that the $A(\widehat{k_i})$ are locally cartesian. But $p'(A(\widehat{k_i}))a_j = p'(A(g)\overline{x}w'_j)$ so we can locally lift the $a_j$ as in the diagram:

\[\begin{tikzcd}
	&&&& {p'A(\overline{d})} \\
	{p'(d'_{jk})} & {p'(d'_j)} & {p(d_i)} & {p'(d')} && {p(d)} \\
	&&&& {p'A(\overline{d}')}
	\arrow["{p'A(g)}", from=1-5, to=2-6]
	\arrow["{p'(w'_{jk})}"', from=2-1, to=2-2]
	\arrow["{p'(a_{jk})}"{description}, bend left = 20, from=2-1, to=2-3]
	\arrow["{a_j}", from=2-2, to=2-3]
	\arrow["{p'(w'_j)}"', bend right = 20, from=2-2, to=2-4]
	\arrow["{w_i}", from=2-3, to=2-4]
	\arrow["{p(\widehat{k_i})}"{pos=0.6}, bend right = 15, from=2-3, to=2-6]
	\arrow["{p'(\overline{x})}", from=2-4, to=1-5]
	\arrow["{p'(\overline{x'})}"', from=2-4, to=3-5]
	\arrow["{p'A(g')}"', from=3-5, to=2-6]
\end{tikzcd}\]

Since we have previously reduced to the case where $g$ and $g'$ were locally cartesians we can also lift the $p'(\overline{x})w_i$ and the $p'(\overline{x}')w_i$ along them, since we respectively have $p(g)p'(\overline{x})w_i = p(\widehat{k_i})$ and $p(g')p'(\overline{x}')w_i = p(\widehat{k_i})$. Keeping in mind that $p=p'A$, this gives:

\[\begin{tikzcd}
	&& {p'A(d_l)} && {p'A(\overline{d})} \\
	{p'(d'_{jk})} & {p'(d'_j)} & {p(d_i)} & {p'(d')} && {p(d)} \\
	&& {p'A(x_{l'})} && {p'A(\overline{d}')}
	\arrow["{p'A(x_l)}", from=1-3, to=1-5]
	\arrow["{p'A(w_l)}", from=1-3, to=2-3]
	\arrow["{p'A(g)}", from=1-5, to=2-6]
	\arrow["{p'(w'_{jk})}"', from=2-1, to=2-2]
	\arrow["{p'(a_{jk})}"{description}, bend left = 20, from=2-1, to=2-3]
	\arrow["{a_j}", from=2-2, to=2-3]
	\arrow["{p'(w'_j)}"{description, pos=0.3}, bend right = 20, from=2-2, to=2-4]
	\arrow["{w_i}", from=2-3, to=2-4]
	\arrow["{p(\widehat{k_i})}"{pos=0.6}, bend right = 10, from=2-3, to=2-6]
	\arrow["{p'(\overline{x})}"{description}, from=2-4, to=1-5]
	\arrow["{p'(\overline{x'})}"{description}, from=2-4, to=3-5]
	\arrow["{p'A(w_ {l'})}"', from=3-3, to=2-3]
	\arrow["{p'A(x'_{l'})}"', from=3-3, to=3-5]
	\arrow["{p'A(g')}"', from=3-5, to=2-6]
\end{tikzcd}\]

There just remains to pull back the coverings $(A(w_l))_l$ and $(A(w'_{l'}))_{l'}$ along the  $(a_{jk})_{jk}$ to (locally) have the connections:

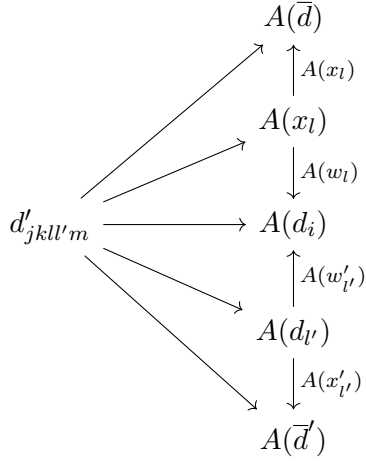
\[\begin{tikzcd}
	&& {A(\overline{d})} \\
	&& {A(x_l)} \\
	{d'_{jkll'm}} && {A(d_i)} \\
	&& {A(d_{l'})} \\
	&& {A(\overline{d}')}
	\arrow["{A(x_l)}"', from=2-3, to=1-3]
	\arrow["{A(w_l)}", from=2-3, to=3-3]
	\arrow[from=3-1, to=1-3]
	\arrow[from=3-1, to=2-3]
	\arrow[from=3-1, to=3-3]
	\arrow[from=3-1, to=4-3]
	\arrow[from=3-1, to=5-3]
	\arrow["{A(w'_{l'})}"', from=4-3, to=3-3]
	\arrow["{A(x'_{l'})}", from=4-3, to=5-3]
\end{tikzcd}\]

A lengthy but straightforward computation gives us that these local connections endowed with the obvious arrows makes  $[\overline{x} : d' \to A(\overline{d}), (g:\overline{d} \to d,\overline{u}' : p(\overline{d}) \to c')]$ and $[\overline{x}' : d' \to A(\overline{d}'), (g':\overline{d}' \to d,\overline{u}'' : p(\overline{d}') \to c')]$ being locally equals elements of  $U^f_{(d,u)}(d')$. This is exactly the mono condition, hence this achieves the proof.

\end{proof}

This specializes into the following theorem which is a relative analogue to the weak Diaconescu's equivalence:

\begin{thm}
Let $p : ({\cal D},K) \to ({\cal C},J)$ be a local fibration and  $f:{\cal F}\to \widehat{\cal C}_J$ be a $\widehat{\cal C}_J$-topos. We have an equivalence
	\[
	\textup{\bf WeakIndGeom}_{\widehat{\cal C}_J}(S_f,S_{C_{p}}) \simeq \textup{\bf ContLocFib}_{({\cal C}, J)}([p], [\pi_f])
	\]
	between the category of indexed weak geometric morphisms from $S_f$ to $S_{C_p}$, where the direction is taken to be the one of the right adjoint, and the category of continuous morphisms of local fibrations over $({\cal C}, J)$ from $[p]$ to $[\pi_f]$.
\end{thm}

\begin{proof}
If we start with an indexed weak geometric morphism $a = (a_* \vdash a^*)$ from $S_f$ to $S_{C_p}$, we precompose $a^*$ with $\eta_{\cal D}$ and this gives us a continuous morphism of local fibrations $a^*\eta_{\cal D} : [p] \to [\pi_f]$. In the other direction, if we have a continuous morphism of local fibrations $A : [p] \to [\pi_f]$, we take its extension $\widetilde{A}^*$ with its right adjoint $\widetilde{A}_*$, which is an \emph{indexed} adjunction because of \ref{morphlocfibcont}. The fact that these two constructions are pseudo-inverses can be deduced in the same way as for the morphism of sites.
\end{proof}

In particular, this indexed weak Diaconescu's theorem simplifies for relative sites:

\begin{thm}\label{weakdiacindex}
Let $p : (\mathcal{G}(\mathbb D),K) \to ({\cal C},J)$ be a relative site and  $f : {\cal F} \to \widehat{\cal C}_J$ be a $\widehat{\cal C}_J$-topos. We have an equivalence
	\[
	\textup{\bf WeakIndGeom}_{\widehat{\cal C}_J}(S_f,S_{C_{p}}) \simeq \textup{\bf ContFib}_{({\cal C}, J)}([p], [\pi_f])
	\]
	between the category of indexed weak geometric morphisms from $S_f$ to $S_{C_p}$, and the category of continuous morphisms of fibrations over $({\cal C}, J)$ from $[p]$ to $[\pi_f]$.
\end{thm}

\begin{proof}
We apply the precedent theorem, recalling that a continuous morphism of local fibrations $A : (\mathcal{G}(\mathbb D),K) \to (({\cal F} / f^*l_J),J_f)$ is exactly a morphism of fibrations. Indeed, since $J_f$-cartesian arrows are the cartesian ones \ref{canonicalcartesian}, the cartesian arrows from $\mathcal{G}(\mathbb D)$ are sent to cartesian ones in $({\cal F} / f^*l_J)$, giving that a morphism of local fibrations is in particular a morphism of fibrations. Conversely, a continuous morphism of fibrations $A : (\mathcal{G}(\mathbb D),K) \to (({\cal F} / f^*l_J),J_f)$ sends locally cartesian arrows to $J_f$-cartesian ones, because of \ref{morphfibimpkfib}.
\end{proof}

The hypothesis of this theorem become more concise when we work with the Giraud topology:

\begin{thm}\label{weakdiacindextrivial}
Let $p : (\mathcal{G}(\mathbb D),J_{\mathbb D}) \to ({\cal C},J)$ be a relative site where the topology on the fibration is the Giraud one, and  $f : {\cal F} \to \widehat{\cal C}_J$ be a $\widehat{\cal C}_J$-topos. We have an equivalence
	\[
	\textup{\bf WeakIndGeom}_{\widehat{\cal C}_J}(S_f,S_{C_{p}}) \simeq \textup{\bf Fib}([p], [\pi_f])
	\]
	between the category of indexed weak geometric morphisms from $S_f$ to $S_{C_p}$, and the category of morphisms of fibrations over $({\cal C}, J)$ from $[p]$ to $[\pi_f]$.
\end{thm}

\begin{proof}
This follows from the fact that, for the Giraud topology, every morphism of fibrations is continuous (Corollary 4.47. \cite{denseness}).
\end{proof}

In the absolute case, for a topos $\cal E$, the following equivalence is well known: 

$$\mathbf{WeakGeom}(\cal E,\mathbf{Set}) \simeq \cal E$$

As already known, the same result holds for indexed weak morphims of toposes. The previous theorem allows us to obtain another proof of this result:

\begin{prop}
Let $f : \cal F \to \cal E$ be a relative topos. We have an equivalence:

$$\mathbf{WeakIndGeom}_{\cal E}(S_f,S_{\cal E}) \simeq \cal F$$
\end{prop}

\begin{proof}
By the previous theorem, we know that $\textup{\bf WeakIndGeom}_{\cal E}(S_f,S_{\cal E}) \simeq \textup{\bf Fib}(\cal E, (\cal F/f^*))$, because we can see the identity functor $\cal E \to \cal E$ as a fibration (associated to the terminal indexed category on $\cal E$): the Giraud topology on it is the canonical one, so that its Giraud topos is the identity geometric morphism on $\cal E$. In this light, morphisms of fibrations $\cal E \to (\cal F/f^*)$ are morphisms of indexed category from the terminal one to $S_f$, that is, global sections of $S_f$; but the category of global sections of $S_f$ is $S_f(\mathbf{1}_{\cal E}) \simeq \cal F$, giving us the wished equivalence.
\end{proof}

As explained in the following remarks, this theorem allows us to draw how different are weak indexed and relative geometric morphisms:

\begin{remarks}
\begin{enumerate}[(i)]
    \item The previous proof is an abstract one; but one can show, more concretly, that the equivalence is given by associating to an indexed weak geometric morphism $a : S_{f} \to S_{\cal E}$ the inverse image of the terminal object of the fiber over the terminal object, that is, $(a^*)_{\mathbf{1}_{\cal E}}(\mathbf{1}_{\cal E})$.
    \item We saw at the beginning of the section that the equivalence between indexed and relative geometric morphisms is given by associating to a relative geometric morphism its $\eta$-extension, and to an indexed geometric morphism its component on the terminal object of the base topos.
    \item Also, we already know from the discussion after \ref{relativeandindexed} that a weak relative geometric morphism is in particular forced to preserve the terminal object.
    \item Because of this, if we have a weak relative geometric morphism $a : [f] \to [1_{\cal E}] $ and we associate to it its $\eta$-extension $\widetilde{a^*} : (\cal F/f^*) \to (\cal E/\cal E)$, the component of its inverse image at the terminal object $(\widetilde{a^*})^*_{\mathbf{1}_{\cal E}}$ will always preserve the terminal object of $\cal E$. Since weak indexed geometric morphisms are precisely characterized by this object, there will be only one (up to equivalence) weak relative geometric morphism giving a weak indexed geometric morphism. This unique weak relative geometric morphism is in fact a relative geometric morphism, namely, the terminal one $f : [f] \to [1_{\cal E}]$.
\end{enumerate}
\end{remarks}

Finally, we present an application of the indexed weak Diaconescu theorem concerning Giraud toposes. One might be tempted to believe that, over a nontrivial site, the appropriate notion of relative site is given by a fibration that is already a stack—that is, a  \textquotedblleft sheaf of categories \textquotedblright rather than a mere \textquotedblleft presheaf of categories\textquotedblright. In this subsection, we show precisely that the Giraud topos of a fibration and that of its stackification are canonically equivalent. 

The following proposition explains why the Giraud toposes of a fibration and of its stackification are equivalent, in light of the universal property they share as a consequence of the indexed weak Diaconescu theorem:

\begin{prop}\label{Giraudtoposofstackif}
Let $p : (\mathcal{G}(\mathbb D),J_{\mathbb D}) \to ({\cal C},J) $ be a fibration and its $J$-stackification $s_J(p) : (s_J(\mathcal{G}(\mathbb D)),J_{s_J(\mathbb D)}) \to ({\cal C},J) $, endowed with their respective Giraud topologies. We have an equivalence of relative toposes $[C_p] \simeq [C_{s_J(p)}]$ induced by the unit of the stackification, as a dense morphism of sites, $\zeta^{\mathbb D}_J : \mathcal{G}(\mathbb D) \to s_J(\mathcal{G}(\mathbb D))$.
\end{prop}

\begin{proof}
With the inclusion-stackification adjunction together with \ref{weakdiacindex}, for a relative topos $f : \cal E \to \widehat{\cal C}_J$ we have the following chain of equivalences:

\[\begin{tikzcd}
	{\mathbf{WeakIndGeom}_{\widehat{\cal C}_J}(S_f,S_{C_p})} & {\mathbf{Fib}_{\cal C}(\mathcal{G}(\mathbb D),({\cal F}/f^*l_J))} \\
	{\mathbf{WeakIndGeom}_{\widehat{\cal C}_J}(S_f,S_{C_{s_J(p)}})} & {\mathbf{Fib}_{\cal C}(s_J(\mathcal{G}(\mathbb D)),({\cal F}/f^*l_J))}
	\arrow["\simeq"{description}, draw=none, from=1-1, to=1-2]
	\arrow["\simeq"{description}, draw=none, from=1-2, to=2-2]
	\arrow["\simeq"{description}, draw=none, from=2-2, to=2-1]
\end{tikzcd}\]

Here, the first and the last equivalences are given by the indexed weak Diaconescu theorem, and the middle one comes from the universal property of the stackification together with the fact that $S_f$ is a stack. Hence, $[C_p]$ and $[C_{s_J(p)}]$ enjoy the same universal property, that is they are equivalent.
\end{proof}

\textbf{Acknowledgements}: Olivia Caramello has benefited for this work of the support of the Université Paris-Saclay in the framework on the Jean D’Alembert 2024 Programme.

\vspace{1cm}

\textsc{Léo Bartoli} 

\vspace{0.2cm}
{\small \textsc{Department of Mathematics, ETH Zurich, Rämistrasse 101
8092 Zurich, Switzerland.}\\
\emph{E-mail address:} \texttt{lbartoli@ethz.ch}

\vspace{0.2cm}

{\small \textsc{Istituto Grothendieck ETS, Corso Statuto 24, 12084 Mondovì, Italy.}\\
	\emph{E-mail address:} \texttt{leo.bartoli@ctta.igrothendieck.org}}

\vspace{0.6cm}

\textsc{Olivia Caramello} 

\vspace{0.2cm}
{\small \textsc{Dipartimento di Scienza e Alta Tecnologia, Universit\`a degli Studi dell'Insubria, via Valleggio 11, 22100 Como, Italy.}\\
	\emph{E-mail address:} \texttt{olivia.caramello@uninsubria.it}}

\vspace{0.2cm}

{\small \textsc{Istituto Grothendieck ETS, Corso Statuto 24, 12084 Mondovì, Italy.}\\
	\emph{E-mail address:} \texttt{olivia.caramello@igrothendieck.org}}

\vspace{0.2cm}

\small \textsc{Université Paris-Saclay, CentraleSupélec, Mathématiques et Informatique pour la Complexité et les Systèmes, 91190, Gif-sur-Yvette, France.}\\
	\emph{E-mail address:} \texttt{olivia.caramello@centralesupelec.fr}}

\end{document}